\documentclass[a4paper, 11pt, reqno]{article}

\usepackage[left=2.65cm,right=2.65cm,top=2.7cm,bottom=2.7cm]{geometry}

\usepackage{amsmath,amsfonts,amssymb,amsthm}
\usepackage[dvipsnames]{xcolor}
\usepackage{hyperref}
\hypersetup{
    colorlinks,
    linkcolor={red!50!black},
    citecolor={blue!50!black},
    urlcolor={blue!80!black}
}
\usepackage{combelow}
\usepackage{tikz-cd}
\usepackage{cleveref}

\DeclareMathOperator{\Homeo}{Homeo}
\DeclareMathOperator{\LocHomeo}{LocHomeo}

\DeclareMathOperator{\Id}{Id}
\DeclareMathOperator{\Ad}{Ad}
\DeclareMathOperator{\Cl}{Cl}

\newcommand{\R}{\mathbb{R}}
\newcommand{\bN}{\mathbb{N}}
\newcommand{\cG}{\mathcal{G}}
\newcommand{\cF}{\mathcal{F}}
\newcommand{\cK}{\mathcal{K}}
\newcommand{\fX}{\mathfrak{X}}
\renewcommand{\d}{\mathrm{d}}

\newcommand{\pp}[2]{\frac{\partial#1}{\partial#2}}
\newcommand{\on}[1]{ \operatorname{#1}}

\theoremstyle{plain}
\newtheorem{theoremintro}{Theorem}
\newtheorem{questionintro}[theoremintro]{Question}
\newtheorem*{questionintro*}{Question}

\newtheorem{theorem}{Theorem}[section]
\newtheorem{proposition}[theorem]{Proposition}
\newtheorem{corollary}[theorem]{Corollary}
\newtheorem{lemma}[theorem]{Lemma}

\newtheorem*{propertyP}{Property P(r)}

\theoremstyle{definition}
\newtheorem{definition}[theorem]{Definition}
\newtheorem{remark}[theorem]{Remark}
\newtheorem{caveat}[theorem]{Caveat}
\newtheorem{example}[theorem]{Example}

\newtheorem{claim}[theorem]{Claim}
\newtheorem{question}[theorem]{Question}

\newtheorem{questionother}{Question}

\definecolor{myorange}{RGB}{199, 95, 95}
\definecolor{mymagenta}{RGB}{139, 39, 95}
\definecolor{myblue}{RGB}{52,105,165}
\definecolor{mydarkblue}{RGB}{40,18,111}

\author{Robert Cardona and Fabio Gironella}

\newcommand{\Addresses}{{
  \bigskip
  \footnotesize

  R.~Cardona, \textsc{Departament de Matem\`atiques i Inform\`atica, Universitat de Barcelona, Gran Via de Les Corts Catalanes 585, 08007 Barcelona, Spain; Centre de Recerca Matemàtica, Campus de Bellaterra, Edifici C, 08193, Barcelona, Spain.}\par\nopagebreak
  \textit{E-mail address}: \texttt{robert.cardona@ub.edu}

  \medskip

  F.~Gironella, \textsc{CNRS - Laboratoire de Mathématiques Jean Leray, Nantes Université, 2 Chem. de la Houssinière, 44322 Nantes, France.}\par\nopagebreak
  \textit{E-mail address}: \texttt{fabio.gironella@cnrs.fr}

}}

\date{}

\title{$C^0$-Poisson geometry, coisotropic submanifolds, \\ and clean intersection points}

\begin{document}

\maketitle

%%%%%%%%%%%%%%%%%%%%%%%%%%%%%%%%%%%%%%%%%%%%%%%%%%%%%%%%%%%%%%%%%%%%%%%%%%%%%%%%%%%%%%%%%%%%%%%%%%%%%%%%%%%%%%%%%
%%%%%%%%%%%%%%%%%%%%%%%%%%%%%%%%%%%%%%%%%%%%%%%%%%%%%%%%%%%%%%%%%%%%%%%%%%%%%%%%%%%%%%%%%%%%%%%%%%%%%%%%%%%%%%%%%

\begin{abstract}
In this work, we initiate the study of rigidity and non-rigidity phenomena for
Poisson homeomorphisms, defined as uniform $C^0$-limits of Poisson
diffeomorphisms. First, we prove that Poisson homeomorphisms preserve the singular
symplectic foliation: they map symplectic leaves to symplectic leaves by
symplectic homeomorphisms. Second, we establish the $C^0$-rigidity of coisotropic
submanifolds in Poisson manifolds. A key ingredient is the notion of ``clean
intersection point” between a submanifold and the leaves of a singular foliation,
whose study is of independent interest for singular foliation theory and Poisson
geometry. In contrast with the symplectic case, characteristic foliations of
coisotropic submanifolds are not rigid under Poisson homeomorphisms, exhibiting
flexibility phenomena specific to the Poisson setting. We discuss partial rigidity
results, introduce a topological invariant of coisotropic submanifolds, the $C^0$-characteristic partition, and show that $C^0$-coisotropic submanifolds are non-smooth objects whose vanishing ideal defines a Lie subalgebra of the Poisson algebra. Finally, we consider Poisson homeomorphisms that lift to symplectic homeomorphisms of a symplectic realization and show that nearly all Poisson manifolds admit non-liftable Poisson homeomorphisms. Our main results answer three questions posed by Joksimovi\'c and M\u{a}rcu\c{t}.
\end{abstract}

\setcounter{tocdepth}{2}

{
  \hypersetup{linkcolor=black}
  \tableofcontents
}

\section{Introduction}

\paragraph{Symplectic homeomorphisms.}  Symplectic geometry has seen an outburst of interest in recent decades as it is characterized by a subtle interaction between \emph{flexible} (i.e., topological), and \emph{rigid} (i.e., geometric) phenomena. Flexibility usually takes the form of h-principles, i.e.\ of results that roughly speaking state that if a certain existence and/or classification problem can be solved ``formally'' (i.e.\ in the set of algebraic topological objects associated to the geometric object in question), then it can also be solved geometrically. Rigidity is probed through the use of ``hard'' techniques such as pseudo-holomorphic curves à la Gromov \cite{gromov1985pseudo}. Among the first rigidity results in symplectic geometry there was exactly Gromov's non-squeezing theorem \cite{gromov1985pseudo}, which implies that symplectomorphisms (diffeomorphisms preserving the symplectic structure) are more rigid, in a rather subtle way, than just volume-preserving diffeomorphisms. 
\medskip

Another emblematic rigidity result for symplectomorphisms is the so-called Eliashberg\,-Gromov's $C^0$-rigidity theorem, stating that the space of symplectomorphisms is $C^0$-closed in the space of diffeomorphisms. 
This is rather surprising as being a symplectomorphism is a condition that involves first-order derivatives of the map, over which there is no control via the $C^0$-norm of the map. This gave rise to the notion of \emph{(local) symplectic homeomorphism}: a homeomorphism that is (locally) the uniform $C^0$-limit of symplectomorphisms. 

\medskip

There is by now a vast literature on the study of symplectic homeomorphisms, which shows that the latter are in certain aspects more flexible than their smooth counterparts, while at the same time preserving some of their rigid features. For instance, on the flexible side, Buhovsky and Opshtein exhibited in \cite{BO2} an example of a symplectic homeomorphism mapping a symplectic disc to an isotropic one. Moreover, Buhovsky, Humilière, and Seyfaddini provided in  \cite{BHS_Arnold} an example of a Hamiltonian homeomorphism violating the lower bound on the number of $1$-periodic orbits prescribed by the Arnold conjecture.
On the other hand, in the direction of rigidity, the same authors proved in \cite{BHSarnoldtype} that there is a meaningful, even if weaker, topological lower bound in the style of the Arnold conjecture that survives the passage to the $C^0$ realm in some situations.

More relevant for our purposes is the study of $C^0$-rigidity results for Lagrangian submanifolds and more generally for coisotropics, which has quite a long history in the literature, starting from \cite{laudenbach_sikorav:C0_limits_lagrangians}, going through \cite{opshtein2009mathcal}, 
and culminating in the work \cite{HLS15}. 
In the latter, Humilière, Leclercq, and Seyfaddini prove the following striking and very general result: if a local symplectic homeomorphism $\varphi$ maps a coisotropic submanifold $C$ to a smooth submanifold $\varphi(C)$, then $\varphi(C)$ is automatically coisotropic, and moreover $\varphi$ maps each leaf of the characteristic foliation of $C$ to a leaf of the characteristic foliation of $\varphi(C)$.
Note that, after \cite{HLS15}, an independent proof of these two results has also been given by Usher in \cite{Usher_CoisotropicC0Rigidity}.
The boundary between flexible and rigid behaviors for symplectic homeomorphisms remains a very subtle and intriguing object of study, as shown in the context of squeezing/non-squeezing of symplectic submanifolds via symplectic homeomorphisms in \cite{BO}.

\medskip

\paragraph{Poisson topology through a symplectic lens.}
\emph{Poisson structures} are a wide natural generalization of symplectic structures. 
These are bivector fields $\Pi\in \Gamma(\Lambda^2TM)$ on the ambient manifold $M$ that are integrable in the sense that $[\Pi,\Pi]=0$, where $[\cdot,\cdot]$ is the Schouten-Nijenhuis bracket on the space of multivector fields.
While these structures are certainly very interesting from the geometric and dynamical points of view, for the purposes of this paper one can think of them topologically as the data of a \emph{singular foliation} equipped with a \emph{leafwise symplectic structure}.
Here, singular foliation is intended in the sense of Stefan \cite{St} and Sussmann \cite{Su}, i.e.\ as an involutive and locally finitely generated $C^\infty(M)$-submodule of the $C^\infty(M)$-module of all smooth vector fields on $M$.
As for regular foliations, any such singular foliation induces a partition of $M$ by (immersed) submanifolds, but in this singular setting, these can be of different dimensions. We refer to \cite{LGLR_notes} for a nice introduction to singular foliations.
In the Poisson case, the singular foliation $\cF$ is given by the image of the map $\Pi^\sharp\colon \Omega^1(M)\to \fX(M)$ given by $\alpha\mapsto \iota_\alpha\Pi$. By definition $\Pi^\sharp$ is invertible over its image, and the inverse $\omega^\flat:=(\Pi^\sharp)^{-1}\vert_{\cF}\colon \cF\to \Omega^1(M)$ induces the leafwise symplectic structure on $\cF$.

There has been quite some interest recently in the study of \emph{topological} properties of Poisson structures through the use of symplectic techniques.
While there are many examples in the literature of adaptations of classical arguments, such as Moser's normal forms applied to special classes of Poisson structures, we are especially interested here in more topological and global decomposition/structural results.

In the general setting, this is quite a hard task, as one lacks local normal forms, and the global structure of the leaves of the singular symplectic foliation associated to the Poisson structure (which are, in general, open without controlled behavior at infinity) is wild. 

\smallskip
A first approach to tackle global topological properties is to restrict to very special classes of Poisson manifolds with very restricted singularities. A prototypical testground for symplectic techniques is that of \emph{b-Poisson manifolds} \cite{bPois}: a Poisson manifold of dimension $2n$ such that $\Pi^n$ vanishes transversely along a hypersurface. In particular, the Poisson structure is a symplectic structure away from a hypersurface. Works in the rigidity direction include understanding which manifolds admit such structures \cite{Cav_examples}, attempts to define Floer homologies \cite{BMO_bFloer, K_floer}, or rigidity results about the ambient topology \cite{Alboresi_LogSymplectic}. Other restricted classes of Poisson manifolds have been considered, like \emph{scattering symplectic structures} \cite{Lanius_scattering,Alboresi_ScatteringFillings}. On the flexible side, certain restrictive $h$-principles have been established for regular Poisson open manifolds \cite{Bert_flexibility,FF_flexibility} and $b$-Poisson open manifolds \cite{FMM_flexibility}.

\smallskip

The second approach is to try to leverage symplectic results that hold for \emph{all} symplectic manifolds, closed or not. 
For instance, leveraging the general energy-capacity inequality \cite{LalMcD}, Joksimović and M\u{a}rcu\c{t} prove in \cite{JoksimovicMarcut_HoferNormPoisson} that the natural extension of the symplectic Hofer norm to Hamiltonians on a general Poisson manifold is indeed a norm.
More recently, relying exactly on \cite{JoksimovicMarcut_HoferNormPoisson},  Joksimović proved in \cite{Jo} the analogue of Eliashberg-Gromov's $C^0$-rigidity result for Poisson diffeomorphisms, leading to the notion of \emph{Poisson homeomorphism}, analogously to the symplectic case.  

\medskip

\paragraph{Scope of this work.} This paper aims to inaugurate the study of Poisson homeomorphisms by trying to understand their rigidity and lack thereof. Contrary to almost all previously mentioned results in the direction of a ``Poisson topology", we obtain rigidity (and some flexibility) results for all (or most, for some results) Poisson manifolds. The bulk of the work will consist in answering the three questions of Joksimovi\'c and M\u{a}rcu\c{t} reported in \cite{Jo}, which we will recall. Along the way, and key to some of the proofs, is the introduction and leverage of the notion of ``clean intersection point" (see Definition \ref{def:cleanpointfol}), which is of independent interest and relevant already in the context of singularly foliated spaces and of smooth Poisson geometry. As the paper involves both a new object in the literature (Poisson homeomorphisms) and a new notion in the setting of singular foliations (clean intersection points), we will also give several examples illustrating their properties.

Also worth of notice is the fact that, all the ingredients considered, one could argue that analogously to \cite{JoksimovicMarcut_HoferNormPoisson}, the deep reason why the new rigid phenomena we present in this paper for Poisson homeomorphisms hold is actually that the symplectic topological rigidity we rely on is of a semi-local nature, so that the fact that singular leaves are non-compact is ultimately not an issue, and the difficulties are pushed to understanding the behavior of different objects within local charts of singular foliations. 
Indeed, we focus on the case of ``local Poisson homeomorphisms'', homeomorphisms that are locally limits of Poisson diffeomorphisms, and one of the main symplectic results which we use is the main result in \cite{HLS15}, which holds indeed for local symplectic homeomorphisms too.
\medskip

\noindent \emph{Conventions:} Unless otherwise specified, we work with possibly open ambient manifolds. By submanifold, we mean an embedded submanifold.

\paragraph*{I. Restriction of Poisson homeomorphisms to the symplectic leaves.}

It is well known that a Poisson diffeomorphism $\varphi\colon (M,\Pi)\to (M,\Pi)$ behaves nicely with respect to the singular symplectic foliation $(\cF,\omega)$ associated to the Poisson structure: namely, for every symplectic leaf $L$ of $(\cF,\omega)$, $f(L)$ is a (possibly different) symplectic leaf, and $f\vert_L\colon (L,\omega_L)\to (f(L),\omega_{f(L)})$ is a symplectomorphism.

Now, already the singular foliation $\cF$ associated to $\Pi$ is an invariant of a smooth nature, as its definition involves submodules of vector fields.
It is hence natural to wonder whether Poisson homeomorphisms also behave nicely with respect to $\cF$, as well as with respect to the leafwise symplectic structure $\omega$.
This is indeed a question of M\u{a}rcu\c{t}, see \cite[Question 1 and Acknowledgments]{Jo}:

\begin{questionother}[M\u{a}rcu\c{t}]\label{q1}
    Let $\varphi:M\longrightarrow M$ be a Poisson homeomorphisms of $(M,\Pi)$. Does $\varphi$ map symplectic leaves to symplectic leaves homeomorphically? If so, does it map them by symplectic homeomorphisms?
\end{questionother}

We prove that the answer to the above question is positive for local symplectic homeomorphisms:

\begin{theorem}\label{thm:main1}
    Let $(M,\Pi)$ be a Poisson manifold. 
    Any local Poisson homeomorphism maps the symplectic leaves to symplectic leaves by local symplectic homeomorphisms.
\end{theorem}
As it will become clear, our proofs only allow us to conclude that in general, even if the Poisson homeomorphism is global, the map induced on each leaf is a priori only a local symplectic homeomorphism. We point out, however, that it is not known whether there exist local symplectic homeomorphism which are not (global) symplectic homeomorphisms. A first step to prove \Cref{thm:main1} establishes a topological version of it, which is of independent interest: homeomorphisms that are limits of foliation-preserving diffeomorphism map leaves to leaves homeomorphically (\Cref{thm:foliatedhomeo}). Once this is established, we deduce \Cref{thm:main1} via a local argument in Weinstein splitting charts for the Poisson structure. One can also show, and the argument here is much simpler, that for the restricted subclass of Hameomorphisms (defined analogously to the symplectic setting), every symplectic leaf is mapped to itself by a symplectic homeomorphism (\Cref{cor:Hameos_preserve_leaves}).

\paragraph{II. Poisson homeomorphisms and coisotropic submanifolds} As we mentioned, one of the pinacles of $C^0$-symplectic topology is the rigidity of coisotropic submanifolds \cite{HLS15}. These submanifolds are also of crucial relevance in Poisson geometry \cite{Weinstein_CoisotropicCalculus}, hence the following natural question.

\begin{questionother}[Joksimovi\'c]\label{q2}
    Let $C$ be a coisotropic submanifold of $(M,\Pi)$, and let $\varphi:M\longrightarrow M$ be a Poisson homeomorphism such that $C'=\varphi(C)$ is a smooth submanifold. Is $C'$ coisotropic?
\end{questionother}

Our second main theorem establishes that this is indeed the case, $C^0$-rigidity of coisotropics holds in the wide context of Poisson manifolds.

\begin{theorem}\label{thm:main2}
    Let $(M,\Pi)$ be a Poisson manifold and $\varphi:M\longrightarrow M$ be a Poisson homeomorphism. Suppose that $C\subset M$ is a coisotropic submanifold and that $C'=\varphi(C)$ is a smooth submanifold. Then $C'$ is coisotropic.
\end{theorem}
There are two key ingredients in the proof of \Cref{thm:main2}.
First, is the symplectic analogue of the result, proven in \cite{HLS15}. The second one tries to leverage the fact that if a coisotropic submanifold $C$ of a Poisson manifold $(M,\Pi)$ ``intersects cleanly'' every leaf of the singular foliation of $\Pi$, then $C\cap L$ is coisotropic (in the symplectic sense) in any symplectic leaf $L$. Of course, the assumption that $C$ intersects cleanly every leaf is far from true in general, and is a rather strong assumption. Motivated by this fact, we introduce a local notion of clean intersection, that of ``clean intersection point" (\Cref{def:cleanpointfol}), which essentially requires that the intersection between the submanifold and the leaf through that point is locally clean. Surprisingly, this notion has never been considered or studied before. The proof then revolves around understanding this notion and its behavior within coisotropic submanifolds. For instance, at these points, being coisotropic is completely characterized by the local intersection with the symplectic leaf (\Cref{prop:cleancoisotropicinleaf}). One of our main technical results is the proof that most points of a submanifold in a singularly foliated space are clean intersection points.

\begin{theorem}\label{thm:clean_intro}
    Let $\mathcal{F}$ be a singular foliation on $M$, and $N$ an immersed submanifold. Then the set of clean intersection points of $N$ with the leaves of $\cF$ contains an open and dense subset of $N$.
\end{theorem}

\Cref{thm:clean_intro} is of independent interest. Indeed, the notion of clean intersections is a natural generalization of transverse intersection that naturally appears in symplectic and Poisson geometry (e.g.\ in the contexts of Lagrangian Floer homology \cite{Pozniak_FloerClean} and of Poisson relations \cite{Weinstein_CoisotropicCalculus}). From this theorem and the properties of clean intersection points for coisotropics, one deduces \Cref{thm:main2}.
\medskip

In the Poisson context, coisotropic submanifolds also have an induced characteristic foliation, which is singular instead of regular (as it would be the case in the symplectic context). The main result in \cite{HLS15} also shows that characteristic leaves are rigid: any symplectic homeomorphism mapping a coisotropic to a smooth (hence, coisotropic by the rigidity result) submanifold must map characteristic leaves to characteristic leaves homeomorphically. Hence, one can complement \Cref{q2} with the following: in the situation of \Cref{thm:main2}, are the leaves of the characteristic foliation of $C$ mapped homeomorphically to leaves of the characteristic foliation of $C'$?

As we will see with explicit examples, this is in general \emph{not} the case. In particular, this exhibits some new flexibility phenomena. As it follows from the proof of \Cref{thm:C0-rigidity_coisotropics} (whose statement includes \Cref{thm:main2}, with some more information about the behavior of characteristic leaves), the answer to the question above (Question \ref{q:dense}) is positive ``locally" on an open and dense subset of $C$. However, a legitimate question for which we do not have a proof or a counterexample is the following.

\begin{questionintro}\label{q:dense}
Is there always a dense set of characteristic leaves of $C$ that are mapped homeomorphically to characteristic leaves of $C'$? If $C$ has a compact closure, is there always a dense and open set of characteristic leaves mapped homeomorphically?
\end{questionintro}

We provide sufficient conditions for a positive answer to this question in Section \ref{sec:when_almost_all_char_leaves_mapped_homeo}, showing that the answer seems to depend again on the properties of clean intersection points, and the related notions of clean characteristic leaf (\Cref{def:cleancharleaf}) and non-clean path (\Cref{def:noncleanpath}). In particular, the answer is positive if one can answer \Cref{q:cleanpaths} in the negative, a question that only involves a property of the set of clean points of coisotropic submanifolds (without involving any Poisson homeomorphism).
\medskip

\paragraph{III. $C^0$-Hamiltonians, rigid partitions and $C^0$-coisotropics.} The previously introduced tools allow us to study a generalization to the Poisson setting of \cite[Theorem 3]{HLS15}.
Recall that the latter states, in the symplectic setup, that ``$C^0$-Hamiltonians'' restrict to a connected coisotropic $C$ as a function of time only if and only if the continuous isotopy they generate preserves $C$ and flows along the characteristic foliation.
Roughly speaking, a $C^0$-Hamiltonian $H$ is a continuous (time-dependent) real-valued function that is $C^0$-limit of smooth Hamiltonians whose generated (smooth) isotopies also $C^0$-converge to a family of homeomorphisms $\phi^t$, that is called Hameotopy.
These notions can be extended naturally to the Poisson setup (see \Cref{def:hameotopy_C0-hamiltonian}), and generalizing (a part of) \cite[Theorem 3]{HLS15} involves understanding more in detail the behavior of non-clean paths in coisotropic submanifolds. Our generalization of \cite[Theorem 3]{HLS15}, Theorem \ref{thm:C0_Hamiltonians_vanishing_C}, gives then a characterization of those $C^0$-Hamiltonians whose Hameotopy preserves $C$. 
It is actually a strengthening of \Cref{thm:main2} analogously to the symplectic case, as one can see that it implies it (see \Cref{rem:Ham_vanish_implies_rigidity}).

We also recall that in the symplectic setup the identity Hameotopy can only be generated by $C^0$-Hamiltonians that solely depend on time (see \cite{Viterbo_UniquenessGeneratingLowRegHamiltonians,BuhovskySeyfaddini_UniquenessGeneratingC0Hamiltonians}); it directly follows as a consequence that, in the Poisson setup, the identity Hameotopy can only be generated by a time-dependent Casimir function (i.e. a 
function that is constant, depending on time, on each singular symplectic leaves). An immediate corollary of it (see \Cref{cor:C0-Hamiltonians_up_to_Casimirs}) is the extension to the Poisson setup of \cite[Corollary 4]{HLS15}, namely the fact that, on a Poisson manifold, a $C^0$-Hamiltonian $H$ is a continuous time-dependent Casimir function if and only if the Hameotopy $\phi^t$ that it generates is the identity.

All this naturally leads us to introduce the notion of \textbf{$C^0$-characteristic partition} of a coisotropic in a Poisson manifold (Definition \ref{def:C0-characteristic_foliation}): this is a new topological invariant of coisotropic submanifolds, given by a certain partition of $C$ that can be shown to be invariant under Poisson homeomorphisms, and has other desirable rigidity properties. We do not tackle this notion in this work, but ask whether the $C^0$-characteristic partition is in fact a partition \emph{by topological submanifolds}; see \Cref{quest:C0-characteristic_leaves_topological_submanifolds}. Lastly, our rigidity Theorem \ref{thm:main2} naturally gives rise to the definition of $C^0$-coisotropic submanifolds, that is, topological submanifolds that are locally Poisson homeomorphic to a smooth coisotropic. 
We deduce the following surprising consequence of \Cref{thm:C0_Hamiltonians_vanishing_C}, pointed out to us by Du\v{s}an Joksimovi\'c, and showing that these non-smooth objects define Lie subalgebras of the Poisson algebra induced by $\Pi$ on $C^\infty(M)$.
\begin{corollary}\label{cor:van_ideal}
    Let $C$ be a $C^0$-coisotropic submanifold of a Poisson manifold $(M,\Pi)$. Then the vanishing ideal 
    $$\mathcal{I}(C)=\{f\in C^\infty(M)\mid f|_C=0\}$$
    is a Lie subalgebra of $(C^\infty(M), \{\cdot,\cdot\})$, where $\{ \cdot,\cdot\}$ is the Poisson bracket associated to the Poisson bivector $\Pi$.
\end{corollary}
To the author's knowledge, even in the symplectic setting, this consequence of the symplectic version of \Cref{thm:C0_Hamiltonians_vanishing_C} (namely, \cite[Theorem 3]{HLS15}) was not explicitly stated in the literature\footnote{Such an observation was known in the symplectic case to, at least, Du\v{s}an Joksimovi\'c and Michael Usher, as we have been informed by the former.}, but it follows from the same computations as in the proof of \Cref{cor:van_ideal}, relying in this case on \cite[Theorem 3]{HLS15}.

\paragraph*{IV. Lifting Poisson homeomorphisms to symplectic integrations.}

As proved in \cite{Weinstein_LocalStructure,coste_dazord_weinstein:groupoides_symplectiques} (c.f.\ also \cite{CrainicMarcut_ContravariantSprays}), Poisson manifolds can be seen as quotient space of a fibration from a symplectic manifold (that will in general be open); such fibrations are called \textit{symplectic realizations}.
When the covering symplectic manifold is a (local) symplectic Lie groupoid and the projection map is the target map (or the source one, depending on conventions), the symplectic realization is called a \emph{symplectic integration}.
Among all symplectic integrations, there is a ``nicer'' one, canonical in the sense that it is the unique one (up to isomorphism) with simply-connected target-fibers\footnote{We use the convention in \cite{CFM_book} of using the target map as projection map onto the base for symplectic integrations.}. Unfortunately (or fortunately), this canonical integration doesn't always exist, and the Poisson manifold $(M,\Pi)$ is called \emph{integrable} if it does;
the works \cite{CrainicFernandes_IntegrabilityLieBrackets,CrainicFernandes_IntegrabilityPoissonBrackets} of Crainic and Fernandes completely characterize when this happens, and give an explicit construction of (a representative in the isomorphism class of) this canonical symplectic integration, namely the ``Poisson homotopy groupoid'' $\Sigma(M)$.

In this context, M\u{a}rcu\c{t} then formulated the following natural question (see \cite[Question 3 and Acknowledgments]{Jo}, and c.f.\ \Cref{cav:sympl_realizations_not_hausdorff}):

\begin{questionother}[M\u{a}rcu\c{t}]\label{q4}
    Let $\varphi:M\longrightarrow M$ be a Poisson homeomorphism of an integrable Poisson manifold $M$, and $\Sigma(M)$ the canonical symplectic integration of $(M,\Pi)$. Is there a symplectic homeomorphism $\tilde \varphi$ of $\Sigma(M)$ lifting $\varphi$, i.e., satisfying $t\circ\varphi=\tilde \varphi \circ t$ where $t:\Sigma(M)\longrightarrow M$ is the target map?
\end{questionother}

More generally, one can call a Poisson homeomorphism of an arbitrary Poisson manifold ``liftable" if there is a symplectic homeomorphism lift for some symplectic realization of $M$. As we will see, whenever a Poisson homeomorphism is liftable in this more general sense, it inherits additional coisotropic rigidity from its symplectic lift (\Cref{prop:full_coisotropic_rigidity_liftable}). Exploiting the flexibility of characteristic foliations of coisotropics for general Poisson homeomorphisms, we show that almost every Poisson manifold admits non-liftable Poisson homeomorphisms, exhibiting some flexibility phenomena in $C^0$-Poisson geometry. 
For the following statement, we say that $\Pi$ is \emph{almost everywhere non-degenerate} if $M$ is of even dimension, say $2n$, and $\Pi^{n}\neq 0$ on a dense subset of $M$.

\begin{theorem}\label{thm:main3}
Let $(M,\Pi)$ be a Poisson manifold such that $\Pi$ is not almost everywhere non-degenerate. Then there exists a non-liftable Poisson homeomorphism of $(M,\Pi)$.
\end{theorem}
One might wonder if the result holds for the class of almost every non-degenerate Poisson manifold, which are as close as possible to the symplectic world. We believe that this is probably the case, at least in dimension four or higher, but it seems that the example has to be constructed ``ad-hoc" according to the particular singularities of the Poisson structure. We illustrate this by showing in Section \ref{sec:examplebPoisson} that, for example, non-liftable examples (obtained from examples where characteristic leaves are not rigid) exist on the prototypical class of $b$-Poisson manifolds.\\

\paragraph{Organisation of the paper.} In Section \ref{sec:prel}, we cover some background material on singular foliations and Poisson geometry. In Section \ref{sec:leafwisebehavior}, we analyse the leaf-wise behavior of Poisson homeomorphisms, proving Theorem \ref{thm:main1}. We then proceed in Section \ref{sec:cleanpoints} to introduce the notion of ``clean intersection point", establish its properties within the context of coisotropic submanifolds, and prove Theorem \ref{thm:clean_intro}. In Section \ref{sec:rig_nonrig_coisotropics}, we establish our coisotropic rigidity Theorem \ref{thm:main1}, exhibit examples showing that the characteristic foliation is not fully rigid in general, and give partial answers to \Cref{q:dense}. In Section \ref{sec:C0coisotropics}, we introduce the notion of $C^0$-characteristic partition and that of $C^0$-coisotropic submanifold, and characterize Hameotopies preserving a coisotropic submanifold and prove Corollary \ref{cor:van_ideal}. In the last Section \ref{sec:non-liftable_homeos}, we recall the notion of liftable Poisson homeomorphisms, give examples, establish their rigidity of characteristic leaves, and prove Theorem \ref{thm:main3}.\\

\paragraph{Acknowledgements.} The authors are grateful to Du\v{s}an Joksimovi\'c for useful comments on the first version of this work, and for suggesting Corollary \ref{cor:van_ideal}.  RC acknowledges partial support from the AEI grant PID2023-147585NA-I00 and the Spanish State Research Agency, through the Severo Ochoa and María de Maeztu Program for Centers and Units of Excellence in R\&D (CEX2020-001084-M).
FG benefits from the support of the ANR grant SiSyFo (ANR-25-ERCS-0010) and, during part of the development of this work, also benefited from the support of the project Étoile Montante 2023 SymFol funded by the région Pays de la Loire.

%%%%%%%%%%%%%%%%%%%%%%%%%%%%%%%%%%%%%%%%%%%%%%%%%%%%%%%%%%%%%%%%%%%%%%%%%%%%%%%%%%%%%%%%%%%%%%%%%%%%%%%%%%%%%%%%%
%%%%%%%%%%%%%%%%%%%%%%%%%%%%%%%%%%%%%%%%%%%%%%%%%%%%%%%%%%%%%%%%%%%%%%%%%%%%%%%%%%%%%%%%%%%%%%%%%%%%%%%%%%%%%%%%%

\section{Preliminaries}\label{sec:prel}
In this section, we recall some basic facts about singular foliations and Poisson manifolds.

%%%%%%%%%%%%%%%%%%%%%%%%%%%%%%%%%%%%%%%%%%%%%%%%%%%%%%%%%%%%%%%%%%%%%%%%%%%%%%%%%%%%%%%%%%%%%%%%%%%%%%%%%%%%%%%%%

\subsection{Singular foliations}

We first recall the notion of singular foliation, as defined in \cite{AS} inspired by the Stefan-Sussmann integrability conditions \cite{Su, St}. We refer to \cite{LGLR_notes} for an extensive discussion of this notion and its properties.

\begin{definition}[\cite{AS}]
\label{def:singfol}
    A \textbf{singular foliation} on $M$ is a subspace $\mathcal{F}\subset \mathfrak{X}(M)$ which is involutive, stable under multiplication by elements in $C^\infty(M)$, and locally finitely generated.
\end{definition}
We will say that a diffeomorphism
$$\phi: M \longrightarrow M$$
is \textbf{foliation-preserving} if $d\phi(\mathcal{F})=\mathcal{F}$. These are also called symmetries of the foliation, see \cite[Section 1.3.1]{LGLR_notes}.
In the same spirit of how we defined the symplectic leaves on a Poisson manifold, we define the leaf $L_p$ of $\mathcal{F}$ along a point $p$.

\begin{definition}
\label{def:leaf}
Let $p$ be a point in a manifold $M$ with a singular foliation $\cF$. 
The \textbf{leaf} $L_p$ of $\cF$ containing $p$ is the set of points $q\in M$ for which there is a sequence of points $p_1=p,p_2,...p_{k-1},p_k=q$ and vector fields $X_1,...,X_{k-1} \in \mathcal{F}$ such that the integral curve of $X_i$ starting at $p_i$ reaches $p_{i+1}$ at time $t=1$.
\end{definition} 

As detailed in e.g.\ \cite[Theorem 1.7.9]{LGLR_notes}, such leaves are smooth manifolds partitioning $M$, and their tangent space is given at any of its points by $\mathcal{F}$. 
Note also that a foliation-preserving diffeomorphism maps leaves of $\mathcal{F}$ to leaves of $\mathcal{F}$ diffeomorphically.

\medskip

The \textbf{dimension of the leaf} through a point $p\in M$, which denote by $\dim(T_p\mathcal{F})$ must not be confused with the \textbf{rank} of $\mathcal{F}$ at $p$, which is denoted by $\operatorname{rk}_p(\mathcal{F})$ and is defined as the minimal number of generators of $\mathcal{F}$ in a sufficiently small neighborhood of that point. 
One has that the latter is always greater than or equal to the former (see e.g.\ \cite[Lemma 1.3.18]{LGLR_notes}). An important property of the dimension of the leaves that we will need is the following (see e.g.\ \cite[Lemma 1.3.19]{LGLR_notes}).
\begin{lemma}\label{lem:rankcont}
    Let $\mathcal{F}$ be a singular foliation in $M$. The map
    \begin{align*}
        \dim_{T\mathcal{F}}: M &\longrightarrow \mathbb{N}\\
        p &\longmapsto \dim_p(T\mathcal{F})
    \end{align*}
    is lower semi-continuous.
\end{lemma}
We also recall explicitly a useful consequence of this lemma and of the fact that $\dim_{T\cF}$ is integer valued, see \cite[Proposition 1.3.23]{LGLR_notes}.
For this, let
$$M_{\operatorname{reg}}=\{p\in M \mid \exists \text{ a neighborhood } U \text{ of } p \text{ such that } \dim_{T\cF}|_U \text{ is constant}\},$$
be the set of \textbf{regular points} of $(M,\cF)$. 

\begin{corollary}
    The set $M_{\operatorname{reg}}$ is a dense open subset of $M$.
\end{corollary}

A foundational result in the theory of singular foliations is a local normal form theorem, known as the local splitting lemma. It admits several equivalent statements, see \cite[Section 1.7.3]{LGLR_notes}; we state here the one, namely \cite[Theorem 1.7.17]{LGLR_notes}, which is most convenient for our purposes.

\begin{theorem}[Local splitting]\label{thm:localsplitsingular}
    Let $\mathcal{F}$ be a singular foliation on $M$, with $\dim M=m$. Consider a point $p\in M$ and let $k$ denote the dimension of $L_p$. Then there exists an open neighborhood $U\subset M$ of $p$, diffeomorphic to $\Omega\subset \mathbb{R}^k \times \mathbb{R}^{m-k}$ via $\phi$, satisfying that $\phi_*(\mathcal{F}|_U)$ is the product of
    \begin{itemize}
        \item[-] the singular foliation spanned by all vector fields on an open ball in $\mathbb{R}^k$,
        \item[-] a singular foliation $\mathcal{T}$ on an open ball in $\mathbb{R}^{m-k}$ whose set of generators vanish at the origin, and which satisfies $\operatorname{rk}_0(\mathcal{T})=\operatorname{rk}_p(\mathcal{F})-k$.
    \end{itemize}
\end{theorem}

%%%%%%%%%%%%%%%%%%%%%%%%%%%%%%%%%%%%%%%%%%%%%%%%%%%%%%%%%%%%%%%%%%%%%%%%%%%%%%%%%%%%%%%%%%%%%%%%%%%%%%%%%%%%%%%%%

\subsection{Poisson geometry}\label{ss:prelim_poisson_geom}

We proceed to review some notions in Poisson geometry that will be important in the next sections. 
We refer to the excellent and extensive monograph \cite{CFM_book} for more background on Poisson geometry. 

%%%%%%%%%%%%%%%%%%%%%%%%%%%%%%%%%%%%%%%%%%%%%%%%%%%%%%%%%%%%%%%%%%%%%%%%%%%%%%%%%%%%%%%%%%%%%%%%%%%%%%%%%%%%%%%%%

\subsubsection{Basic notions}

Throughout this section, we let $M$ be a manifold and $\Pi \in \Gamma(\wedge^2TM)$ be a Poisson bivector field, i.e., a bivector field satisfying the integrability condition $[\Pi,\Pi]=0$, where $[\cdot,\cdot]$ denotes the Schouten-Nijenhuis bracket of multivector fields. 

A map 
    $$f: (M,\Pi) \longrightarrow (M',\Pi')$$
is a \textbf{Poisson map} if $\Pi'(\alpha,\beta)|_{f(p)}=\Pi(df^*\alpha, df^*\beta)|_p$ for any $\alpha, \beta \in T^*M'$. 
We will denote by $\operatorname{Diff}(M,\Pi)$ the set of Poisson diffeomorphisms of $(M,\Pi)$.

Any smooth function $H\in C^\infty(M)$ is referred to as \textbf{Hamiltonian function} whenever one is interested in looking at the following associated Poisson-preserving dynamical system: the \textbf{Hamiltonian vector field} associated to $H$ is $X_H:=\Pi(\d H,\cdot)$, and its flow, which we will denote by $\phi_H^t: M\rightarrow M$, generates a path of Poisson diffeomorphisms of $(M,\Pi)$. 
One can do the same with possibly time-dependent Hamiltonian functions, and the time-one map of their flow will be called a \textbf{Hamiltonian diffeomorphism}. 
We will denote by $\operatorname{Ham}(M,\Pi)$ the set of all such Hamiltonian diffeomorphisms; note that this is a subgroup of the group of all Poisson diffeomorphisms.

The Poisson bivector field $\Pi$ induces a bundle map
\begin{align*}
    \Pi^\sharp: \Omega^1(M)=\Gamma(T^*M)&\longrightarrow \fX (M)=\Gamma(TM)\\
            \alpha &\longmapsto \Pi(\alpha,\cdot)
\end{align*}
whose image $\operatorname{Im}(\Pi^\sharp)$ defines a singular foliation in the sense of \Cref{def:singfol} (i.e.\ of \cite{St,Su}). 

    More explicitly, the leaf of $(M,\Pi)$ through a point $p\in M$ is given by
    $$L_p= \{ \phi^t_H(p) \mid H\in C^\infty(M), t\in \mathbb{R} \},$$
    and $\Pi$ naturally induces on it the following symplectic form 
    \[
    \omega_{L_p}(\Pi^\sharp(\alpha),\Pi^\sharp(\beta))=-\Pi(\alpha,\beta) \; .
    \]
    This singular foliation equipped with its leafwise symplectic forms will be referred to as the \textbf{(singular) symplectic foliation} associated to $\Pi$.
    
    In this Poisson setup, there is also a local splitting form that takes into account the Poisson structure. More precisely, as originally proved by Weinstein \cite{Weinstein_LocalStructure}, every Poisson structure can be put in the following local split form (from which the local description of its associated singular symplectic foliation is immediate):
    \begin{theorem}[Weinstein splitting theorem]
        \label{thm:weinstein_splitting_poisson}
        Let $(M,\Pi)$ be an $m$-dimensional Poisson manifold, and $p\in M$. There is a neighborhood $U$ of $p$ with coordinates $(x_1,y_1,\ldots,x_k,y_k,z_1,\ldots,z_{m-2k})$ centered at $p$ such that
        \[
        \Pi\vert_U = \sum_{i=1}^k \partial_{x_i}\wedge \partial_{y_i} + \sum_{1\leq u,v\leq m-2k}f_{u,v}(z) \partial_{z_u}\wedge \partial_{z_v} \, ,
        \]
        where $f_{u,v}(z)$ are smooth functions vanishing at the origin.
    \end{theorem}

%%%%%%%%%%%%%%%%%%%%%%%%%%%%%%%%%%%%%%%%%%%%%%%%%%%%%%%%%%%%%%%%%%%%%%%%%%%%%%%%%%%%%%%%%%%%%%%%%%%%%%%%%%%%%%%%%

\subsubsection{Symplectic realizations and symplectic groupoids}

Another relevant notion for us is that of symplectic realization, first studied in \cite{Weinstein_LocalStructure}:
\begin{definition}
\label{def:symplectic_realization}
    Given a symplectic manifold $(W,\Omega)$, a map
    $$\mu: (W,\Omega) \longrightarrow (M,\Pi)$$
    is a \textbf{symplectic realization} of $(M,\Pi)$ if it is a surjective submersion which is also a Poisson map.
\end{definition}
More generally, one can require some additional structure on $W$ to be compatible with the Poisson manifolds. 
It turns out that the natural structure to expect in $W$ is that of a symplectic groupoid, which leads to the notion of symplectic integration and integrable Poisson manifold. 
We now briefly recall these notions as well.

A Lie groupoid is a category with structural spaces/maps behaving nicely from a smooth viewpoint. More precisely:
\begin{definition}
    \label{def:lie_groupoid}
    A \textbf{Lie groupoid}\footnote{As we will not look at non-Lie groupoids, with a little abuse of notation we will simply write ``groupoid'' for ``Lie groupoid'' in the rest of the paper.} consists of a set $M$, called set of \textbf{objects}, and a set $\cG$, called set of \textbf{arrows}, together with:
    \begin{enumerate}
        \item smooth and submersive \textbf{source} $s\colon \cG\to M$ and \textbf{target} $t\colon \cG\to M$ maps;
        \item a smooth \textbf{multiplication}  map
        \[
        m\colon \cG^{(2)}=\{(g,h)\in \cG\times \cG \mid s(g)=t(h)\} \to \cG \; , 
        \quad 
        m(g,h)=g\cdot h
        \]
        such that
        \begin{itemize}
            \item[-] $s(g\cdot h)=s(h)$ and $t(g\cdot h)=t(g)$,
            \item[-] $(g\cdot h)\cdot k = g\cdot(h\cdot k)$;
        \end{itemize}
        \item a smooth \textbf{unit} map $u\colon M \to \cG$, $p\mapsto 1_p$, such that
        \begin{itemize}
            \item[-] $s(1_p)=t(1_p)=p$,
            \item[-] $g\cdot 1_{s(g)}=1_{t(g)}\cdot g=g$;
        \end{itemize}
        \item a smooth \textbf{inverse} map $\iota \colon \cG\to \cG$, $g\mapsto g^{-1}$, such that
        \begin{itemize}
            \item[-] $s(g^{-1})=t(g)$ and $t(g^{-1})=s(g)$,
            \item[-] $g^{-1}\cdot g = 1_{s(g)}$ and $g\cdot g^{-1}=1_{t(g)}$.
        \end{itemize}
    \end{enumerate}
\end{definition}

\begin{definition}
\label{def:symplectic_integration}
    A \textbf{symplectic integration} of $(M,\Pi)$ is a symplectic realization $\mu:(W,\Omega) \longrightarrow (M,\Pi)$, and a Lagrangian section $u:M\longrightarrow W$ for which there is a symplectic groupoid structure $(W,\Omega)\rightrightarrows (M,\Pi)$ with target map $\mu$ and unit map $u$. 
\end{definition}
A Poisson manifold is called \textbf{integrable} if such a symplectic integration exists.
In this case, after work of Crainic--Fernandes \cite{CrainicFernandes_IntegrabilityLieBrackets,CrainicFernandes_IntegrabilityPoissonBrackets} there is a unique canonical symplectic integration for which the target map has simply-connected fibers, which we call the symplectic groupoid of $(M,\Pi)$. While we will not need this, note that the existence of a ``local symplectic integration'' was previously known due to work of Weinstein \cite{coste_dazord_weinstein:groupoides_symplectiques}.

\begin{caveat}[Lack of Hausdorff-ness]
\label{cav:sympl_realizations_not_hausdorff}
    Symplectic realizations are not required to be Hausdorff.
    This is because, even if Hausdorff \emph{local} symplectic realizations do exist, in general, global symplectic realizations and integrations will not have this property.
    In particular, the Poisson homotopy groupoid is quite often \emph{not} Hausdorff. 
    This is analogous to the fact that holonomy groupoid of a smooth foliation is more often than not non-Hausdorff.
    Note, however, that even though the uniqueness of limits does not hold in $\Sigma(M)$, considering symplectic homeomorphisms on symplectic realizations as we will do later on still makes sense.
\end{caveat}

%%%%%%%%%%%%%%%%%%%%%%%%%%%%%%%%%%%%%%%%%%%%%%%%%%%%%%%%%%%%%%%%%%%%%%%%%%%%%%%%%%%%%%%%%%%%%%%%%%%%%%%%%%%%%%%%%

\subsubsection{Coisotropic submanifolds}
\label{sec:coisotropics}

We will be interested in coisotropic submanifolds, whose definition is as follows.
\begin{definition}
    A submanifold $C\subset M$ is \textbf{coisotropic} if $\Pi^\sharp(TC^\circ)\subseteq TC$, where $TC^\circ|_p=\{\alpha\in T^*_pM\mid \alpha(v)=0 \text{ for all } v\in T_pC\}$ is the \emph{annihilator} of $TC$.
\end{definition}

When $\Pi$ is non-degenerate, and thus dual to a symplectic form, this coincides with the notion of coisotropic submanifold of a symplectic manifold. 
Recall that in the symplectic context, a coisotropic submanifold admits an induced characteristic foliation, which is a regular foliation given by the kernel of the restriction of $\omega$ to the coisotropic submanifold itself. 

In the Poisson setting, one can similarly define the characteristic distribution of a coisotropic submanifold $C\subset (M,\Pi)$ as $\mathcal{K}_C:=\Pi^\sharp(TC^\circ)$.
In fact, one has the following standard fact (c.f.\ \cite[p. 50]{LGLR_notes}), which we explicitly prove below for completeness:
\begin{lemma}\label{lem:charfol}
    The singular distribution $\mathcal{K}_C$ is a singular foliation.
\end{lemma}

We will refer to $\cK_C$ as the \textbf{(singular) characteristic foliation} of $C$. Notice that $\cK_C$ is contained in $\operatorname{Im}(\Pi^\#)$, which generates the symplectic foliation. In particular, each characteristic leaf lies in a symplectic leaf of $(M,\Pi)$.

\begin{proof}
    Given any point $p \in C$, we consider a neighborhood $U$ of $p$ in $M$ and coordinates $(x_1,...,x_n)$ such that $C\cap U=\{x_{k+1},...,x_n=0\}$, where $k$ is the dimension of $C$ and $n$ the dimension of $M$. Then we have
    $$\mathcal{K}_C|_{C\cap U}=\langle \Pi^\sharp(dx_{k+1}),...,\Pi^\sharp(dx_n)\rangle.$$
    The distribution is thus stable under multiplication by elements in $C^\infty(C)$ (by the $C^\infty$-multilinearity of $\Pi$) and finitely generated. In addition, we have 
    \begin{equation}\label{eq:eq1}
        [\Pi^\sharp(dx_i), \Pi^\sharp(dx_j)]=[X_{x_i},X_{x_j}]=X_{\{x_i,x_j\}}=\Pi^\sharp(d\{x_i,x_j\}).
    \end{equation}
    Using the identity 
    $$\{x_i,x_j\}=\Pi^\sharp(dx_i)(dx_j)$$ 
    and the fact that that $\Pi^\sharp(dx_i)|_C\in TC$ and $dx_j|_C\in TC^\circ$, we deduce that $\{x_i,x_j\}|_C=0$. It follows that $d\{x_i,x_j\}|_C\in TC^\circ$, which implies by Equation \eqref{eq:eq1} that    
$$[\Pi^\sharp(dx_i), \Pi^\sharp(dx_j)] \in \Pi^\sharp(TC^\circ)=\mathcal{K}_C,$$ as we wanted to show.
\end{proof}

\section{Poisson homeomorphisms and their leafwise behavior}\label{sec:leafwisebehavior}
We begin in this section our study of Poisson homeomorphisms, focusing on their leafwise behavior.

%%%%%%%%%%%%%%%%%%%%%%%%%%%%%%%%%%%%%%%%%%%%%%%%%%%%%%%%%%%%%%%%%%%%%%%%%%%%%%%%%%%%%%%%%%%%%%%%%%%%%%%%%%%%%%%%%

\subsection{Poisson, Hamiltonian and foliated homeomorphisms}

We start by defining the notion of $C^0$-convergence of maps that will be used.
Given two manifolds $M$ and $N$, a compact subset $K\subset M$, a Riemannian distance $d$ on $N$ and two maps $f,g\colon M\to N$, one denotes
\[
d_K(f,g)=\sup_{p\in K}d(f(p),g(p)) \, .
\]
Then, a sequence of maps $f_i\colon M\to N$ is said to \textbf{$C^0$-converge} to a map $f\colon M\to N$ if, for every compact subset $K\subset M$, the sequence of real numbers $d_K(f_i,f)$ converges to $0$.
In this case, we denote $f_i\xrightarrow{C^0}f$.
(Note that this notion is independent of the auxiliary Riemannian metric on $N$.)

\medskip

At this point, following \cite{Jo}, one can generalize the definition of symplectic homeomorphism to the Poisson setting. A common definition that is used in the symplectic setting is that the homeomorphism can be globally $C^0$-approximated by symplectic diffeomorphisms. We adapt the definition given in \cite{BO}, as it is more general.

\begin{definition}
    A homeomorphism $\varphi: M \rightarrow M$ is a \textbf{(global) Poisson homeomorphism} of $(M,\Pi)$ if there is a sequence of open sets $U_1\subset U_2 \subset \dots \subset M$ which exhaust $M$, and a sequence of Poisson embeddings $\varphi_i:U_i\longrightarrow M$ that $C^0$-converge to $\varphi$.
\end{definition}
We will denote the set of Poisson homeomorphisms of $(M,\Pi)$ by $\Homeo(M,\Pi)$.

More generally, one can define ``local Poisson homeomorphisms'', analogously to the symplectic case.
\begin{definition}
    A homeomorphism $\varphi: M \longrightarrow M$ is a \textbf{local Poisson homeomorphism}\footnote{Whenever we talk about ``Poisson homeomorphism'', the adjective ``global'' is always implicitly intended. Whenever local Poisson homeomorphisms are considered, the adjective ``local'' will always be explicit.} of $(M,\Pi)$ if for every point $p\in M$, there is a neighborhood $U$ of $p$ such that $\varphi|_U: U \longrightarrow \varphi(U)$ is a Poisson homeomorphism.
\end{definition}
We will denote the set of local Poisson homeomorphisms by $\LocHomeo(M,\Pi)$. 
Already in the symplectic case, it is unknown whether a locally symplectic homeomorphism is always a symplectic homeomorphism. The following result follows by mimicking the proof in the symplectic case, described e.g.\ in \cite[Proposition 1.5]{BO}.
\begin{lemma}
\label{lem:composition_inverse_poisson_homeos}
Let $\varphi, \phi:M\longrightarrow M$ be Poisson (respectively locally Poisson) homeomorphisms in $(M,\Pi)$. Then $\varphi^{-1}$ and $\varphi\circ \phi$ are Poisson (respectively locally Poisson) homeomorphisms. For an open set $U\subset M$, the map $\varphi|_U:U\longrightarrow \varphi(U)$ is a Poisson (respectively locally Poisson) homeomorphism. 
\end{lemma}

Among Poisson homeomorphisms, one can define ``Hameomorphisms''. The terminology is reminiscent of ``Hamiltonian homeomorphism", although the latter is often used in the symplectic setting for a weaker notion, see Remark \ref{rmk:Hamhomeos}. 
The notion of Hameomorphism is the natural generalization of the one introduced in \cite{MO}, and widely exploited in \cite{HLS15} in the setting of coisotropic submanifolds in symplectic manifolds. 

\begin{definition}
\label{def:hameotopy_C0-hamiltonian}
    Let $(M,\Pi)$ be a Poisson manifold. An isotopy $\phi^t: M\longrightarrow M$ is called a \textbf{Hameotopy} if there is a sequence of time-dependent smooth Hamiltonian functions $H_i\in C^\infty([0,1]_t\times M)$ and with support in a compact subset $K\subset M$ uniformly in $t$, such that:
    \begin{itemize}
        \item[-] The sequence of flows $\phi^t_{H_i}$ $C^0$-converges to $\phi^t$ uniformly in $t$.
        \item[-] The sequence $H_i(t,\cdot)$ $C^0$-converges to a continuous time-dependent function $H(t,\cdot)$.
    \end{itemize}
\end{definition}
We will say that the function $H$ \textbf{generates} $\phi^t$, and call it a \textbf{continuous Hamiltonian}, or \textbf{$C^0$-Hamiltonian} in short.
The time-one map of a Hameotopy will be called a \textbf{Hameomorphism}. These are also called \emph{strong Hamiltonian homeomorphisms} in the literature, see e.g.\ \cite{Buh}. 
\begin{remark}\label{rmk:nonuniqueness_C0_Hamiltonians}
    In the symplectic setting, the $C^0$-Hamiltonian generating a Hameotopy is known to be unique up to a time-dependent constant \cite{Vit_uniqueness, BS_uniqueness}. 
    This is false in the Poisson setting, already in the smooth setup, as one can change a given Hamiltonian by a function that is constant in each symplectic leaf without changing the isotopy generated by it. 
    This being said, what still holds in the Poisson setup is the following: analogously to the fact that smooth Hamiltonian isotopies are generated by a smooth Hamiltonian that is unique up to addition of a smooth Casimir function (i.e.\ a smooth function that is constant on each leaf of the singular symplectic foliation), it follows directly from said $C^0$-symplectic results that Hameotopies are generated by a $C^0$-Hamiltonian that is unique up to addition by a \emph{continuous} Casimir function.
\end{remark}
\begin{remark}\label{rmk:Hamhomeos}
    As we mentioned, other notions involving continuous Hamiltonian diffeomorphisms exist, see for instance the discussion in \cite{Buh}. A \emph{Hamiltonian homeomorphism} often refers to a homeomorphism that is the $C^0$-limit of Hamiltonian diffeomorphisms. One can further require that the Hamiltonians generating the approximating Hamiltonian diffeomorphisms are uniformly bounded in the $C^0$-norm, in which case one talks about \emph{finite energy Hamiltonian homeomorphism}. In this work, we will construct some pathological examples of Hameomorphisms, as this is the strongest definition.
\end{remark}

Lastly, as Poisson diffeomorphisms preserve the singular symplectic foliation associated to the Poisson structure, Poisson homeomorphisms are in particular $C^0$-limits of diffeomorphisms preserving such singular foliation. 
More generally, one can speak of (global and local) foliated homeomorphisms.
\begin{definition}
Let $\mathcal{F}$ be a singular foliation in $M$. 
A homeomorphism $\varphi: M\longrightarrow M$ is a \textbf{(global) foliated homeomorphism} if there is a sequence of open sets $U_1\subset U_2 \subset \dots \subset M$ which exhaust $M$, and a sequence of foliation-preserving embeddings $\varphi_i:U_i\longrightarrow M$ that $C^0$-converge to $\varphi$.
\end{definition}
One defines analogously \textbf{local foliated homeomorphisms}. As in \Cref{lem:composition_inverse_poisson_homeos} and its symplectic equivalent, composition and inverses of (global, resp.\ local) foliated homeomorphisms are also (global, resp.\ local) foliated homeomorphisms. This definition leads to the topological counterpart of Question \ref{q1}: does a (possibly locally) foliated homeomorphism map leaves to leaves homeomorphically? We will first address this question. 

%%%%%%%%%%%%%%%%%%%%%%%%%%%%%%%%%%%%%%%%%%%%%%%%%%%%%%%%%%%%%%%%%%%%%%%%%%%%%%%%%%%%%%%%%%%%%%%%%%%%%%%%%%%%%%%%%

\subsection{Foliated homeomorphisms map leaves homeomorphically}
In this section, we show that (locally) foliated homeomorphisms do indeed map leaves to leaves homeomorphically. 

\begin{theorem}\label{thm:foliatedhomeo}
    Let $\mathcal{F}$ be a singular foliation of $M$, and $\varphi: M\longrightarrow M$ a local foliated homeomorphism. 
    Then $\varphi$ maps leaves of $\mathcal{F}$ to leaves of $\mathcal{F}$ homeomorphically. 
\end{theorem}

Since (local) Poisson homeomorphisms are in particular (local) foliated homeomorphisms of the symplectic foliation, we obtain a positive answer to the first part of Question \ref{q1}. 
\begin{corollary}\label{coro:Poisleaveshom}
    Let $\varphi:M\longrightarrow M$ be a locally Poisson homeomorphism of $(M,\Pi)$. 
    Then $\varphi$ maps leaves of the symplectic foliation to leaves of the symplectic foliation homeomorphically. 
\end{corollary}

We first establish an auxiliary lemma proving that every point in a leaf of given dimension must be mapped to the union of all the leaves of that same dimension. 
To this end, define the subsets 
    \[
    C_k:= \{x\in M\mid \dim_{T\mathcal{F}}(x) \leq k\} \, ,
    \quad 
    D_k:=C_k\setminus C_{k-1} 
    \text{ for } k\in \{1,...,n\} \, , 
    \quad 
     D_0:=C_0 \, ,
    \]
    where $n=\dim(M)$.
    In particular, the set $D_k$ is the union of all the leaves of dimension $k$.

    \begin{lemma}\label{lem:dimensionleaveshomeo}
    Let $\varphi:M\longrightarrow M$ be a local foliated homeomorphism. Then $\varphi(D_k) = D_k$.
    \end{lemma}

\begin{proof}
    By Lemma \ref{lem:rankcont}, the function $\dim_{T\mathcal{F}}:M\longrightarrow \mathbb{N}$ is lower semi-continuous; in other words, the sets $C_k$ are closed in $M$ for all $k\in\{0,\ldots , \dim(M)\}$.

    Now, for a sufficiently small neighborhood $U'$ of $p$, we know that the map 
    $$\varphi|_{U'}:U' \longmapsto \varphi(U')\subset M$$ 
    is a foliated homeomorphism; that is, there exists a sequence of foliation preserving embeddings 
    $$\varphi_i: U_i\longrightarrow M$$ 
    converging to $\varphi$ on every compact subset of $U'$, where the sequence $U_i$ is an open exhaustion of $U'$. 
    Let $U\subset U'$ be a smaller neighborhood of $p$ such that  $U\subset U_i$ for all $i$'s, and $V:=\varphi(U)$; we then have $\varphi_i|_U\xrightarrow{C^0} \varphi|_U$.

    Note now that since $\varphi_i$ preserves $\cF$, we must have that $\varphi_i(D_k)= D_k$ for each $i$; in particular, $\varphi_i(D_k\cap U)= D_k\cap \varphi_i(U)$.
    We now claim that $\varphi(D_k\cap U)= D_k\cap V$ for all $k$, where we recall that $V=\varphi(U)$.
    Note that, by the arbitrary choice of initial point $p$, this readily implies $\varphi(D_k)= D_k$, as desired.
    Let us then prove such claim, which we do by induction on $k=0,\ldots , \dim (M)$.
    \smallskip
    
    \textbf{Case $k=0$.} Consider a point $x\in U\cap D_0$, and $y=\varphi(x)$. Then we know that $\varphi_i(x) \xrightarrow{i \to \infty} y$, and that $\varphi_i(x)\in D_0$ for all $i$. Since $D_0=C_0$, which is closed, we must then have $y\in D_0$. This proves that $\varphi(D_0\cap U)\subset D_0\cap V$; applying the same argument to $\varphi^{-1}$ gives the equality, as desired.
\medskip

    \textbf{Inductive step.} 
    By the inductive hypothesis, we know that $\varphi(D_j\cap U)= D_j\cap V$ for each $j=0,\ldots ,k$, where $k<\dim(M)$. 
    Now, since $C_{k+1}$ is closed and $\varphi_i(C_{k+1}\cap U)\subset C_{k+1}$, arguing as in the case of $k=0$ (on $\varphi$ and its inverse), we deduce that $\varphi(C_{k+1}\cap U) = C_{k+1}\cap V$. 
    Moreover, since $C_{k}\cap U= \bigsqcup_{j=0}^k (D_k\cap U)$, arguing as before with the inverse we must have $\varphi(C_k\cap U)=C_k\cap V$. 
    We then also have
    \begin{align*}
        \varphi(D_{k+1}\cap U)&=\varphi((C_{k+1}\setminus C_k)\cap U)\\
        &=\varphi(C_{k+1}\cap U)\setminus \varphi(C_k\cap U)\\
        &=\varphi(C_{k+1}\cap U)\setminus (C_k\cap V)\\
        &= \left(C_{k+1}\cap V\right)\setminus (C_k\cap V)\\
        &=D_{k+1}\cap V,
    \end{align*}
    as we wanted to show.
\end{proof}

We can now proceed to the proof of Theorem \ref{thm:foliatedhomeo}.

\begin{proof}[Proof of Theorem \ref{thm:foliatedhomeo}]
Given $p\in M$, we want to show that the leaf $L_p$ of the singular foliation passing through $p$ is mapped to the leaf $L_q$ through $q=\varphi(p)$. 
Note that, once this is established, applying the same conclusion to the inverse $\varphi^{-1}$ of $\varphi$, we deduce that $\varphi$ maps $L_p$ to $L_q$ homeomorphically, as desired.

By contradiction, let us then assume that there is a point $x\in L_p$ such that $y=\varphi(x) \not \in L_q$; equivalently, we would have $L_y\neq L_q$. 
By Lemma \ref{lem:dimensionleaveshomeo}, the set of leaves of any given dimension is invariant under $\varphi$.
We must then have
$$\begin{cases}
    \dim L_p= \dim L_q,\\
    \dim L_x= \dim L_y,
\end{cases}$$
and since $x$ and $p$ belong to the same leaf, i.e.\ $L_x=L_p$, all these dimensions agree, and we denote them by $k$. 
Let now
$$\gamma: I \longrightarrow L_p,$$
be a smooth path in $L_p$ satisfying $\gamma(0)=x$ and $\gamma(1)=p$. 
Again, by \Cref{lem:dimensionleaveshomeo} we know that $\varphi(\gamma(I))\subset D_k$. 
Since the extreme points of $\gamma$ are mapped by $\varphi$ to different leaves, there must be some value $t_0\in [0,1]$ such that for every $\delta>0$, the set $(\varphi\circ\gamma)((t_0-\delta,t_0+\delta))$ intersects more than one leaf. 
(Strictly speaking, it could happen that $t_0=0$ or $t_0=1$, in which case the sets to consider are those of the form $(\varphi\circ\gamma)([0,\delta))$ or $(\varphi\circ\gamma)((1-\delta,1])$, but to avoid overburdening notation we don't treat explicitly these situations as one can argue analogously.)
\medskip

Taking a small enough neighborhood $U$ of $w_0:=\gamma(t_0)$, we can assume that there is a sequence of foliation-preserving diffeomorphisms $\varphi_i: U_i \longrightarrow M$ converging to $\varphi|_U$ with $U\subset U_i$ for each $i$. 
According to the local splitting \Cref{thm:localsplitsingular}, there is a neighborhood $V$ of the point $z_0=\varphi(w_0)$ that is (up to diffeomorphism) of the form $W_1\times W_2 \subset \mathbb{R}^k \times \mathbb{R}^{m-k}$, where $m=\dim(M)$, in such a way that $z_0\simeq (0,0)\in W_1\times W_2$ and that $\mathcal{F}|_U$ is the product of
    \begin{itemize}
        \item[-] the singular foliation in $W_1\subset \mathbb{R}^k$ spanned by all vector fields,
        \item[-] a singular foliation $\mathcal{T}$ on $W_2$ in $\mathbb{R}^{m-k}$ whose set of generators vanish at the origin.
    \end{itemize}
In particular, all leaves of $\mathcal{F}|_U$ of dimension $k$ are of the form $W_1\times \{w_2\}$ for some $w_2\in W_2$. In other words, we have
$$D_k\cap U = \bigcup_{a\in A} W_1 \times \{a\} = W_1 \times A, $$
for some subset $A\subset W_2$ containing $0\in \R^{m-k}$. Up to shrinking $U$, we can moreover assume that $\varphi(U)\subset V$. 
Consider then the projection onto the second factor
$$\pi_2 : W_1\times W_2 \longrightarrow W_2\subset \R^{m-k},$$
and the curve
\begin{align*}
    \eta:(-\tau,\tau) &\longrightarrow W_2\\
    t &\longmapsto (\pi_2\circ\varphi\circ\gamma)(t_0+t), 
\end{align*}
for some small enough $\tau>0$. 
We know that this curve is not constant by choice of $t_0$, and that its image is contained in the set $A\subset W_2$ (which is, a priori, an arbitrary closed set of $\R^{n-k}$), see Figure \ref{fig:foliatedhomeo}.

\begin{figure}[!ht]
    \centering
    \begin{tikzpicture}
    \node at (0,0) {\includegraphics[width=0.7\linewidth]{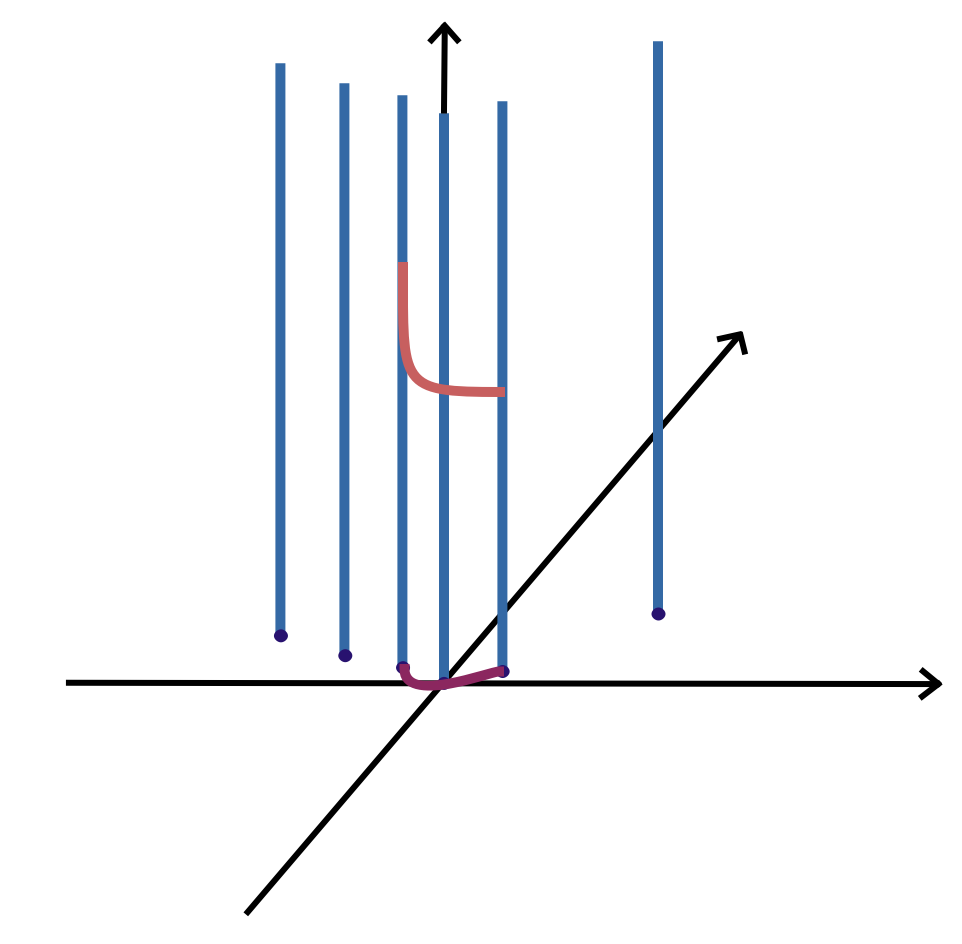}};
    \node[scale=1.2] at (4.6,-0.6) {$\R^{n-k}$};
    \node[scale=1.2] at (0,5.2) {$\R^k$};
    \node[scale=1.2, color=myorange] at (1,1.5) {$\varphi\circ \gamma$};
    \node[scale=1.2, color=mymagenta] at (0.8,-1.6) {$\eta$};
    \node[scale=1.2, color=myblue] at (-2.7,1) {$D_k$};
    \node[scale=1.2, color=mydarkblue] at (-2.7, -1.6) {$A$};
    \end{tikzpicture}
    \caption{Curve in $D_k$ and in $A$}
    \label{fig:foliatedhomeo}
\end{figure}
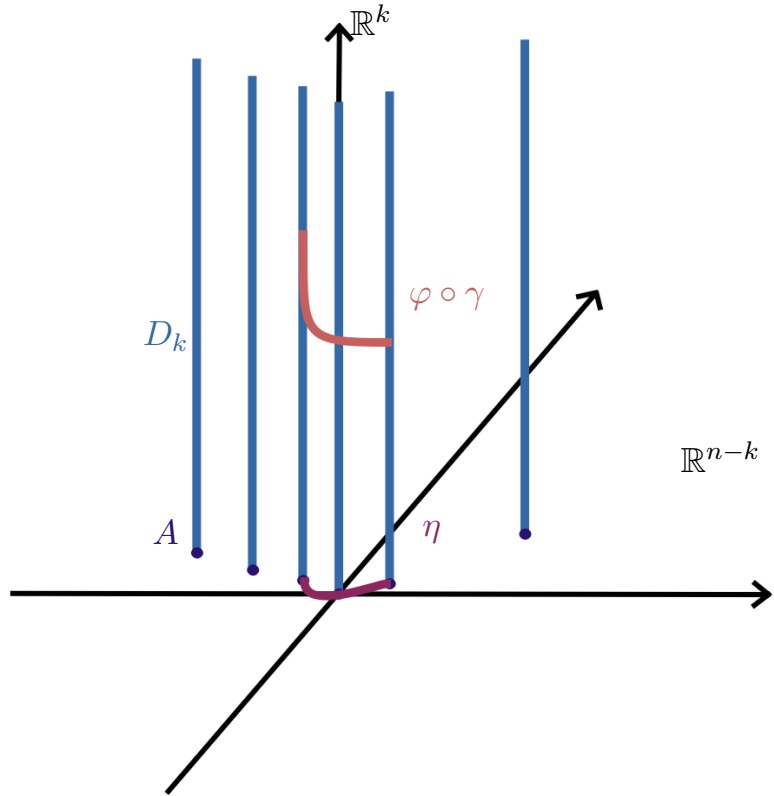

By continuity, for $i$ large enough, the curve $\eta_i=\pi_2(\varphi_i(\gamma(t_0+\cdot))$, defined on some small interval centered at $0$ and valued in $W_2$, is not constant either. 
By how the singular foliation looks like in the local splitting neighborhood, this implies that $t\mapsto \varphi_i(\gamma(t_0+t))$ cannot be completely contained in a single leaf: indeed, if this were the case, its post-composition with $\pi_2$ would need to map to the single point $0\in A\subset W_2$. 
Now, this is a contradiction with the fact that $\varphi_i$ is a foliation-preserving embedding. 
We hence conclude that $\varphi(L_p)\subset L_q$, as we wanted to show.
\end{proof}

\medskip

We now want to argue that, under strong enough assumptions (inspired by the notion of Hameomorphism), a limit of foliated diffeomorphisms mapping each leaf to itself must also preserve each leaf individually. In order to state this precisely, let us introduce some notions:
\begin{definition}
    \label{def:C0-isotopy}
    Let $M$ be a smooth manifold.
    A \textbf{continuous isotopy} $\Phi\colon [0,1]\times M \to M$ is a continuous map such that $\phi^t:=\Phi(t,\cdot) \colon M \to M$ is a homeomorphism for all $t\in[0,1]$, and $\phi_0=\Id$.
\end{definition}
\begin{definition}
    \label{def:foliated_leafwise_C0-isotopy}
    Let $M$ be a smooth manifold and $\cF$ a singular foliation on $M$.
    A continuous isotopy $\Phi\colon [0,1]\times M \to M$ is said to be a \textbf{(global) foliated $C^0$-isotopy} if there is a sequence of open sets $U_1\subset U_2\subset \ldots \subset M$ exhausting $M$, and a sequence $(\Psi_n)_n$ of smooth maps $\Psi_n\colon [0,1]\times U_i \to M$ that $C^0$-converges to $\Phi$ and such that each $\psi_{n}^t:=\Psi_n(t,\cdot)$ is an embedding preserving the singular foliation $\cF$.
    We also call $\Phi$ a \textbf{local foliated $C^0$-isotopy} if for each point $p\in M$ there is an open neighborhood $U$ such that $\Phi\vert_{[0,1]\times U}$ is a foliated $C^0$-isotopy.
    Lastly, $\Phi$ is called \textbf{(global, resp. local) leafwise $C^0$-isotopy} if it is a (global, resp.\ local) foliated $C^0$-isotopy in such a way that, for $n$ big enough and for all $t\in[0,1]$, $\psi_n^t$ (set-wise) preserves each leaf of $\cF$.
\end{definition}

The desired result is then easy to prove, and does not even rely on Theorem \ref{thm:foliatedhomeo}.
\begin{lemma}
    Let $M$ be a smooth manifold and $\cF$ a singular foliation on $M$.
        If $\Phi$ is a leafwise $C^0$-isotopy, then $\phi^t:=\Phi(t,\cdot)$ maps each leaf to itself.
\end{lemma}

\begin{proof}
Consider a point $p\in M$ and a splitting chart $W$ centered at $p$. Taking a small enough neighborhood $W'$ of $p$ in $W$, there is an approximating sequence
$$ \psi_i^t : W'\subset W \longrightarrow M, \quad t\in [0,1]$$
that preserves each leaf of the foliation and that $C^0$-converges to $\psi^t|_{W'}$.

For $t$ small enough each $\phi^t_i(p)$ belongs to the plaque $L_p^W\subset L_p$ of $p$ inside $W$.
Since each plaque of dimension $\dim L_p$ is a locally closed set in $W$, we must have that the point $\phi^t(p)$
belongs to $L_p^W$ for $t$ small enough. 
Now, covering the path $\phi^t(p)$ by splitting charts and repeating this argument, we conclude that $\phi^t(p)$ is in $L_p$ for all $t$. 

This holds for every point $p$, proving that every leaf is preserved by $\phi^t$, as desired. 
\end{proof}

We point out the following consequence for Hameomorphisms:

\begin{corollary}\label{cor:Hameos_preserve_leaves}
        If $\varphi:(M,\Pi) \longrightarrow (M,\Pi)$ is a Hameomorphism, then it maps each leaf to itself.
\end{corollary}
%%%%%%%%%%%%%%%%%%%%%%%%%%%%%%%%%%%%%%%%%%%%%%%%%%%%%%%%%%%%%%%%%%%%%%%%%%%%%%%%%%%%%%%%%%%%%%%%%%%%%%%%%%%%%%%%%

\subsection{Poisson homeomorphisms restrict to symplectic homeomorphisms}

Using our preliminary result on foliated homeomorphisms, we can establish \Cref{thm:main1} from the Introduction, i.e.\ that, in the Poisson case, the induced homeomorphism between leaves is in fact a symplectic homeomorphism.

\begin{theorem}\label{thm:poisson_homeos_are_leafwise_sympl_homeos}
    Let $\varphi$ be a local Poisson homeomorphism. Then for every symplectic leaf $L$ of $(M,\Pi)$, the restriction
    $$\varphi|_L : L \longrightarrow \varphi(L),$$
    is a local symplectic homeomorphism.
\end{theorem}
Note that $\varphi(L)$ is a leaf of the singular (symplectic) foliation associated to $\Pi$ by \Cref{thm:foliatedhomeo}.

\begin{proof}
We aim to show that the restriction of $\varphi$ to a sufficiently small neighborhood $U$ of any point $p$ inside its symplectic leaf $L_p$ is a symplectic homeomorphism to a neighborhood $V$ of $\varphi(p)$ inside $L_{\varphi(p)}$. 

\medskip

Choose a small enough neighborhood $U$ of $p$ such that there exists a family of Poisson embeddings $\varphi_i:U_i\longrightarrow M$ so that $\varphi_i|_{U_i}$ converges to $\varphi|_U$.
Consider then a neighborhood $W_q$ of $q=\varphi(p)$ given by the Weinstein splitting \Cref{thm:weinstein_splitting_poisson}, and denote by $2k$ the dimension of $L_p$ and by $m$ the dimension of $M$.
There are then coordinates $(x_1,y_1,...,x_k,y_k, z_1,...,z_{m-2k})$ of $W_q$ centered at $p$ and such that the Poisson structure writes as 
\begin{equation}
    \Pi|_{W_q}= \sum_{i=1}^k \partial_{x_i}\wedge \partial_{y_i} + \sum_{1\leq u,v \leq m-2k} f_{u,v}(\mathbf{z}) \partial_{z_u}\wedge \partial_{z_v},
\end{equation}
where $f_{u,v}(\mathbf{z})$ are functions vanishing at the origin. In particular, very much like for the splitting theorem for singular foliations, the intersection of any leaf of dimension $2k$ with $W_q$ is necessarily given by disjoint unions of subsets of the form $\{z_i=c_i \mid i=1,...,m-2k \}\subset W_q$, where the $c_i$'s are constants. 

\medskip

Notice now that if we knew that $q_i=\varphi_i(p)$ belonged to $L_q$ for $i$ large enough, the conclusion would follow. 
Indeed, in this case, there would be a small neighborhood $V$ of $p$ inside $L_p$ for which $\varphi_i|_{V}\colon V\hookrightarrow L_q$ would give a $C^0$-approximation by symplectic embeddings of the homeomorphism $\varphi|_{V}\colon V \to L_q$, as required by the definition of local symplectic homeomorphism.

Guided by this observation, the main idea is then as follows. The points $q_i=\varphi_i(p)$ are in $W_q$ for $i$ large enough, but in general belong to a different $k$-dimensional plaque of the foliation than the plaque $Z_q$ where $q$ belongs. We ``correct" the $\varphi_i$'s by composing with the projection onto $Z_q$ in a Weinstein splitting chart of the form $Z_q\times T$ where $Z_q$ is a symplectic factor and $T$ is the transverse Poisson factor with Poisson structure vanishing at the origin.
Here are the details.

\medskip

By \Cref{thm:foliatedhomeo} we know that a small neighborhood $V$ of $p$ inside $L_p$ must be mapped, for $i$ large enough, to the intersection of a leaf of dimension $2k$ with $W_q$. 
In other words, we have that $V_i':=\varphi_i(V)$ lies in a set of the form
$$L_i:=\varphi_i(L_p)\cap W_p\subset \{(x_j,y_j,z_l)\mid j=1,\ldots, k \, , \,\text{ and }\,  z_l=b_l^i \text{ for } l=1,...,m-2k\},$$ 
for some constants $b_l^i$. 
Since $\varphi_i(p)\to\varphi(p)$ and $\varphi(p)$ is mapped to the origin in $W_q$, these constants must satisfy
\[
b_l^i \xrightarrow{i\rightarrow \infty} 0 \, .
\]
At this point, consider the projection 
\begin{align*}
    \pi: W_q &\longrightarrow Z_q\subset L_q\cap W_q\\
        (x_i,y_i,z_i) &\longmapsto (x_i,y_i,0)
\end{align*}
where $Z_q$ is the connected component of $L_q\cap W_q$ containing $q\simeq 0$,
and notice that 
$$\phi_i:=\pi \circ \varphi_i:V \longrightarrow L_q$$ 
is a symplectomorphism from $U$ to $\pi\circ \varphi_n(U)$. 
Indeed, we know that 
$$\varphi_i: (V, \omega_{L_p})\longrightarrow Z_i\subset (L_i, \omega_{L_i})$$
where $Z_i$ is the connected component of $L_i\cap W_q$ containing $\varphi_i(p)$,
is a symplectic embedding, because $\varphi_i$ is a Poisson diffeomorphism into its image, and the projection map 
$$\pi|_{Z_i}: (Z_i, \sum_{i=1}^k dx_i\wedge dy_i)\longrightarrow (L_q, \sum_{i=1}^k dx_i\wedge dy_i)$$
is trivially a symplectomorphism onto its image. 

Now, denoting by $V':=\varphi(V)$, we know by \Cref{thm:foliatedhomeo} that $V'$ is an open neighborhood of $q$ inside $L_q$; up to shrinking $V$, we can assume that $V'\subset Z_q\subset W_q$. 
We then define the open sets
$$V_i= (\phi_i)^{-1}(\phi_i(V)\cap V'),$$
and we denote the restricted maps by
$$\psi_i:=\phi_i|_{V_i}:V_i \longrightarrow V' \, .$$
Then we have
$$\psi_i \xrightarrow{C^0} \varphi \, ,$$
and each $\psi_i$ is a symplectic embedding by what said above, thus proving that $\varphi|_V$ is a symplectic homeomorphism, as desired.
\end{proof}

We point out that the proof above relies strongly on the local description of the Poisson structure, and hence of its singular symplectic foliation, given by the Weinstein splitting \Cref{thm:weinstein_splitting_poisson}. 
In particular, it does not allow one to prove that (global) Poisson homeomorphisms induce (global) symplectic homeomorphisms between leaves in general.

%%%%%%%%%%%%%%%%%%%%%%%%%%%%%%%%%%%%%%%%%%%%%%%%%%%%%%%%%%%%%%%%%%%%%%%%%%%%%%%%%%%%%%%%%%%%%%%%%%%%%%%%%%%%%%%%%

%%%%%%%%%%%%%%%%%%%%%%%%%%%%%%%%%%%%%%%%%%%%%%%%%%%%%%%%%%%%%%%%%%%%%%%%%%%%%%%%%%%%%%%%%%%%%%%%%%%%%%%%%%%%%%%%%
%%%%%%%%%%%%%%%%%%%%%%%%%%%%%%%%%%%%%%%%%%%%%%%%%%%%%%%%%%%%%%%%%%%%%%%%%%%%%%%%%%%%%%%%%%%%%%%%%%%%%%%%%%%%%%%%%

\section{Clean intersection points}\label{sec:cleanpoints}
In this section, we introduce and study the notion of ``clean intersection point", which plays a pivotal role in the understanding of the $C^0$-rigidity of coisotropic submanifolds.

\subsection{Clean intersections points}
We start by recalling the classical notion of ``clean intersection" between two submanifolds:
\begin{definition}\label{def:cleanintsub}
    Two submanifolds $N_1, N_2$ of an ambient manifold $M$ are said to \textbf{intersect cleanly} if
\begin{itemize}
    \item[-] $N_1\cap N_2$ is a smooth submanifold, and
    \item[-] for every point $x\in N_1\cap N_2$, we have $T_x(N_1\cap N_2)=T_xN_1\cap T_xN_2$.
\end{itemize}
\end{definition}
Notice that if two submanifolds intersect transversely, then they intersect cleanly, but the converse is not necessarily true. 

In the setting where $M$ is endowed with a singular foliation $\mathcal{F}$, which is the case of interest for us, one can similarly require that a submanifold intersects cleanly with $\mathcal{F}$:
\begin{definition}\label{def:cleansubfol}
    Let $\mathcal{F}$ be a singular foliation on $M$. 
    A submanifold $N\subset M$ is said to \textbf{intersect cleanly with $\mathcal{F}$} if for every leaf $L$ of $\mathcal{F}$, the submanifolds $N$ and $L$ intersect cleanly.
\end{definition}

When studying submanifolds in Poisson manifolds, requiring that they intersect cleanly with the symplectic foliation is a common (yet rather strong) assumption made in the literature, see e.g.\ \cite{Weinstein_CoisotropicCalculus}. 
In order to study arbitrary submanifolds and to make sense of ``the locus where they intersect cleanly'', we introduce the key notion of ``clean intersection point", which we could not find defined/studied anywhere in the literature.

\begin{definition}\label{def:cleanpointsub}
    Let $N_1,N_2$ be two submanifolds of $M$. 
    A point $p\in N_1\cap N_2$ is a \textbf{clean intersection point} if there are neighborhoods $U_1, U_2$ of $p$ in $N_1$ and $N_2$ respectively, such that $U_1$ and $U_2$ intersect cleanly.
\end{definition}

\begin{remark}
\label{rmk:alternative_def_cleanpoint}
Another possible natural definition for clean intersection point is that sufficiently small neighborhoods satisfy that $U_1\cap U_2$ is a submanifold and $T_p(U_1\cap U_2)=T_pU_1\cap T_pU_2$. 
In fact, it turns out that this notion is equivalent to the one in \Cref{def:cleanintsub}, despite sounding weaker a priori. 
Indeed, suppose that a point $p$ is clean in this second sense, but not in the sense of \Cref{def:cleanintsub}. 
Then there is a sequence of points $p_i\in U\cap V$ such that $p_i\rightarrow p$ such that $T_{p_i}(U\cap V)\subsetneq T_{p_i}U\cap T_{p_i}V$. 
If $k$ denotes the dimension of $U\cap V$, we must have $\dim (T_{p_i}U\cap T_{p_i}V) \geq k+1$. 
As this is a closed condition, we conclude that $\dim (T_pU\cap T_pV)\geq k+1$, which is a contradiction.
\end{remark}

Similarly, we have clean intersection points with an ambient singular foliation:

\begin{definition}\label{def:cleanpointfol}
    Let $\mathcal{F}$ be a singular foliation in $M$, and $N\subset M$ a submanifold. We say that $p\in N$ is a \textbf{clean intersection point of $N$ with $\cF$} if there are neighborhoods $U$ of $N$ and $V$ of the leaf $L_p$ of $\mathcal{F}$ containing $p$ such that $U$ and $V$ intersect cleanly.
\end{definition}

We denote by $\Cl(N,\cF)\subseteq N$ the set of clean intersection points of $N$ with $\cF$.

\begin{remark}\label{rmk:openleafclean}
As we will see later on, the set $\Cl(N,\cF)$ is, in general, not open, even if $\cF$ is regular. However, it follows from the definition that if $p$ is a clean intersection point, then for small enough neighborhoods $U$ and $V$ of $p$ inside $C$ and $L_p$ respectively, the set $U\cap V$ is made of clean intersection points.
\end{remark}

\begin{remark}
    It is straightforward to check that given a Poisson diffeomorphism $\psi:(M,\Pi)\longrightarrow (M,\Pi)$, a submanifold $N$, and a point $p\in N$ that is a clean intersection point with the singular symplectic foliation, then $\psi(p)$ is a clean intersection point of $\psi(N)$ with the singular symplectic foliation. 
    As we will see later in \Cref{ex:cleantononclean}, this is not the case anymore for Poisson homeomorphisms (for which the image of $N$ is a smooth submanifold), even if $N$ is coisotropic.
\end{remark}

\begin{remark}\label{rmk:cleanimmersed}
    In the above definitions, by submanifold we really mean \emph{embedded submanifold}. 
    Even though we will not need this, let us mention that one can generalize  \Cref{def:cleanintsub,def:cleanpointsub} to the case where $N_1$ and/or $N_2$ is/are (possibly non-injectively) immersed submanifolds. 
    Similarly, \Cref{def:cleansubfol,def:cleanpointfol} can be adapted to the case where $N$ is a (possibly non-injectively) immersed submanifold in a singular foliated manifold. 
    In these cases, clean intersection points are defined using neighborhoods in the source manifold of parameterizing immersions. 
    Namely, for instance in the case of an immersed submanifold and a singular foliation, if $\iota:N \longrightarrow M$ denotes the immersion of $N$ in $M$, a point $p\in N$ is a clean intersection point if there is a neighborhood $U$ of $p$ in $N$ such that $\iota(U)$ is an embedded submanifold for which $\iota(p)$ is a clean intersection point with $\cF$.
\end{remark}

Let us give an easy but illustrative example of a submanifold with different types of clean and non-clean intersection points with respect to a regular foliation. 

\begin{example}
Consider on $\mathbb{R}^n$, equipped with coordinates $(x_1,...,x_n)$, the regular corank-one foliation whose leaves are given by fixing the value of the last coordinate to a constant. 
Consider also a submanifold of $\R^n$ of the form
$$N=\{(x_1,...,x_n)\mid x_n=\varphi(x_1)\},$$
where $\varphi: \R\longrightarrow \R$ is a smooth function with the following properties:
\begin{itemize}
    \item[-] it is strictly increasing in $(-\infty,-1)\cup(1,\infty)$;
    \item[-] it is constant and equal to zero in $[-1,1]$.
\end{itemize}
All points of $N$ in with $x_1\in (-\infty,1)\cup (1,\infty)$ are points where $N$ intersects transversely the corresponding leaf $x_n=\varphi(x_1)$, and thus are clean intersection points. 
The codimension-two affine subspaces $A_{-1}=\{x_1=-1, \, x_n=0\}\subset N$ and $A_1=\{x_1=1,\, x_n=0\}\subset N$ are made of non-clean intersection points:
indeed, any neighborhood $U$ of any point $p\in A_{-1}\cup A_1$ inside $N$, the intersection $U\cap L_p=U\cap \{x_n=0\}$ is not a submanifold (it is in fact a submanifold with boundary). 
Lastly, every point in $A:=\{-1<  x_1< 1, \, x_n=0\}\subset N$ is a point of non-transverse clean intersection with the foliation: indeed, given $p\in A$, a sufficiently small neighborhood $U\subset N$ of $p$ satisfies $U\subset L_p$, and thus $U\cap L_p=U$ is a smooth submanifold which trivially satisfies $T(U\cap L_p)=TU\cap TL_p$.
\end{example}

\subsection{Clean points of coisotropic submanifolds}
When $C$ is a coisotropic submanifold of $M$, clean intersection points have some remarkable geometric properties. The first one is that every characteristic leaf contains only clean or non-clean points.
\begin{lemma}\label{lem:cleanleafcoiso}
    Let $C$ be a coisotropic submanifold. 
    If $p\in C$ is a clean intersection, then so is every point in the leaf $K_p\subset C$ of $\cK_C$ containing $p$. 
\end{lemma}
\begin{proof}
    Let $p\in C$ be a clean intersection point; namely, there exist a neighborhood $U$ of $p$ in $C$ and a neighborhood $V$ of $p$ in $K_p$ such that:
    \begin{itemize}
        \item[-] $U\cap V$ is a smooth submanifold,
        \item[-] $T(U\cap V)=TU\cap TV$.
    \end{itemize}
    Recall that $\cK_C= \Pi^\sharp(TC^\circ)$, and that, as argued in \Cref{lem:charfol}, if $U$ is an open set with coordinates $(x_1,...,x_n)$ such that $C\cap U=\{x_{k+1}=...=x_n=0\}$, one has 
     $$\cK_C|_{C\cap U}=\langle \Pi^\sharp(dx_{k+1}),...,\Pi^\sharp(dx_n)\rangle.$$
    It follows that locally we have 
    \begin{equation}\label{eq:charHam}
    \cK_C=\Pi^\sharp(\{dH\mid H\in \mathcal{I}(C)\},     
    \end{equation}
    where $$\mathcal{I}(C)=\{H\in C^\infty(M) \mid H(q)=0 \text{ for all }q\in C\}$$ 
    is the vanishing ideal of $C$. 
    By the definition of a leaf of a singular foliation, we know that for any $q\in K_p$, there is a family of vector fields $X_1,...,X_m\in \cK_C$ such that their time-one flows $\varphi_1,...,\varphi_m$ satisfy 
    $$\varphi_m\circ...\circ\varphi_1(p)=q.$$

    Now, one can break the orbits of $X_1,...,X_m$ into smaller pieces so that each piece of the relevant integral curve lies in a local chart where Equation \eqref{eq:charHam} holds. Thus, without loss of generality, we can assume that
    $$X_i= \Pi^\sharp(dH_i),$$
    for some function $H_i$ defined in a neighborhood of the relevant orbit.  In particular, the local diffeomorphisms $\varphi_i$ are Hamiltonian diffeomorphisms. In particular, their flows are tangent both to the leaves of $\cK_C$ and of the symplectic foliation of $(M,\Pi)$. In other words, if $\psi:=\varphi_m\circ...\circ \varphi_1$, we have:
    \begin{itemize}
        \item[-] $U_q= \psi(U)$ is a neighborhood of $q$ inside $C$,
        \item[-] $V_q= \psi(V)$ is a neighborhood of $q$ inside the symplectic leaf of $q$, which is $L_p$.
    \end{itemize}
    It follows that $U_q\cap V_q=\psi(U\cap V)$ is a submanifold, and that 
    \begin{align*}
        T(U_q\cap V_q)&= \d\psi(T(U_q\cap V_q))\\
        &=\d\psi(TU\cap TV)\\
        &=\d\psi(TU)\cap \d\psi(TV)\\
        &= TU_q\cap TV_q.
    \end{align*}
    This proves that $q$ is a clean intersection of $C$ with $\cF$, as we wanted to show.
\end{proof}

This allows us to split characteristic leaves into two types.
\begin{definition}\label{def:cleancharleaf}
    Let $C$ be a coisotropic submanifold in $(M,\Pi)$. A characteristic leaf $K$ is a \textbf{clean characteristic leaf} if $K\subset \Cl(C,\cF)$, and a \textbf{non-clean characteristic leaf} if $K\cap \Cl(C,\cF)=\emptyset$.
\end{definition}
As we mentioned, Proposition \ref{prop:cleancoisotropicinleaf} implies that every leaf is either clean or non-clean.

\begin{example}
We now give an example of a coisotropic submanifold with clean and non-clean intersection points with the symplectic foliation of a Poisson manifold. 
Variations on this example will be useful to exhibit different possible behaviors, with respect to the characteristic foliation, of Poisson homeomorphisms mapping a coisotropic submanifold into a smooth submanifold. 
Consider in $\mathbb{R}^{2n+1}$ with $n\geq 1$ and coordinates $(x_i, y_i, z)$ with $i=1,...,n$ the regular Poisson structure
$$\Pi=\sum_{i=1}^n \pp{}{x_i}\wedge \pp{}{y_i}.$$
Then, for $k$ a positive integer, the submanifold
$$C= \{(x_i,y_i,z)\in \R^{n+1}\mid z=x_1^k\}\, ,$$
is automatically coisotropic, because it is of codimension $1$. 
Moreover, if the integer $k$ is strictly larger than $1$, $C$ is moreover transverse to the symplectic leaves except along $\{z=0\}$, which is exactly the locus of non-clean intersection points of $C$ with the singular symplectic foliation. 
\end{example}

\medskip

Let us now pass to another important geometric property of clean intersection points of coisotropics with the singular symplectic foliation, namely that, along clean intersection points, being coisotropic is characterized by restricting to a coisotropic submanifold along the symplectic leaf.

\begin{proposition}\label{prop:cleancoisotropicinleaf}
    Let $C$ be a smooth submanifold in $(M,\Pi)$. Let $p$ be a clean intersection point of $C$ with the symplectic foliation $\cF$. 
    Then $C$ is coisotropic at $p$ if and only if the following holds: 
    for a neighborhood $U$ of $p$ in $C$ and a neighborhood $V$ of $p$ in $L_p$ such that $U\cap V$ is a smooth submanifold satisfying $T(U\cap V)|_p=(TU\cap V)|_p$, the latter is coisotropic (in the symplectic sense) at $p$ inside $(L_p,\omega_{L_p})$.
\end{proposition}

Indeed, it is known that if a coisotropic submanifold intersects cleanly a leaf, then the intersection is coisotropic in the symplectic leaf \cite[Exercise 8.52]{CFM_book}); we nonetheless give an explicit proof of both implications for completeness.

\begin{proof}
 We first prove the first implication. By the definition of clean intersection point, we can find neighborhoods $U$ and $V$ of $p$ inside $C$ and $L_p$ respectively, such that:
 \begin{itemize}
     \item[-] $U\cap V$ is a smooth submanifold of $L_p$,
     \item[-] $T(U\cap V)=TU\cap TV=TC\cap TL_p|_{U\cap V}$.
 \end{itemize}
 To show that $U\cap V$ is coisotropic at $p$ inside $(L_p, \omega_p)$, we first recall that $\omega_{L_p}$ is given by the following identity: for all $q\in L_p$ and $u,v\in T_q (L_p)$, as $T_q(L_p)=\on{Im}(\Pi^\sharp_q)\subset T_qM$ one can write $u=\Pi^\sharp(\alpha)$ and $v=\Pi^\sharp(\beta)$ for some $\alpha,\beta\in T_q^*M$, and 
 $$\omega_{L_p}(u,v)=\omega_{L_p}(\Pi^\sharp(\alpha),\Pi^\sharp(\beta))=-\Pi(\alpha,\beta) \, .$$
 Now, using the symplectic notion of coisotropic, we need to show that
 $$T(U\cap V)^{\omega_{L_p}}=\{u\in TL_p\mid \omega_{L_p}(u,v)=0 \text{ for all } v\in T(U\cap V)\}$$
 is contained in $T(U\cap V)$ at $p$.
 We argue as follows.
 \medskip 
 
 Given $u\in T_p(U\cap V)^{\omega_{L_p}}$, consider the covector $\nu=\omega_{L_p}(u,\cdot)\in T_p^*L_p$. This satisfies
 $$\nu(v)=0,\qquad  \text{for all} \qquad v\in T_p(U\cap V)=T_pC\cap T_pL_p\, .$$
 Since we have that $T(U\cap V)=TC\cap TL_p$, we can find a linear complement $A$ of $T_p(U\cap V)$ inside $T_pC$. Extend then $\nu$ to a covector $ \alpha\in T^*_pM$ by requiring 
$$\begin{cases}
    \alpha(v)= \nu(v) \quad \text{if} \quad v\in T_pL_p\\
     \alpha(v)=0 \quad \text{if} \quad v\in A
 \end{cases},$$
and then further extend by zero on the complement $B$ of $T_pL_p\oplus A$ in $T_p M$. 
Notice that $ \alpha(v)=0$ for all $v\in T_pC$, and thus $ \alpha \in (T_pC)^\circ$. 
In particular, since $C$ is coisotropic, we have $\Pi^\sharp( \alpha)\in T_pC$, and thus $\Pi^\sharp( \alpha)\in T_pC\cap T_pL_p$, since $TL_p$ is the image of $\Pi^\sharp$ at the point $p$. Moreover, as $\Pi^\sharp$ writes, as as map from $T^*_p M \simeq T_p^*(L_p) \oplus A^* \oplus B^*$ to $T_pM=T_p(L_p)\oplus A\oplus B$, in blocks as follows
\[
\left(
\begin{array}{ccc}
   (\omega_{L_p}^\flat)^{-1}  & * & *  \\
   0 & 0 & 0 \\
   0 & 0 & 0 \\
\end{array}
\right) \, ,
\]
we have $\Pi^\sharp (\alpha ) = u$.
We have thus shown that $u\in T_pC\cap T_pL_p=T_p(U\cap V)$, concluding the proof.

\medskip

Conversely, suppose that $U\cap V$ is coisotropic at $p$ inside $(L_p, \omega_{L_p})$, and we want to show that $p$ is coisotropic in $(M,\Pi)$. Let $\alpha \in T_pC^\circ$ be an arbitrary covector in $T^*M_p$ in the annihilator of $T_pC$. The vector $u=\Pi^\sharp(\alpha)$ belongs to $T_pL_p$, since $\on{Im}(\on{\Pi^\sharp})=TL_p$ at the point $p$. We claim that $u$ belongs to $T_p(U\cap V)^{\omega_{L_p}}$. 

To see this, observe that given $v\in T_p(U\cap V)=T_pC\cap T_pL_p$ we have 
$$\omega_{L_p}(u,v)= \omega_{L_p}(\Pi^\sharp(\alpha),v),$$
and, writing $v=\Pi^\sharp(\beta)$ for some $\beta\in T_p^*M$, one gets 
\begin{align*}
    \omega_{L_p}(\Pi^\sharp(\alpha),\Pi^\sharp(\beta)) =- \Pi(\alpha,\beta)=-\iota_\beta\iota_\alpha\Pi=\iota_\alpha(\iota_\beta\Pi)=\alpha(v) \, .
\end{align*}
Hence, we deduce that
$$\omega_{L_p}(u,v)=\alpha(v),$$
which is zero because $\alpha \in T_pC^\circ$ and $v\in T_pC$. 
This shows that $u$ belongs to $T_p(U\cap V)^{\omega_{L_p}}$, and the latter is included in $T_p(U\cap V)=T_pC\cap T_pL_p$ using that $U\cap V$ is coisotropic at $p$ inside $L_p$. We conclude that $u\in T_pC$, and thus that $\Pi^\sharp(T_pC^\circ)\subseteq T_pC$, as we wanted to show.
\end{proof}

In such a situation, characteristic leaves do coincide with the characteristic leaves in the symplectic sense.

\begin{corollary}
    \label{cor:poisson_char_fol_equals_sympl_one_at_clean_intersect}
    Let $C$ be a coisotropic submanifold of a Poisson manifold $(M,\Pi)$, $p$ a clean intersection point of $C$ with the singular symplectic foliation of $\Pi$, and $L$ the singular symplectic leaf of $\Pi$ passing through $p$.
    Then, the characteristic leaf $K_{p}$ of $C$ (in the Poisson sense) passing through $p$ coincides with the characteristic leaf $K_{p}^L$ of $C\cap L$ (in the symplectic sense, i.e.\ in the symplectic leaf $L$) passing through $p$.
\end{corollary}

\begin{proof}
    By \Cref{lem:cleanleafcoiso}, up to applying the same argument moving the point $p$ along its (symplectic and Poisson) characteristic leaf, it is enough to prove the following:
    there is a neighborhood $U$ of $p$ in $M$ such that the connected component $K^U_p$ of $K_p\cap U$ containing $p$ coincides with the connected component $K_p^{L,U}$ of $K_p^L\cap U$ containing $p$.
    We do so as follows.

    \medskip

    Consider a neighborhood $U\simeq \R^{2k}\times \R^{n-2k}$ centered at $p$, as given by the Weinstein split neighborhood \Cref{thm:weinstein_splitting_poisson}, i.e.\ so that $\Pi$ is the sum of the standard non-degenerate Poisson structure $\Pi_1$ on $\R^{2k}$ and a Poisson structure $\Pi_2$ on $\R^{n-2k}$ vanishing at the origin, where $n$ is the dimension of $M$.
    Note that the local symplectic leaf $\R^{2k}$ is a subset of $L$.
    
    By the assumption that $p$ is a clean intersection point of $C$ with the singular symplectic foliation, in this split neighborhood $C$ intersects the local singular leaf $\R^{2k}\times \{0\}$ in a coisotropic submanifold $P\subset \R^{2k}$ passing through the origin.
    Moreover, by the coisotropic normal neighborhood theorem in symplectic geometry, up to applying a Poisson diffeomorphism of the form $\psi\times \Id$, where $\psi$ is a Poisson diffeomorphism of $\Pi_1$, i.e.\ a symplectomorphism of the dual (standard) symplectic form $\omega_1$ on $\R^{2k}$, we can arrange that, for some $1\leq r \leq k$, $P=\{x_{r}=x_{r+1}=\ldots=x_k=0\}\subset \R^{2k}$.
    
    It follows then from an explicit computation that the characteristic leaf (in the symplectic sense) of $P=C\cap \R^{2k}$ passing through $p$ coincides with the characteristic foliation (in the Poisson sense) of $C\cap U$ passing through $p$, and that both are just given by the submanifold
    \[
    \{x_r=y_r=x_{r+1}=y_{r+1}=\ldots=x_k=y_k=0\}\subset P \, .
    \]
    In other words, locally we have $K_p^{L,U}=K_p^U$, concluding the proof.
\end{proof}

%%%%%%%%%%%%%%%%%%%%%%%%%%%%%%%%%%%%%%%%%%%%%%%%%%%%%%%%%%%%%%%%%%%%%%%%%%%%%%%%%%%%%%%%%

\subsection{Points of clean intersection with singular foliations}

In this section, we will prove the following result, which is a reformulation of \Cref{thm:clean_intro} from the Introduction via our new notation.
\begin{theorem}\label{thm:clean}
    Let $\mathcal{F}$ be a singular foliation on $M$, and $N$ an immersed submanifold. Then $\Cl(N,\cF)$ contains an open and dense subset of $N$.
\end{theorem}

We are going to deduce the result above from the following local statement:
\begin{proposition}
    \label{prop:cleanregular_local}
    Consider $\R^n=\R^k\times \R^{n-k}$, with $1\leq k\leq n$, equipped with a singular foliation $\cF$ spanned by all the coordinate vector fields on $\R^k$, and by a singular foliation $\mathcal{T}$ on $\R^{n-k}$ such that the origin is a leaf of dimension $0$.
    Let also $V$ be an embedded (possibly open) submanifold of $\R^n$ passing through the origin, whose projection to $\R^{n-k}$ is contained in the set of leaves of $\mathcal{T}$ of dimension $0$.
    Then, there is a non-empty open subset $U\subset V$ such that every point $p\in U$ is a clean intersection point of $U$ (and hence of $V$) with $\cF$.
\end{proposition}

We postpone the proof of \Cref{prop:cleanregular_local}, and first deduce \Cref{thm:clean} from it.
For readability purposes, we do so first in the case of $\cF$ a \emph{regular} foliation, and afterwards in the case of $\cF$ a general singular foliation.

\begin{proof}[Proof of \Cref{thm:clean} in the case of $\cF$ regular]
    The statement is equivalent to proving that in any open set $V$ of $N$, there is a non-empty open subset of clean intersection points. 
Indeed, if that were the case, let $V_i$ be a countable basis of open sets of $N$. For each $V_i$, there is some non-empty open subset $U_i\subset V_i$ of clean intersection points. Then we have
$$\overline{\bigcup_{i} U_i}=N, \qquad \text{and} \qquad \bigcup_i U_i\subset \Cl(N,\mathcal{F}),$$
and since $\bigcup_i U_i$ is open, this proves the statement.

Let us then prove that in any open subset $V\subset N$ there is a non-empty open subset $U$ made of clean intersection points with the foliation.

\medskip

Up to shrinking $V$, we can assume that $V$ is an embedded submanifold and that $V=N\cap W$ where $W$ is a foliated chart $W\simeq \R^k\times \R^{n-k}$ inside $M$, equipped with coordinates $(x_i, y_j)$ for $i=1,\ldots , k$ and $j=1,\ldots, n-k$, in such a way that the foliation in $W$ is given by leaves $\{y_j=c_j \mid j=1,\ldots,n-k\}$ for some $(c_1,\ldots,c_{n-k})\in \R^{n-k}$. 

Let $e\colon V\hookrightarrow W$ be the inclusion map, which we write in coordinates as
\begin{align*}
    e: V &\longrightarrow W\subset \mathbb{R}^{k}\times \mathbb{R}^{n-k}\\
     p & \longmapsto (g_i(p), f_j(p))
\end{align*}
with $g_i,f_j\colon V\to \R$ are just the compositions of the inclusion of $V$ with the obvious coordinate projections, for all $i=1,\ldots,k$ and $j=1,\ldots,n-k$.

As being a clean intersection point is a local property, this reduces the problem to proving the statement for $M=\mathbb{R}^n$ with leaves of the form $\R^{k}\times \{a\}$ with $a\in \mathbb{R}^{n-k}$, and $V\subset \R^n$. 
At this point,  the conclusion follows directly from \Cref{prop:cleanregular_local}.
\end{proof}

\begin{proof}[Proof of \Cref{thm:clean} in the general case]
As before, since the statement can be localized to arbitrarily small subsets of $N$, we assume without loss of generality that $N$ is embedded. 
Consider the closed subset
$$K:=\{ p\in M \mid \dim T_p\mathcal{F}\text{ is not constant near } p\}= M\setminus M_{\operatorname{reg}},$$
where $M_{\operatorname{reg}}\subset M$ is the open and dense set of regular points of $(M,\cF)$.
We decompose $N$ as 
$$N=N_{\operatorname{reg}}\sqcup N_{\operatorname{sing}},$$ 
where
$$N_{\operatorname{reg}}:= N\cap M_{\operatorname{reg}} \, ,
\quad 
N_{\operatorname{sing}}:=N\cap K=N\setminus N_{\operatorname{reg}}.$$
Notice that even though $M_{\operatorname{reg}}$ is open and dense in $M$, and $N_{\operatorname{reg}}$ is open (since $N$ is embedded), it is not true in general that $N_{\operatorname{reg}}$ is dense in $N$. For example, the latter could even be entirely contained in $K$. However, notice that if $N_{\operatorname{sing}}$ has no interior, then it holds that $N_{\operatorname{reg}}$ is dense, and it is enough to show that there is an open and dense set of clean points in $N_{\operatorname{reg}}$. Otherwise, we also need to show that the interior of $N_{\operatorname{sing}}$ in $N$ admits an open and dense set of clean intersection points.

\medskip

\textbf{Clean intersection points in $N_{\operatorname{reg}}$.} The foliation $\mathcal{F}|_{M_{\operatorname{reg}}}$ is a regular foliation of $M_{\operatorname{reg}}$, since the function $\dim_{T\cF}$ is locally constant on $M_{\operatorname{reg}}$ (by definition). 
Note that $M_{\operatorname{reg}}$ will in general have several connected components.
The subset $N_{\on{reg}}$ is a submanifold of $M_{\on{reg}}$, and thus by the regular case proved above, there is an open and dense subset $U_{\operatorname{reg}}\subset N_{\operatorname{reg}}$ of clean intersection points. 
\medskip

\textbf{Clean intersections points in $N_{\on{sing}}$.} 
We assume here that $N_{\on{sing}}$ is non-empty and that it has interior, since otherwise we are done. We need to prove that there is an open and dense subset of points, inside the interior $U:=(N_{\operatorname{sing}})^{\mathsf{o}}$ of $N_{\operatorname{sing}}$, that is made of clean intersection points. 

Note that $U$ is a codimension zero submanifold of $N$ (understood as the source manifold of the embedding map into $M$).

Define $$k:=\max\{\dim T_p\mathcal{F} \mid p\in U\},$$
and since every set 
$$M_j:=\{p\in M \mid \dim T_p\mathcal{F}\geq j\} \qquad \text{with} \qquad j=1,...,k$$ 
is open by \Cref{lem:rankcont} and $\dim_{T\cF}$ is integer-valued, it follows that
$$U_k= \{p\in U \mid \operatorname{dim}T\mathcal{F}_p=k\}=U\cap M_k\subset U$$
is an open subset of $U$. 
It is also non-empty by the definition of $k$. 
We will now show that there is an open and dense subset $O_k\subset U_k$ made of clean intersection points.

\medskip

In order to prove this, it is enough to show that given any point $p\in U_k$ and an arbitrary neighborhood $V'$ of $p$ in $U_k$, there is an open subset of clean intersection points in $V'$. Without loss of generality, we assume that $V'$ is a small enough neighborhood such that it lies in a local splitting chart of $\mathcal{F}$ as in \Cref{thm:localsplitsingular}.
Namely, we denote by $(x_i, y_j)$ the splitting coordinates in 
$$W\subset \R^k\times \R^{n-k},$$ 
where $i=1,...,k$ and $j=1,...,n-k$ (here, $n$ denotes the dimension of $M$). 
The singular foliation $\mathcal{F}$ in $W$ is moreover given by the product of the trivial foliation in $\R^k$ and a singular foliation $\mathcal{T}$ in $\R^{n-k}$ whose generators vanish at the origin.

Now, as $U_k$ (by definition) lies entirely in the subset of $W$ foliated by $k$-dimensional leaves, which are all given by subspaces of the form $\R^{k}\times \{c\}$, \Cref{prop:cleanregular_local} implies that $U_k$ has an open and dense subset of clean intersection points. We now split the set $U$ as
$$U=U_k \sqcup U_{\leq k-1},$$
where 
$$U_{\leq k-1}=\{p\in U\mid \dim T\cF_p\leq k-1\}.$$
As we argued for $N_{\on{reg}}$ and $N_{\on{sing}}$, if $U_{\leq k-1}$ has no interior, we are done. Otherwise, we iterate the reasoning by considering its interior $({U_{\leq k-1}})^\circ$ and the open subset of points where $\dim T\cF$ is maximal, and so forth. At the end of the process, it might happen that we are left with the set $U_0$, possibly with a non-empty interior. Every point in $U_0$ is trivially a clean intersection point, which concludes the proof.
\end{proof}

\begin{proof}[Proof of \Cref{prop:cleanregular_local}]
As $V$ is contained in the union of all leaves of $\cF$ of dimension $k$, and since the local splitting \Cref{thm:localsplitsingular} allows us to assume that all such leaves are locally given by subspaces of the form $\R^k\times \{*\}$, it is in fact enough to prove the statement for $\cF$ the regular foliation given by leaves of the form $\R^k\times \{*\}$.
We hence work in this (notationally easier) setting.

\medskip 

Let us use coordinates $(x_i, y_j)$ in $\R^n=\mathbb{R}^k \times \mathbb{R}^{n-k}$, with $i=1,...,k$ and $j=1,...,n-k$. 
As we did before, we denote by 
$$f_j: V\longrightarrow \R$$
the composition of the inclusion map of $V$ with the projection onto the $y_j$-coordinate.
Consider now the following property $\mathrm{P} (r)$, with $r\in\{1,\ldots,n-k\}$:
\begin{propertyP}
    There is a non-empty open subset $U_r\subset V$ such that the map
    \begin{align*}
        F_r\colon U_r &\longrightarrow \R^r\\
    p&\longmapsto (f_1(p),\ldots,f_r(p))
    \end{align*}
    has fibers $V_r^p:=(F_r)^{-1}(F_r(p))$ giving a smooth (in $p$) family of submanifolds of 
    \[
    L_p^r:=\R^k\times \{y_j=f_j(p) \mid j=1,\ldots, r\}\times \R^{n-k-r}\subset \R^k\times\R^{n-k}
    \]
    satisfying $V_r^p=U_r\cap L_p^r$ (by definition) and $T(V_r^p) = T U_r \cap T L_p^r$ at all points of $V_r^p$.
\end{propertyP}

\noindent
Notice that \Cref{prop:cleanregular_local} immediately follows from $\mathrm{P}(n-k)$.
It is hence enough to prove the above property, which we do by induction on $r\in\{1,\ldots, n-k\}$.

\medskip

Let us start from $r=1$. 
Consider the map
$f_1: V \longrightarrow \mathbb{R}$.
Two cases arise.

\begin{itemize}
    \item \textbf{Case 1: The map $f_1$ is constant near $p$.} Notice that this can only happen if $\dim V\leq n-1$. 
Consider a neighborhood $V_1$ of $p$ in $C$ for which $f_1$ is constant. 
Then 
$$e_1:=e|_{V_1}: V_1 \longrightarrow \mathbb{R}^{k}\times \{f_1(p)\}\times \R^{n-k-1}=L^1_p$$
is an embedding. 
Simply considering $U_1:=V_1$, we have thus found the desired neighborhood $U_1$ of $p$ inside $V$. 
Indeed, $U_1\cap L^1_p=V_1$ is a submanifold, and since the neighborhood $U_1$ and the submanifold $V_1$ coincide in this first case, the equality $T(V_1)=TU_1\cap TL_p^1$ holds trivially.

\item \textbf{Case 2: The map $f_1$ is not constant near $p$.} We can then find a non-empty open set $U_1\subset V$ which only contains regular points of $f_1$. Up to shrinking $U_1$, 
$$V_1^p:=U_1\cap f_1^{-1}(f_1(p)) \, ,$$
is a smooth (in $p$) family of submanifolds; for a well chosen $U_1$, we can even arrange, for instance, that $U_1$ is a ball and $V_1^p$ is a smooth family of balls of codimension one in $U_1$. 
Notice now that $V_1^p=U_1\cap L_p^1$. 
Since $V$ is transverse to $L_p^1$ for each $p\in U_1$, we have $T(U_1\cap L_p^1)= \ker (\d f_1\vert_{U_1})$. 
Now, since $\ker (\d f_1\vert_{U_1})=TL_p^1 \cap TU_1$, we deduce that $T(U_1\cap L_p^1)=TU_1 \cap T L_p^1$.
In other words, $U_1$ is the desired neighborhood of $p$.

\end{itemize}

To resume, in both cases, the described $U_1$ together with $F_1:=f_1|_{U_1}$ satisfy the conclusion of the property $\mathrm{P}(1)$. 
Let us then pass to the inductive step.

\medskip
By induction, assume that for some $r<n-k$ we have found an open subset $U_r\subset V$ as described by the property $\mathrm{P}(r)$.
Consider then the smooth (in $p$) family of smooth maps
\begin{align*}
    f_{r+1}^p:=f_{r+1}\vert_{V_r^p}: V_r^p &\longrightarrow \mathbb{R}\\
    q &\longmapsto f_{r+1}(q).
\end{align*}
There are again two possibilities. 

\begin{itemize}
    \item \textbf{Case 1: for all $p\in U_r$, $f_{r+1}^p$ is constant.}
    In this case, define $U_{r+1}:=U_r$, and consider the map
    \begin{equation}
        \label{eq:Frplusone}
            F_{r+1}\colon U_{r+1}\to \R^{r+1} \, ,
            \quad
            p\to (f_1(p),\ldots,f_{r+1}(p)) \, .
    \end{equation}
    Then, for all $p\in U_{r+1}=U_r$, we define the sets 
    \begin{equation}
        \label{eqn:Vrplusone}
        V_{r+1}^p:=  (F_{r+1})^{-1}(F_{r+1}(p)) \, .
    \end{equation}
    Now, because
    \begin{equation*}
        (F_{r+1})^{-1}(F_{r+1}(p)) 
         = (F_r)^{-1}(F_r(p)) \cap (f_{r+1})^{-1}(f_{r+1}(p))
         = V_r^p\cap (f_{r+1}^p)^{-1}(f_{r+1}^p(p))
    \end{equation*} 
    and $f_{r+1}^p$ is constant on the $V_r^p$, we get that $V_{r+1}^p$
    simply coincides with $V_{r}^p$ itself.
    Moreover, since $V_r^p=V_{r+1}^p\subset L^{r+1}_p$ and $TL^{r+1}_p\subset TL^{r}_p$, we have $TV_r^p= TU_r\cap TL^r_p = TU_r \cap TL^{r+1}_p$ at all points of $U_r$.
    This gives that $T V_{r+1}^p=TU_{r+1}\cap TL^{r+1}_p$ at all points of $U_{r+1}=U_r$ as well.

    \item \textbf{Case 2: there is some $q$ for which $f_{r+1}^q$ is not constant.} 
    Since the family of maps $(f_{r+1}^p\mid p \in U_r)$ is smooth in $p$, we can find a non-empty open neighborhood $U_{r+1}\subset U_r$ of the point $q$ such that for any $p\in U_{r+1}$, the set $W_p:=U_{r+1} \cap V_r^p$ is an open set in $V_r^p$ made of regular points for $f_{r+1}^p$. 
    We can then define the map $F_{r+1}$ as in Equation \eqref{eq:Frplusone}, and, for any $p\in U_{r+1}$, the set $V_{r+1}^p$ as in \eqref{eqn:Vrplusone},
which in fact is just equal to $U_{r+1}\cap L^{r+1}_p$. 
Notice that this set is also equal to
$$j_p \Big(\big(f_{r+1}^p\big)^{-1}\big(f_{r+1}^p(p)\big)\cap W_p\Big) \, , $$
where $j_p$ is the inclusion of $W_p$ inside $U_{r+1}$.
Since all points in $W_p$ are regular for $f_{r+1}^p$, the family $V_{r+1}^p$ is a smooth family of submanifolds.

In addition, we know that $W_p$ is transverse to $R_s:=\{(x_i,y_i) \mid y_{r+1}=f_{r+1}(s)\}$ for any $s\in W_p$. 
This implies that $TV_{r+1}^p= TW_p \cap TR_p$.
Moreover, as $L^r_p\cap R_p = L^{r+1}_p $ we also have $T(L^r_p \cap R_p)=TL^{r+1}_p$.
Lastly, recalling that $W_p=U_{r+1}\cap V_r^p$ with $U_{r+1}$ being an open subset of $U_r$, and $TV_{r}^p=TU_r\cap TL^r_p$ by the induction hypothesis, we get
$$TW_p=TU_r\cap TL^r_p|_{W_p}=TU_{r+1}\cap TL^r_p \, . $$
All together, this implies 
\begin{align*}
    TV_{r+1}^p &= TW_p\cap TR_p\\
    &=TU_{r+1}\cap TL^{r}_p \cap TR_p\\
    &=TU_{r+1}\cap TL^{r+1}_p
\end{align*} 
as we wanted to show.

\end{itemize}
In other words, in both cases the pair $(U_{r+1},F_{r+1})$ satisfies $\mathrm{P}(r+1)$ as desired. 
\end{proof}

\medskip

We point out that, as the following examples show, the conclusion of \Cref{thm:clean} is sharp: the set $\Cl(N,\cF)$ is in general not open. 

\begin{example}[Clean points are not open in singular foliated spaces]
\label{exa:clean_not_open_singular}
    Consider $\R^3$ equipped with coordinates $(x,y,z)$ and with the Poisson structure $\Pi=(x^2+y^2)\pp{}{x}\wedge \pp{}{y}$. 
    Then, the origin is a clean intersection point of $N=\{z=x^2\}$ with the singular symplectic foliation of $\Pi$, while the $y$-axis besides the origin are all made of non-clean intersection points.
\end{example}
In the regular case of corank one, it can be proven that clean intersection points form an open set.
As the following example shows, already in codimension two, this does not hold anymore.
\begin{example}[Clean points are not open in regular foliated spaces]
    \label{exa:clean_not_open}
    Consider on $\R^3$ the foliation $\cF$ with leaves given by lines parallel to the $y$-axis, and $N=\{z=0\}\subset \R^3$. 
    The main idea is that one can perturb $N$, by ``pushing it" in the $z$-direction at a sequence of its points $(p_n)_{n\in \bN}$ accumulating to the origin, to a new hypersurface $N'$ for which the origin is a clean intersection point while each $p_n$ is not.
    An explicit perturbation achieving this property is, for instance, the following.

    Consider an auxiliary smooth function $f\colon \R\to \R_{\geq 0}$ that is supported in $[-1,1]$, is equal to $1$ at $0$, and is strictly smaller than $1$ on $[-1,1]\setminus \{0\}$.
    Let now $p_n=(1/n,0,0)\subset N$, and define $N'=\{(x,y,z)\in \R^3 \mid z=F(x,y) \}$
    where 
    \[
    F\colon \R^2\to \R \, , \quad (x,y)\mapsto \sum_{n\in \bN}e^{-n^2} f(e^n \, ((x-1/n)^2+y^2)) \, .
    \]
    Note that $F$ is a smooth function by the Weierstrass M-test of convergence for series of functions. It is now immediate to check that, because $f'(0)=0$, $N'$ has as a sequence of non-clean intersection points $p_n$, accumulating on the origin, which is, on the contrary, a clean intersection point, as the whole $y$-axis lies in $N'$.
\end{example}

%%%%%%%%%%%%%%%%%%%%%%%%%%%%%%%%%%%%%%%%%%%%%%%%%%%%%%%%%%%%%%%%%%%%%%%%%%%%%%%%%%%%%%%%%%%%%%%%%%%%%%%%%%%%%%%%%

\section{Rigidity and non-rigidity of coisotropics}\label{sec:rig_nonrig_coisotropics}
In this section, we use the previously developed theory to analyze the behavior of coisotropic submanifolds and their characteristic foliations under Poisson homeomorphisms.

\subsection{Rigidity of coisotropic submanifolds}
We are now ready to establish our main result concerning the $C^0$-rigidity properties of coisotropic submanifolds, by combining our study of clean intersection points in Section \ref{sec:cleanpoints} and symplectic rigidity statements \cite{HLS15}. 
To specify those leaves for which rigidity can be established, we need to introduce the following set. 
\begin{definition}
    Let $C$ be a submanifold in a foliated manifold $(M,\cF)$ and $\varphi$ a homeomorphism such that $C':=\varphi(C)$ is smooth. A point $p\in C$ is a $\varphi$-clean point if $p\in \Cl(C,\cF)$ and $\varphi(p)\in \Cl(C',\cF)$,
\end{definition}
For the set of $\varphi$-clean points, we will use the notation
$$\on{Cl}(C,\varphi):=\{p\in C\mid p\in \on{Cl}(C,\mathcal{F}) \text{ and } \varphi(p)\in \on{Cl}(C',\cF)\} \subset C.$$
In the case where $\varphi$ is a Poisson homeomorphism of $(M,\Pi)$, the foliation $\cF$ is the singular symplectic foliation of $(M,\Pi)$. 
It follows from Theorem \ref{thm:clean} that $\on{Cl}(C,\varphi)$ contains an open and dense set of points of $C$. In the following theorem, Item 1 below implies Theorem \ref{thm:main2} stated in the introduction, and Item 2 gives a sufficient condition for a characteristic leaf to be mapped homeomorphically to a characteristic leaf.
\begin{theorem}\label{thm:C0-rigidity_coisotropics}
    Let 
$$\varphi : (M,\Pi) \longrightarrow (M,\Pi)$$
be a Poisson homeomorphism, $C$ a coisotropic submanifold, and assume that $C'=\varphi(C)$ is a smooth submanifold.  Then the following hold:
\begin{enumerate}
    \item $C'$ is a coisotropic submanifold, 
    \item if a characteristic leaf of $C$ is completely contained in $\Cl(C,\varphi)$, then it is mapped homeomorphically to a characteristic foliation of $C'$.
\end{enumerate}
\end{theorem}

\begin{remark}\label{rem:Nonrigiditycharacteristic}
     It is not true in general that all leaves of $\cK_C$ are mapped homeomorphically to leaves of $\cK_{C'}$, even if $\varphi$ is required to be a Hameomorphism.
     We will see examples in \Cref{sec:examples_nonrigid_char_fol}.
\end{remark}

\begin{proof}
\textbf{Item 1.} Let $p\in \Cl(C,\varphi)$, and $(L_p,\omega_{L_p})$ the symplectic leaf passing through $p$ of the singular symplectic foliation associated to $\Pi$. 
Consider also small enough neighborhoods $U$ of $p$ in $C$ and $V$ of $p$ in $L_p$ such that $U\cap V$ is a smooth submanifold of $L_p$.
Then, as $p$ is in particular a clean intersection point, according to \Cref{prop:cleancoisotropicinleaf} we know that $U\cap V$ is a coisotropic submanifold (in the symplectic sense) of $(L_p, \omega_{L_p})$. 
What is more, by \Cref{rmk:openleafclean} we can moreover choose $U$ and $V$ such that $U\cap V\subset \Cl(C,\cF)$.

Note also that, since $q=\varphi(p)\in \Cl(C',\cF)$ by choice of $p\in \Cl(C,\varphi)$, we can also choose $U$ and $V$ small enough so that $\varphi(U\cap V)$ is made only of clean intersection points of $C'$ with $\cF$, and thus $U\cap V\subset \Cl(C,\varphi)$. 

Now, by \Cref{thm:poisson_homeos_are_leafwise_sympl_homeos}, we know that $\varphi$ maps $L_p$ to $L_{\varphi(p)}$. 
Thus, as $q$ is a clean intersection point, denoting  by $U'$ the neighborhood $\varphi(U)$ of $q=\varphi(p)$, for $U$ and $V$ small enough we must have that $V':=\varphi(V)$ is a neighborhood of $q$ in $L_q$ and it satisfies that $U'\cap V'$ is a smooth submanifold and
$$T(U'\cap V')= TU'\cap TV'=(TC'\cap TL_q)|_{U'\cap V'} \; .$$
Moreover, \Cref{thm:poisson_homeos_are_leafwise_sympl_homeos} also tells us that 
$$\phi:=\varphi|_{L_p}:(L_p,\omega_{L_p})\longrightarrow (L_q,\omega_{L_q})$$ 
is a local symplectic homeomorphism. 
Since $U\cap V$ is a coisotropic submanifold of $(L_p, \omega_{L_p})$ and $\phi(U\cap V)=U'\cap V'$ is a smooth submanifold of $L_q$, \cite[Theorem 1]{HLS15} and \cite[Remark 10]{HLS15} then imply that $U'\cap V'$ is a coisotropic submanifold of $(L_q,\omega_{L_q})$ and that the characteristic leaves are mapped to characteristic leaves homeomorphically via $\varphi$. 
By \Cref{prop:cleancoisotropicinleaf}, we conclude that $C'$ is in fact coisotropic at all points in $\varphi(\Cl(C,\varphi))\subset C'$. 

Now, by \Cref{thm:clean}, there are two subsets 
$$W_1\subseteq \on{Cl}(C,\cF)\subseteq C, \qquad \text{and} \qquad W_2\subseteq \on{Cl}(C',\cF)\subset C',$$
that are open and dense in $C$ and $C'$ respectively, where $\cF$ denotes the symplectic foliation of $(M,\Pi)$. Consider the subset
$$W=W_1\cap \varphi^{-1}(W_2)\subseteq \on{Cl}(C',\cF),$$
which is also open and dense in $C$. 
Then $C'$ is coisotropic at every point in $\varphi(W)$, which is an open and dense set in $C'$. Lastly, notice that the set of points at which a submanifold is coisotropic is closed, which implies that $C'$ is coisotropic. \\

\textbf{Item 2.} The previous argument also shows that for every point $p\in \Cl(C,\varphi)$, the characteristic leaf $K_p$ is sent (locally near $p$) homeomorphically to the characteristic leaf through $\varphi(p)$. 
Thus, if the whole leaf is included in $\Cl(C,\varphi)$, the leaf is mapped homeomorphically to a characteristic leaf, which gives Item 2 as well.
\end{proof} 

%%%%%%%%%%%%%%%%%%%%%%%%%%%%%%%%%%%%%%%%%%%%%%%%%%%%%%%%%%%%%%%%%%%%%%%%%%%%%%%%%%%%%%%%%%%%%%%%%%%%%%%%%%%%%%%%%

\subsection{Examples of non-rigidity of characteristic foliation}
\label{sec:examples_nonrigid_char_fol}

We end this section by giving some examples of Poisson homeomorphisms on trivial regular Poisson manifolds for which there is no complete rigidity of the characteristic foliation of a coisotropic submanifold. These examples use hypersurfaces as coisotropic submanifolds, but can be easily generalized to coisotropic submanifolds of higher codimension. 

\begin{example}[A Hameomorphism mapping coisotropic submanifolds with non-homeomorphic characteristic foliations]\label{ex:cleantononclean}
    Consider in $\mathbb{R}^{2n+k}$, with coordinates 
    $$(x_1,x_2,...,x_n,y_n,z_1,...,z_k),$$ and regular Poisson structure 
    $$\Pi=\sum_{i=1}^n\pp{}{x_i}\wedge \pp{}{y_i},$$ 
    the coisotropic submanifold $C=\{(x_i,y_i,z_j) \mid z_1=x_1\}$ and the coisotropic submanifold $C'=\{(x_i, y_i, z_j) \mid z_1=x_1^3\}$. Let $h_n\colon \R\to\R$ be a family of diffeomorphisms of the real line such that 
    $$h_n \xrightarrow{C^0} h \,, $$
    where $h(s)=\sqrt[3]{s}$ for all $s\in \R$.
    Consider the Hamiltonian functions
    $$H_n=(z_1-h_n(z_1))y_1$$
    whose Hamiltonian vector field is
    $$X_{H_n}= \left(h_n(z_1)-z_1\right)\pp{}{x_1}.$$
    Its flow is given by
    $$\phi_{H_n}^t(x_i,y_i,z_j)=\big(x_1+t \left(h_n(z_1)-z_1\right), y_1, x_2,y_2,...,x_n,y_n, z_1,...,z_k\big),$$
    and thus
    $$\phi^1_{H_n}(C)=\{z_1=h_n^{-1}(x_1)\}.$$
    Taking the limit as $n$ goes to infinity, we get $$H_n \xrightarrow[C^0]{} H \, ,$$
    where $H\colon\R\to\R$ is given by $H=(z_1-\sqrt[3]{z_1})y_1$, 
    and $\phi_{H_n}^t$ $C^0$-converges to
    $$\phi^t_H(x_i,y_i,z_j):= \left(x_1+t(\sqrt[3]{z_1}-z_1), y_1, x_2,y_2,...,x_n,y_n,z_1,...,z_k\right).$$
    This function satisfies $\phi_H^1(C)=C'$, and thus sends a coisotropic submanifold all of whose points are of clean intersection with the leaves, to one that has non-clean intersection points. 
    Indeed, given a point $p=(\hat x_i,\hat y_i,0, \hat z_2,...,\hat z_k)\in C'$ and any sufficiently small neighborhood of $p$ in $C'$, the intersection $U\cap L_p$, where $L_p$ is the symplectic leaf of $\Pi$ containing $p$, is a smooth submanifold satisfying
    $$T(U\cap L_p)=\left\langle \pp{}{y_1},\pp{}{x_2}, \pp{}{y_2},...,\pp{}{x_n}\pp{}{y_b}\right\rangle,$$
    which is strictly contained in 
    $$TU\cap TL_p= \left\langle \pp{}{x_1},\pp{}{y_1},...,\pp{}{x_n}\pp{}{y_n}\right\rangle.$$
    Notice that at $p$ we have $TC'|_p=TL_p$, and thus $\Pi^\sharp\big((TC')^\circ\big)=\{0\}$, namely, the characteristic leaves of $C'$ along $z_1=0$ are points. 
    However, the characteristic leaf of $C$ along $p=(\hat x_i, \hat y_j,0, \hat z_2, ..., \hat z_k)$ is the line 
    $$K_p=\{(\hat x_1, \hat y_1+t,\hat x_2,\hat y_2,...,\hat x_n,\hat y_n,0, \hat z_2, ..., \hat z_k)\mid t\in \R\},$$
    which is mapped by $\varphi$ to a collection of $0$-dimensional leaves of $\cK_{C'}$.
\end{example}

\medskip

In the previous example, the homeomorphism type of $(C,\cK_C)$ is different from that of $(C',\cK_{C'})$. 
With a similar argument, one can also construct, as we detail next, an example of a Poisson homeomorphism preserving a smooth coisotropic submanifold set-wise, but \emph{not} mapping characteristic leaves to homeomorphic ones.

\begin{example}[A Poisson homeomorphism not mapping all characteristic leaves homeomorphically]
    \label{exa:char_fol_not_C0-rigid}
    Consider $\mathbb{R}^{2n+k}$ with the same coordinates and Poisson structure as in the previous example. Let $C$ be the coisotropic submanifold $\{(x_i,y_i,z_j)\mid z_1=x_1^3\}$. Consider the homeomorphism
    \begin{align*}
    \phi\colon \R^{2n+k} &\longrightarrow \R^{2n+k} \\
    (x_i,y_i,z_j)&\longmapsto (x_1+1,y_1,x_2,y_2,...,x_n,y_n, (\sqrt[3]{z_1}+1)^3,
    z_2,...,z_k)        
    \end{align*}
    It is a Poisson homeomorphism, indeed if $h_n\colon \R\to \R$ is a sequence of diffeomorphisms so that $h_n\xrightarrow{C^0} h$ with  $h(s)=\sqrt[3]{s}$, then 
    \begin{align*}
    \psi_n\colon \R^{2n+k} &\longrightarrow \R^{2n+k} \\
    (x_i,y_i,z_j) &\longmapsto (x_1+1,y_1,x_2,y_2,...x_n,y_n, h_n^{-1}(h_n(z_1)+1),z_2,...,z_k)        
    \end{align*}
    is a sequence of Poisson diffeomorphisms that $C^0$-converges uniformly on compact sets to $\phi$. 
    Moreover, the map $\phi$ preserves $C$, but, it sends $C\cap \{z=0\}$ to $C\cap \{z=1\}$, and hence sends $0$-dimensional leaves of the characteristic foliation of $C$ to $1$-dimensional ones.
\end{example}

\medskip

We point out that the above Poisson homeomorphism is not a Hameomorphism, as it does not map each symplectic leaf of the Poisson structure to itself, which Hameomorphisms always do according to \Cref{cor:Hameos_preserve_leaves}.
As we will see in \Cref{prop:nonliftregular}, there also exist examples of Hameomorphisms preserving a coisotropic submanifold $C$ for which some leaf of $\cK_C$ is not mapped homeomorphically to another leaf.

\subsection{Non-clean paths and almost rigidity}
\label{sec:when_almost_all_char_leaves_mapped_homeo}

A natural question that stems from the rigidity of coisotropics and some of the examples we have seen is whether ``most" characteristic leaves are mapped homeomorphically to characteristic leaves. For instance, one could ask the following question. As before, we suppose that $\varphi$ is a Poisson homeomorphism mapping a coisotropic submanifold $C$ to a smooth (and hence coisotropic) submanifold $C'$.

\begin{question}[\Cref{q:dense} in the introduction]\label{q:dense2}
Is there always a dense set of characteristic leaves of $C$ that are mapped homeomorphically to characteristic leaves of $C'$? If $C$ has a compact closure, is there always a dense and open set of characteristic leaves mapped homeomorphically?
\end{question}
A few comments are in order. First, as it is customary in the regular foliation case, the topology that we consider in the leaf space is the quotient one: namely, the leaf space is here simply seen as the quotient (with induced topology) of the ambient manifold $M$ with respect to the relation $x\sim y$ if $x$ and $y$ are on the same leaf of the singular foliation $\cF$. Note that, as in the case of regular foliations already, the leaf space is somewhat ``pathological'': for instance, it is very often not Hausdorff.

Secondly, we insist on the necessity of the fact that $C$ has compact closure. Indeed, if $C$ is allowed to have a non-compact closure (for instance, if $M$ is not compact itself), there exist examples where there is only a dense set of leaves (and not an open and dense set) that is mapped homeomorphically. 
Density could still hold in the general case, though. 

Thirdly, the set $\Cl(C,\varphi)$ is the relevant set that needs to be understood. Indeed, the question is equivalent to asking whether there is a dense, or an open and dense set of leaves that is contained in $\Cl(C,\varphi)$. Indeed, Theorem \ref{thm:C0-rigidity_coisotropics} shows that leaves in $\Cl(C,\varphi)$ are mapped homeomorphically, so having an open and dense (or only dense) set of leaves contained in $\Cl(C,\varphi)$ imply the desired statement. Conversely, the fact that $\Cl(C,\varphi)$ contains an open and dense set and Lemma \ref{lem:cleanleafcoiso} together show that if there is an open and dense set (or only a dense set) of leaves that are mapped homeomorphically, then there is an open and dense set (or dense, respectively) of leaves contained in $\Cl(C,\varphi)$. 
\medskip

Although we do not give a definitive answer to the question, we give here a sufficient condition on the characteristic foliation of the involved coisotropics, which involves the notion of ``non-clean path". (This notion will be useful again later in Section \ref{sec:C0coisotropics}.)

\begin{definition}\label{def:noncleanpath}
    Let $C$ be a coisotropic submanifold of $(M,\Pi)$. 
    Let $K$ be a clean characteristic leaf of $\mathcal{K}_C$ and $L$ the symplectic leaf such that $K\subset L$. 
    A \textbf{non-clean path} of $K$ is a continuous path $\gamma:[0,1]\longrightarrow L$ such that $\gamma([0,1))\subset K$ and $\gamma(1)$ is not a clean intersection point with $\cF$.
\end{definition}
First, we give more information about the endpoint of a non-clean path.
\begin{lemma}
    Let $K$ be a clean characteristic leaf of $C$ in a symplectic leaf $L$. Suppose that $\gamma:[0,1]\rightarrow L$ is a non-clean path. Then $\gamma(1)$ belongs to a different characteristic leaf $\tilde K$, made of non-clean intersection points, and satisfying $\dim \tilde K<\dim K$.
\end{lemma}
\begin{proof}
    The fact that $\gamma(1)$ belongs to a different leaf directly follows from the fact that $K$ is made of clean intersection points. Since every leaf is either clean or non-clean, the endpoint belongs to a leaf $\tilde K$ made of non-clean intersection points. We are left with proving that $\dim \tilde K<\dim K$. But this easily follows from the following two facts. First, we know that the set of leaves of dimension at most $\dim K$ is closed, which implies that $\dim \tilde K\leq \dim K$. By contradiction, suppose that $\dim \tilde K=\dim K$. On a splitting chart $W$ near $\gamma(1)$, each plaque of dimension $\dim K$ is locally closed. Hence, since $\gamma(1-\delta)$ belongs to exactly one of these plaques for $\delta>0$ small enough, and $\gamma(1-\delta)\in K$, we must have $\tilde K=K$, reaching a contradiction.
\end{proof}

Suppose that there is an example of a Poisson homeomorphism that does not map a dense set of characteristic leaves homeomorphically. 
Then there must be an open set of leaves that are not mapped homeomorphically to characteristic leaves. 
We can then assume that this open set of leaves satisfies that each leaf is clean, and each one somewhere intersects $\Cl(C,\varphi)$, since the set of clean leaves is open and dense, and so is the set $\Cl(C,\varphi)$. The next proposition shows why non-clean paths are relevant in understanding whether this situation can happen, and thus in understanding whether the answer to \Cref{q:dense} is negative. 

\begin{proposition}\label{prop:badleaves_have_nonclean_paths}
    Let $\varphi:(M,\Pi)\longrightarrow (M,\Pi)$ be a Poisson homeomorphism mapping a coisotropic submanifold $C$ to a smooth (and hence coisotropic) submanifold $C'$. Let $K$ be a clean characteristic leaf of $C$ intersecting $\Cl(C,\varphi)$. Then if $\varphi$ does not map $K$ homeomorphically to a characteristic leaf of $C'$, for any point $p\in K\cap \Cl(C,\varphi)$, the characteristic leaf $K_{\varphi(p)}'$ of $C'$ passing through $\varphi(p)\in C'$ admits a non-clean path.
\end{proposition}

\begin{proof}
    Recall that this set can be described as
$$\Cl(C,\varphi)= \Cl(C,\cF) \cap \varphi^{-1}(\Cl(C',\cF)).$$
Choose $p\in K\cap \Cl(C,\varphi)$. Recall that by Lemma \ref{lem:cleanleafcoiso}, if a point $p$ is in $\Cl(C,\cF)$, then so are all points in its characteristic leaf $K$. In adition, by Theorem \ref{thm:C0-rigidity_coisotropics}, we know that $K\not \subset \Cl(C,\varphi)$. We then must have
$$K_p \not \subseteq \varphi^{-1}(\Cl(C',\cF)).$$
On the other hand, we know that $\varphi(p)\in \Cl(C',\cF)$, and thus $K_{\varphi(p)}'$ is a leaf made of clean intersection points of $C'$ with $\cF$ by Lemma \ref{lem:cleanleafcoiso}. Since $\varphi(K_p)$ is not contained in $\Cl(C',\cF)$, there are non-clean points in $\varphi(K_p)$. This necessarily implies that $\varphi(K_p)$ intersects more than a single leaf, since $K_{\varphi(p)}'$ is a clean leaf. Consider then a path 
$$\gamma: [0,1] \longrightarrow K_p$$
satisfying
\begin{itemize}
    \item[-] $\gamma([0,1])\subset K_p$,
    \item[-] $\varphi\circ \gamma([0,1))\subset K_{\varphi(p)}'$,
    \item[-] $\varphi\circ\gamma(1)\not \in K_{\varphi(p)}'$.
\end{itemize}
To conclude the proof, it is hence enough to prove the following claim, which we do right below.
\begin{claim}\label{claim:noncleanpath}
    The path $\gamma$ is a non-clean path. \qedhere
\end{claim}
\end{proof}

\begin{proof}[Proof of Claim \ref{claim:noncleanpath}]
First, recall that by Theorem \ref{thm:main1}, the homeomorphism $\varphi$ preserves the symplectic foliation. Since every characteristic leaf is contained in a symplectic leaf, we have $K_p\subset L_p$ and thus 
$$\varphi\circ \gamma([0,1])\subset \varphi(L_p)=L_{\varphi(p)}.$$
First, observe that $q=\varphi\circ \gamma(1)$ belongs to a leaf that has a dimension that is lower than $k=\dim K_{\varphi(p)}'$. Indeed, suppose it is not the case. Since the set of leaves of dimension $\leq k$ is closed, we must have that the dimension of the characteristic leaf of $q=\varphi\circ \gamma(1)$ is $k$ too. Take a splitting chart of $\cK_{C'}$ near $q$, and observe that on the splitting chart, the plaques of dimension $k$ are closed. Hence, the point $q$ must belong to the same plaque as $\varphi\circ\gamma(1-\delta)$ for sufficiently small $\delta$. This contradicts the fact that $K_q'\neq K_{\varphi(p)}'$.
\medskip

We have thus shown that $q$ belongs to a leaf of dimension strictly smaller than $k$. Suppose by contradiction that $q$ is a clean intersection point of $C'$ with $\cF$. There are neighborhoods $U$ of $q$ in $C'$ and $V$ of $q$ in $L_q$ such that $U\cap V$ is a smooth submanifold, and $T(U\cap V)=TU\cap TV$. In addition, by Corollary \ref{cor:poisson_char_fol_equals_sympl_one_at_clean_intersect}, we know that the characteristic leaves of $C'$ at points in $U\cap V$ correspond to the symplectic characteristic leaves of $U\cap V$, which is a coisotropic submanifold of $L_q$ with the induced symplectic structure by Proposition \ref{prop:cleancoisotropicinleaf}. Notice that for $\delta$ small enough, the points $\varphi\circ\gamma(1-\delta)$ belong to $U\cap V$. But the characteristic leaves of coisotropic submanifolds in symplectic manifolds have a constant dimension, which contradicts that $\dim K'_q\lneq \dim K'_{\varphi\circ \gamma(1-\delta)}=\dim K'_{\varphi(p)}=k$.
\end{proof}

\medskip

We are now ready to deduce our condition on the characteristic foliation of $C'$ (or of $C$ itself, as one can consider the inverse homeomorphism), to conclude that almost all leaves are mapped homeomorphically.
\begin{theorem}\label{thm:almostrigidity}
    Suppose that $C$ and $C'$ are coisotropic submanifolds.
    \begin{enumerate}
        \item Suppose that for every open set $U\subset C'$, there is a point $p\in U$ such that $K_p$ admits no non-clean path. Then any Poisson homeomorphism such that $\varphi(C)=C'$ maps a dense set of characteristic leaves homeomorphically.
        \item Suppose that for every open set $U\subset C'$ there is an open set $V\subset U$ of points such that for all $p\in V$ the leaf $K_p$ admits no non-clean path. Then any Poisson homeomorphism such that $\varphi(C)=C'$ maps an open and dense set of characteristic leaves homeomorphically. 
    \end{enumerate} 
\end{theorem}
\begin{proof}
As explained before Proposition \ref{prop:badleaves_have_nonclean_paths}, if the set of leaves mapped homeomorphically is not dense, we would be able to find an open set $U\subset \Cl(C,\cF)$ such that every leaf throught $U$ is clean, but every leaf is not mapped homeomorphically. Then by Proposition \ref{prop:badleaves_have_nonclean_paths}, the open set $\varphi(U)$ would be an open set such that for every point $q\in \varphi(U)$, the leaf $K_q'$ admits a non-clean path, which contradicts our assumption.
\medskip

Similarly, under the conditions of the second item, we know already that the set of leaves that are mapped homeomorphically is dense. Suppose that it does not contain an open and dense set. Then there must be an open set $U\subset C$ where the set of leaves that are mapped homeomorphically has no interior, or equivalently, that the set of leaves not mapped homeomorphically are dense in $U$. Without loss of generality, we can assume that $U$ is contained in $\Cl(C,\varphi)$ and contains only clean leaves. By Proposition \ref{prop:badleaves_have_nonclean_paths}, the set $\varphi(U)$ is an open set in $C'$ for which a dense set of points belong to leaves with non-clean paths. This contradicts our assumption.
\end{proof}

Theorem \ref{thm:almostrigidity} sparks a question whose negative answer would imply a positive answer to \Cref{q:dense}.

\begin{question}\label{q:cleanpaths}
    Let $C$ be a coisotropic submanifold in $(M,\Pi)$. Can it happen that there is an open set $U\subset C$ such that every leaf through $U$ admits a non-clean path? If $C$ has a compact closure, can it happen that there is an open set $U$ where the set of points that belong to leaves with non-clean paths is dense in $U$?
\end{question}

In some very simple cases, like regular Poisson manifolds given by symplectic fiber bundles, one can prove that the answer to \Cref{q:cleanpaths} is indeed negative (for instance with arguments as in Proposition \ref{prop:always_one_clean_leaf_regular}), and as a consequence that the answer to \Cref{q:dense} is positive.
However, in the case of general Poisson manifolds (or even just for general regular Poisson manifolds), it is not clear to the authors how to reach the same conclusion.

%%%%%%%%%%%%%%%%%%%%%%%%%%%%%%%%%%%%%%%%%%%%%%%%%%%%%%%%%%%%%%%%%%%%%%%%%%%%%%%%%%%%%%%%%%%%%%%%%%%%%%%%%%%%%%%%%
%%%%%%%%%%%%%%%%%%%%%%%%%%%%%%%%%%%%%%%%%%%%%%%%%%%%%%%%%%%%%%%%%%%%%%%%%%%%%%%%%%%%%%%%%%%%%%%%%%%%%%%%%%%%%%%%%

{

\section{$C^0$-notions related to coisotropics}\label{sec:C0coisotropics}

We have seen that in general one cannot expect Poisson homeomorphisms (or even Hameomorphism, as we will see in Section \ref{sec:non-liftable_homeos}), to behave nicely with respect to the characteristic foliation of smooth coisotropics at every point. This phenomenon is always satisfied in the symplectic case \cite[Theorem 3]{HLS15}, and is a key step to establish the rigidity of characteristic foliations that we know does not hold in the Poisson setting. Motivated by this, we will characterize $C^0$-Hameotopies preserving a coisotropic submanifold and introduce the (rigid) notion of $C^0$-characteristic partition of a smooth coisotropic. Lastly, we introduce $C^0$-coisotropic submanifolds and show that, surprisingly, these non-smooth objects define Lie subalgebras of the Poisson algebra in $C^\infty(M)$.

%%%%%%%%%%%%%%%%%%%%%%%%%%%%%%%%%%%%%%%%%%%%%%%%%%%%%%%%%%%%%%%%%%%%%%%%%%%%%%%%%%%%%%%%%%%%%%%%%%%%%%%%%%%%%%%%%

\subsection{The $C^0$-characteristic partition}

We first introduce $C^0$-rigidity for the characteristic foliation of a coisotropic submanifold.
\begin{definition}
\label{def:C0-rigid_char_fol}
    The characteristic foliation of a coisotropic submanifold $C$ is said to be \textbf{$C^0$-rigid} if every Hameomorphism generated by a function $H$ such that $H|_C$ only depends on time maps every characteristic leaf to itself.
\end{definition}
It is a simple consequence of the arguments in \cite[Proof of Theorem 1]{HLS15} and the fact that in the Poisson context characteristic leaves are spanned by autonomous Hamiltonians that vanish along $C$ (see the proof of Lemma \ref{lem:cleanleafcoiso}) that if $C$ has a $C^0$-rigid characteristic foliation, then any Poisson homeomorphism mapping $C$ to itself must map every characteristic leaf to a characteristic leaf. As we will see later in \Cref{prop:nonliftregular}, there are examples of coisotropic submanifolds whose characteristic foliation is not $C^0$-rigid. A simple rigid example is that of the coisotropic submanifold $C=\{(x,y,z)\mid z=x\}$ in the standard corank $1$ Poisson structure in $\R^3$.

Hence, sometimes, $C^0$-Hamiltonians vanishing along $C$ allow us to connect points in different characteristic leaves of $C$. More generally, one could hope that, like in the symplectic case \cite[Theorem 3]{HLS15}, $C^0$-Hamiltonians vanishing along $C$ generate flows that always leave $C$ invariant. That would lead naturally to considering the equivalence relation given by points that can be connected by $C^0$-Hamiltonians vanishing along $C$, which would then be a (new) topological invariant of coisotropic submanifolds. We aim to establish in this section the following result, which does not follow directly from our main rigidity statement, Theorem \ref{thm:C0-rigidity_coisotropics}.

\begin{theorem}\label{thm:C0_Hamiltonians_vanishing_C}
Let $C$ be a closed coisotropic submanifold in $(M,\Pi)$, and $H$ be a (possibly time-dependent) $C^0$-Hamiltonian. 
Then, the following properties are equivalent.
\begin{enumerate}
    \item\label{item:H_function_time} For every symplectic leaf $L$, the function $H$ restricted to $C\cap L$ is (locally) a function of time.
    \item\label{item:flow_preserves_C} The Hameotopy $\phi^t_H$ generated by $H$ maps $C$ to itself. 
\end{enumerate}
\end{theorem}

\begin{remark}\label{rem:Ham_vanish_implies_rigidity}
Strictly speaking, Theorem \ref{thm:C0_Hamiltonians_vanishing_C} can be seen as a strenghthening of Theorem \ref{thm:main2}. Indeed, a submanifold in a Poisson manifold is coisotropic if and only if for every smooth (autonomous) Hamiltonian such that $H|_C=0$, the flow of $H$ preserves $C$. By mimicking the proof of \cite[Theorem 1]{HLS15}, this characterization of coisotropic submanifolds and Theorem \ref{thm:C0_Hamiltonians_vanishing_C} imply Theorem \ref{thm:main2}. We keep the derivation Theorem \ref{thm:main2} (and more generally, of Theorem \ref{thm:C0-rigidity_coisotropics}) in a previous section as it can be deduced directly from Theorem \ref{thm:clean} and Proposition \ref{prop:cleancoisotropicinleaf} using the $C^0$-rigidity of symplectic coisotropics from \cite{HLS15}, while the proof of Theorem \ref{thm:C0_Hamiltonians_vanishing_C} requires additional more arguments.  
\end{remark}

An immediate consequence, choosing $C=M$, is the following extension to the Poisson setup of \cite[Corollary 4]{HLS15}.
\begin{corollary}
    \label{cor:C0-Hamiltonians_up_to_Casimirs}
    A $C^0$-Hamiltonian $H$ is a continuous Casimir function (i.e.\ a function constant on the symplectic leaves of $\Pi$) if and only if the Hameotopy $\phi_H^t$ it generates is the identity.
\end{corollary}

As we will see, there is one implication, the one showing that item $1$ implies item $2$, that does not follow directly from a combination of \ref{thm:clean} and the symplectic result. Indeed, if one considers the set of clean intersection points $\Cl(C,\cF)$, which we know to contain an open and dense set, it satisfies thanks to \cite[Lemma 21]{HLS15} that any point $p\in \Cl(C,\cF)$ will stay in $C$ when flowing through $\phi^t_H$ for some short time. However, this time depends on the point and might not admit a uniform positive lower bound, ultimately preventing us from deducing the statement above. Instead, a better understanding of the notion of non-clean path (\Cref{def:noncleanpath}) will be needed. We first prove the theorem assuming the needed property of non-clean paths established in \Cref{prop:always_one_clean_leaf_regular} below.

\begin{proof}[Proof of Theorem \ref{thm:C0_Hamiltonians_vanishing_C}]
\textbf{\Cref{item:flow_preserves_C} implies \Cref{item:H_function_time}.} For any given time $t_0$, we know by assumption that $\phi^{t_0}_H(C)=C$. As in the proof of Theorem \ref{thm:C0-rigidity_coisotropics}, consider a subset $W_{t_0} \subset \Cl(C, \phi^{t_0}_H(C))$ that is open and dense in $C$. For any $p\in W_{t_0}$, there exists two neighborhoods $U_p\subset W_{t_0}$ and $V'_p\subset \Cl(C,\cF)$ of $p$ and $\phi^{t_0}_H(p)$ respectively, and a $\delta_p>0$ such that
$$ \phi^s_H(q)\in V_p', \qquad \text{for all} \quad q\in U_p \text{ and } s\in ({t_0}-\delta_p, {t_0}+\delta_p) $$
Since $V_p=\phi^{t_0}_H(U_p)$ is made of clean intersection points, for any (local) symplectic leaf $L$ we know that $L\cap V_p$ is a submanifold such that $T(L\cap V_p)=TC\cap TL$, and $L\cap V_p$ is then a coisotropic submanifold in the symplectic leaf by Proposition \ref{prop:cleancoisotropicinleaf}. The proof of \cite[Theorem 3]{HLS15} then shows that $H$ must be locally constant in $L\cap V_p$ for times in $({t_0}-\delta_p, {t_0}+\delta_p)$. Applying this to every point in $W_{t_0}$, (even if the $\delta_p$ can potentially go to zero as $p$ approaches points in the complement of $W_{t_0}$) we deduce that $H(x,{t_0})|_{\phi^{{t_0}}_H(W_{t_0})}$ is locally constant (in the $x$ variable) on any leaf of the symplectic foliation. 
But this implies that $H({t_0},x)|_{C\cap L}$ is constant on each connected component of $L\cap C$ for any symplectic leaf, since otherwise, as one can deduces from the splitting form of a singular foliation, there would be some other leaf such that $H({{t_0}},x)|_{\phi^t_H(W_t)\cap L}$ is not locally constant. 
Applying this for every time ${t_0}$, the conclusion follows. \\

\textbf{\Cref{item:H_function_time} implies \Cref{item:flow_preserves_C}.} By contradiction, let $U\subset C$ and $t_0\in (0,1]$ be an open set subset and a time such that $\phi^{t_0}_H(U)\subset  M\setminus C$. As argued in the proof of Theorem \ref{thm:clean}, we can assume without loss of generality that $U$ is contained in the set 
$$D_{2k}= \{x\in M \mid \dim{T\cF}(x)=2k\}$$ of leaves of $\cK$ of some given dimension $2k$, for some $k$. By Theorem \ref{thm:clean}, we can also assume that $U\subset \Cl(C,\cK)$. In addition, since $\phi^{t}_H$ is a Hameotopy, it follows from Corollary \ref{cor:Hameos_preserve_leaves} that each point $p\in U$ remains in the same symplectic leaf for all $t$. In particular, we must have $\phi^{t}_H(U)\subset D_{2k}$ for all $t$, and $2k>0$ since otherwise $\phi^{t}_H$ would be the identity along $U$. Consider the infimum time
$$t_U= \inf_{t}\{ \phi^{t}_H(U)\not \subset C\},$$
which must satisfy $t_U<t_0$. There must be a point $p\in \overline{U}$, which can be assumed to be in $D_{2k}$ without loss of generality, that satisfies that there exists arbitrarily small values $s$ such that $\phi^{t_U+s}_H(p) \not \in U$. Let $W$ be a Weinstein splitting chart of $\cF$ at the point $q=\phi^{t_U}(p)$, which must be of the form $\R^{2k}\times \R^{n-2k}$ since $q\in D_{2k}$. We can find an open neighborhood $V\subset U$ of $p$ with the following properties:
\begin{itemize}
\item[-] Thanks to the continuity of the isotopy, for every $x\in V$  the point $q_x= \phi^{t_x}_H(x)$ must be contained in $W$, where  
$$t_x= \inf_t \{ \phi^{t}_H(x)\not \in C\}.$$
\item[-] Up to shrinking $V$, the set
$$V'=\phi^{t_U-\delta}_H(V)$$
is in $W$, and so does $\phi^t_H(x)$ for each $x\in V$ and each $t\in [t_U-\delta, t_x]$,
\item[-] Thanks to Theorem \ref{thm:clean}, the set $V'$ is contained in $\Cl(C,\cF)$.
\end{itemize}
We now claim that for each point $x'=\phi^{t_U-\delta}(x) \in V'$, the path
\begin{align*}
\gamma_x: [t_U-\delta, t_x] &\longrightarrow W\\
s &\longmapsto \phi^{s}(x')
\end{align*}
must, up to considering possibly changing the endpoint to some $t'_x<t_x$, be a non-clean path in the characteristic leaf $K_{x'}$. Indeed, notice that by Proposition \ref{prop:cleancoisotropicinleaf} and \cite[Lemma 21]{HLS15}, any clean point in $C$ must stay for some short time in $C$ (and in fact, in its characteristic leaf) when applying a Hameotopy. Hence, $\gamma_x$ will remain in $K_{x'}$ until it either hits a non-clean point of $C$ or $q_x$. The latter, which satisfies that for arbitrarily small values of $s$ $\phi^{t_x+s}_{H}(q_x)\not \in C$ cannot be a clean intersection point, and hence must be a non-clean point of $C$. We have thus shown that $C\cap W$ is a submanifold in the local Weinstein splitting chart $W\subset \R^{2k} \times \R^{n-2k}$ (for which the leaf at the origin has dimension $2k$), for which there is an open set $V'\subset C\cap W\cap D_{2k}$ of points that all admit a non-clean path completely contained in $W$. Since non-clean paths must belong to a single leaf of $\cF$ by definition, this implies that every point $p\in V'$ must belong to a leaf of $\cF$ that does not intersect $V'$ cleanly. This contradicts Proposition \ref{prop:always_one_clean_leaf_regular} below, concluding the proof.
\end{proof}

We now establish the local statement, which is a variation of Proposition \ref{prop:cleanregular_local}.

\begin{proposition}\label{prop:always_one_clean_leaf_regular}
    Consider $\R^n=\R^k\times \R^{n-k}$ equipped with a singular foliation $\cF$ spanned by all the coordinate vector fields on $\R^k$ and by a singular foliation $\mathcal{T}$ on $\R^{n-k}$ such that the origin is a leaf of dimension $0$, with $k\geq 1$. Let $C$ be a submanifold. Then, there is a leaf $L$ of $\cF$ such that $C\cap L\neq \emptyset$ and the intersection is clean.
\end{proposition}
\begin{proof}
    Since $C$ is contained in $D_{k}$, it is enough to prove it in the case that all leaves are of the form $\R^{k}\times \{*\}$.
    Consider the functions 
    $$f_j= z_j|_{C}: C \longrightarrow \R.$$
    Starting with $f_1$, two things can happen. Either $f_1(C)$ contains some regular values, or $f_1$ is constant. In the first case, we consider the subset $U_1=f_1^{-1}(c_1)$ for a regular value $c_1$ of $f_1$. Otherwise, we consider $U_1=C$, and denote by $c_1$ the constant value of $f_1$. In both cases, we have $U_1=C\cap L_1$ and $TU_1=TC\cap TL_1$, where 
    $$L_1=\{(x_i,y_i,z_j)\mid z_1=c_1\}.$$
    
    We now consider the functions $$f_j^1=f_j|_{U_1}: U_1 \longrightarrow \R, \qquad \text{where} \qquad j=2,...,n-k.$$
    If $f_2^1(U_1)$ has some regular value $c_2$, we consider $U_2=(f_2^1)^{-1}(c_2)$, otherwise, we consider $U_2=U_1$ and denote by $c_2$ the constant value of $f_2^1$. The subset $U_2$ satisfies $U_2=C\cap L_2$ and $TU_2=TC\cap TL_2$ in any case, where 
    $$L_2=\{(x_i,y_i,z_j \mid z_1=c_1, z_2=c_2\}.$$
    Iterating this with $j=3,...,n-k$, we obtain a subset $U:=U_{n-k}\subset C$ satisfying:
    \begin{itemize}
        \item[-] $U=C\cap L$ and $TU=TC\cap TL$ where $L$ is the leaf $L=\{z_j=c_j\mid j=1,...,n-k\}$,
        \item[-] $TU=TC\cap TL$.
    \end{itemize}
    In other words, we have shown that there is at least one leaf $L$ of $\cF$ such that $C$ intersects $L$ non-trivially and cleanly. 
\end{proof}

We are now in position to introduce the following notion, which is well-defined thanks to \Cref{thm:C0_Hamiltonians_vanishing_C}.

\begin{definition}
    \label{def:C0-characteristic_foliation}
    Let $(M,\pi)$ be a Poisson manifold, and $C$ a (smooth) coisotropic submanifold.
    We call the \textbf{$C^0$-characteristic partition} the partition of $C$ given by the following equivalence relation: two points $p$ and $q$ in $C$ are identified if there are an integer $k>0$ and $C^0$-Hamiltonians $H_1,...,H_k$ depending only on time along $C$ such that $q=\phi^1_{H_k}\circ...\circ \phi^1_{H_1}(p).$
    The equivalence classes of this relation are called \textbf{leaves}\footnote{This nomenclature is simply in analogy with the case of the usual (i.e.\ smooth) characteristic foliation. We do not make any claim that these leaves determine in any way a $C^0$-foliation (intended as a $C^0$-subbundle of $TC$) on the coisotropic $C$.} of the $C^0$-characteristic partition.
\end{definition}
The first and most relevant property of the $C^0$-characteristic partition is the following, which follows directly from the definition and Theorem \ref{thm:C0_Hamiltonians_vanishing_C}.
\begin{corollary}\label{cor:C0-char_fol_stable_under_C0-Hamiltonians}
    The $C^0$-characteristic partition of a coisotropic $C$ is preserved by the flow of every $C^0$-Hamiltonian whose restriction to $C$ only depends on time.
    Moreover, the (smooth) characteristic foliation of a coisotropic $C$ is $C^0$-rigid if and only if the partition by leaves that it induces coincides with the $C^0$-characteristic partition of $C$.
\end{corollary}

The following remark and lemma clarify the relation with the smooth characteristic foliation.

\begin{remark}
    \label{rmk:C0-characteristic_leaves_almost_everywhere_same_as_smooth_ones_but_can_be_bigger}
    It follows from the proof of \Cref{thm:C0-rigidity_coisotropics} that at almost every point of $C$ the $C^0$-characteristic leaf and the smooth characteristic leaf coincide in a neighborhood of that point.
    This being said, as the example of $C=\{z=x^3\}$ in the Poisson manifold $(\R^3,\pi=\partial_x\wedge\partial_y)$ already mentioned several times shows, on a subset of the coisotropic with empty interior, the $C^0$-characteristic leaves can be \emph{strictly} larger than the smooth characteristic leaves.
\end{remark}

\begin{lemma}
    \label{lem:C0-characteristic_leaves_union_of_smooth_characteristic_leaves}
    Given a smooth coisotropic $C$ in a Poisson manifold $(M,\pi)$, each leaf of the $C^0$-characteristic partition is given by a union of leaves of the (smooth) characteristic foliation.
\end{lemma}

\begin{proof}
    Observe that the concatenation $H\# K$ of a $C^0$-Hamiltonian $H$ with any smooth Hamiltonian $K$ is a $C^0$-Hamiltonian, and that if both $H$ and $K$ vanish on $C$ then $H\# K$ does as well.
    This readily implies that if $p$ and $q$ are in the same leaf of the $C^0$-characteristic partition, then the same holds for $p$ and every $q'$ in the same smooth characteristic leaf as $q$.
\end{proof}

Lastly, we have the following important property that follows from Corollary \ref{cor:C0-char_fol_stable_under_C0-Hamiltonians}, which essentially tells us that it is a $C^0$-rigid object.
\begin{corollary}\label{cor:C0char_is_C0rigid}
    Let $\varphi:(M,\Pi)\longrightarrow (M,\Pi)$ be a Poisson homeomorphism mapping a coisotropic submanifold $C$ to a smooth (and hence coisotropic) submanifold $C'$. Then $\varphi$ maps each leaf of the $C^0$-characteristic partition to a leaf of the $C^0$-characteristic partition. 
\end{corollary}
\begin{proof}
    The proof follows from the following observation (see e.g. \cite{HLS15}): if $H$ is a $C^0$-Hamiltonian and $\varphi$ is a Poisson homeomorphism, then $H'=H\circ \varphi$ is also a $C^0$-Hamiltonian. Let $p,q$ be two points in the same $C^0$-characteristic leaf of $C$, and assume for simplicity that there is a single $C^0$-Hamiltonian $H$ vanishing on $C$ such that $\phi^1_H(p)=q$. Consider the Hamiltonian $H'=H\circ \varphi^{-1}$, which is a Hamiltonian vanishing along $C'$, and satisfies $\phi^t_{H'}=\varphi\circ \phi^t_{H}\circ \varphi^{-1}$. It satisfies 
    $$\phi^t_{H'}(\varphi(p))=\varphi(q).$$
    which shows that $\varphi(p)$ and $\varphi(q)$ belong to the same $C^0$-characteristic leaf. This shows that each $C^0$-characteristic leaf of $C$ is mapped to a $C^0$-characteristic leaf of $C'$. Using the same argument applied to $\varphi^{-1}$ yields the desired conclusion.
\end{proof}

Lastly, we formulate a natural question about $C^0$-characteristic leaves:
\begin{questionintro}
    \label{quest:C0-characteristic_leaves_topological_submanifolds}
    Are $C^0$-characteristic leaves topologically embedded submanifolds of $C$?
\end{questionintro}

%%%%%%%%%%%%%%%%%%%%%%%%%%%%%%%%%%%%%%%%%%%%%%%%%%%%%%%%%%%%%%%%%%%%%%%%%%%%%%%%%%%%%%%%%%%%%%%%%%%%%%%%%%%%%%%%%

\subsection{$C^0$-coisotropics}

Analogously to the definition given in \cite{HLS15} for symplectic manifolds, in view of \Cref{thm:C0-rigidity_coisotropics} the following definition is well-posed:
\begin{definition}
    \label{def:C0-coisotropic}
    Let $(M,\Pi)$ be a Poisson manifold, and $C\subset M$ a subset. We say that
    $C$ is a \textbf{$C^0$-coisotropic} if, for each point $p\in C$, there is an open neighborhood $U$ of $p$ in $M$, an open neighborhood $\Omega$ of the origin in $\R^n$, a Poisson structure $\Pi_\Omega$ on $\Omega$, and a Poisson homeomorphism
    \[
    \varphi\colon (\Omega,\Pi_\Omega)\rightarrow (U,\Pi\vert_U)
    \]
    such that $\varphi^{-1}(C\cap U)$ is a smooth coisotropic submanifold of $(\Omega,\Pi_\Omega)$.
\end{definition}

\begin{example}
    \label{exa:C0-coisotropic}
    Consider $\R^3$ with coordinates $(x,y,z)$, equipped with the regular Poisson structure $\Pi=\partial_x\wedge \partial_y$.
    Let also $\psi=(\psi_x,\psi_y)$ be a symplectic homeomorphism of $(\R^2,\omega=\d x \wedge \d y)$.
    Then, the Poisson homeomorphism
    \[
    \varphi\colon \R^3\to\R^3 \, ,
    \quad 
    (x,y,z)\to (\psi_x(x,y)+\vert z \vert, \psi_y(x,y) , z) \, ,
    \]
    is a global chart realizing $C:=\varphi(\{x=0\})$ as a $C^0$-coisotropic (non-smooth) subset of $(\R^3,\Pi)$.
\end{example}

There are two fundamental differences between $ C^0$-coisotropic in this setting and in the symplectic case.
First, there isn't a unique local model of smooth coisotropic submanifolds. 
(Even worse, there is no unique local model of Poisson structure, so one cannot hope to prescribe $\Pi_\Omega$ explicitly.) 

Second, as explained in \cite[Proposition 24]{HLS15}, in the symplectic case, one gets a well-defined notion of characteristic leaves of such a $C^0$-coisotropic, thanks to the fact that local changes of charts are symplectic homeomorphisms and hence send characteristic leaves to characteristic leaves by \cite[Theorem 1]{HLS15}.
In the Poisson case we have seen that this is not in general the case, so just looking at the images under the local charts $\varphi$ of the characteristic foliation of each $\varphi^{-1}(C)$ does not give even a well defined partition of $C$, as such images depend on the choice of the chart $\varphi$.

\begin{example}
    Let $\R^3$ with coordinates $(x,y,z)$, equipped with the Poisson structure $ \Pi=\partial_x\wedge \partial_y$, and $C=\{z=x\}$ a smooth coisotropic.
    Then, $\varphi_1=\mathrm{Id}$ as well as $\varphi_2:=(\phi^1)^{-1}$ where $\phi^1$ is the $C^0$-Hamiltonian flow constructed as in \Cref{ex:cleantononclean} are two possible local (in fact, global) charts satisfying \Cref{def:C0-coisotropic}.
    However, while $\varphi_1$ clearly pushes forward the characteristic foliation of $C$ to itself, $\varphi_2$ pushes forward the characteristic foliation of $C':=\{z=x^3\}$ to $C$, and hence defines some $0$-dimensional leaves as well. 
\end{example}

The advantage of the notion given in \Cref{def:C0-characteristic_foliation} is then clear in this context. By Corollary \ref{cor:C0char_is_C0rigid}, the partition on $C$ is a well-defined (i.e.\ independent of the choice of $\varphi$ as in \Cref{def:C0-coisotropic}) partition of $C$. This can then be called \textbf{$C^0$-characteristic partition} of the $C^0$-coisotropic $C$. A positive/negative answer to \Cref{quest:C0-characteristic_leaves_topological_submanifolds} would immediately imply a positive/negative answer to its analogue for $C^0$-characteristic partitions of $C^0$-coisotropics. 

\bigskip

A first corollary of Theorem \ref{thm:C0_Hamiltonians_vanishing_C} is the following.
\begin{corollary}\label{cor:Hamiltonians_C0_coisotropics}
    Let $C$ be a closed $C^0$-coisotropic of $(M,\Pi)$. Given a $C^0$-Hamiltonian such that $H|_{C\cap L}$ is locally a function of time for each symplectic leaf $L$, its flow leaves $C$ invariant.
\end{corollary}
\begin{proof}
    Given $p\in C$, choose a coisotropic chart as above, i.e., a neighborhood $U$ of $p$ and a Poisson homeomorphism 
    $$\varphi: (\Omega, \Pi_{\Omega}) \longrightarrow (U,\Pi|_U)$$
    such that $C'=\varphi^{-1}(C\cap U)$ is a smooth coisotropic submanifold of $(\Omega, \Pi_{\Omega})$. Consider the $C^0$-Hamiltonian $H'= H\circ\varphi$. By Theorem \ref{thm:poisson_homeos_are_leafwise_sympl_homeos}, we know that $H'|_{C'}$ restricts to each symplectic leaf as a function that is locally a function of time. By Theorem \ref{thm:C0_Hamiltonians_vanishing_C}, its flow $\phi_{H'}^t$ maps $\varphi^{-1}(p)$ to $C'$ (for short times). It follows that the continuous flow generated by $H$, which writes as
    $$\phi^t_H=\varphi \circ \phi^t_{H'}\circ \varphi^{-1},$$
    maps $p$ to $C$ for short times. This shows that the set of times for which $p$ is mapped to $C$ is open, and it is also closed because $C$ is closed. We conclude that the flow of $H$ preserves $C$.
\end{proof}

Lastly, as an unexpected consequence, which was pointed out to us by Du\v{s}an Joksimovi\'c, we deduce \Cref{cor:van_ideal} in the introduction. 

\begin{proof}[Proof of \Cref{cor:van_ideal}]
    Let $H,F$ be in $\mathcal{I}(C)$. 
    We need to prove that $\{H,F\}\in \mathcal{I}(C)$. 
    By Corollary \ref{cor:Hamiltonians_C0_coisotropics}, we know that 
    \begin{equation}\label{eq:preservation}
        \phi^t_H(C)=C, \qquad \text{and} \qquad \phi^t_G(C)=C,
    \end{equation}
    for all $t\in \R$. We now use the following identity (already ingeniously exploited in \cite{Jo}):
    $$ H\circ \phi^t_G -H=\int_0^t \{H,F\}\circ \phi^s_Gds.$$
    By Equation \ref{eq:preservation} and the fact that $H|_C=0$, we deduce that $H\circ \phi^t_G|_C=0$ for any $t\in \R$. It follows that $H\circ \phi^t_G-H|_C=0$, and hence its time derivative vanishes along $C$. On the other hand
    \begin{align*}
        \pp{}{t}\left(\int_0^t \{H,F\}\circ \phi^s_Gds \right)=\{H,F\}\circ \phi^t_G,
    \end{align*}
    and thus $(\{H,F\}\circ \phi^t_G)|_C=0$. Since $\phi^t_G(C)=C$, we conclude that $\{H,F\}\in \mathcal{I}(C)$.
\end{proof}

%%%%%%%%%%%%%%%%%%%%%%%%%%%%%%%%%%%%%%%%%%%%%%%%%%%%%%%%%%%%%%%%%%%%%%%%%%%%%%%%%%%%%%%%%%%%%%%%%%%%%%%%%%%%%%%%%
%%%%%%%%%%%%%%%%%%%%%%%%%%%%%%%%%%%%%%%%%%%%%%%%%%%%%%%%%%%%%%%%%%%%%%%%%%%%%%%%%%%%%%%%%%%%%%%%%%%%%%%%%%%%%%%%%

\section{Non-liftable Poisson homeomorphisms}\label{sec:non-liftable_homeos}

In this last section, we are going to study how Poisson homeomorphisms behave in terms of symplectic realizations.

%%%%%%%%%%%%%%%%%%%%%%%%%%%%%%%%%%%%%%%%%%%%%%%%%%%%%%%%%%%%%%%%%%%%%%%%%%%%%%%%%%%%%%%%%%%%%%%%%%%%%%%%%%%%%%%%%

\subsection{Liftable Poisson homeomorphisms}

The first natural property to look at in this context is whether Poisson homeomorphisms can be lifted to symplectic homeomorphisms of a given symplectic realization.
To capture the special case of symplectic integrations, we introduce the following notions:

\begin{definition}
    Let $(M,\Pi)$ be a Poisson manifold and $\varphi:M \longrightarrow M$ a Poisson homeomorphism.
    \begin{itemize}
        \item[-] The homeomorphism $\varphi$ is \textbf{liftable} to a symplectic realization 
        $$\mu:(W,\omega)\longrightarrow (M,\Pi)$$ if there is a symplectic homeomorphism $\tilde \varphi:W\longrightarrow W$ such that $\varphi\circ \mu =\mu\circ \tilde\varphi$.
        \item[-] The homeomorphism $\varphi$ is \textbf{groupoid-liftable} if $(M,\Pi)$ is integrable, and $\varphi$ is liftable to a symplectic homeomorphism of the canonical symplectic Lie groupoid integrating $(M,\Pi)$.
    \end{itemize}
\end{definition}

\begin{remark}
\label{rmk:liftable_potentially_not_compatible_groupoid}
Groupoid-liftable Poisson homeomorphisms do not in general admit lifts that are compatible with the groupoid structure, as in the definition we don't require any compatibility with the groupoid structure for the sequence of smooth symplectomorphisms of the groupoid approximating the lift of the homeomorphism. 
This being said, the natural class examples in \Cref{sec:lagrangian_C0-bisections} will have this additional nice property.
\end{remark}
\begin{remark}
\label{rmk:locally_liftable}
Many of the results and observations we will make for (groupoid-)liftable Poisson homeomorphisms, for instance \Cref{prop:full_coisotropic_rigidity_liftable}, actually hold for what one could call \textbf{locally (groupoid-)liftable} ones, namely Poisson homeomorphisms $\varphi\colon M \to M$ such that, for any $p\in M$, there is an open neighborhood $U\subset M$ such that $\varphi\vert_U$ is (groupoid-)liftable as a Poisson homeomorphism of $(U,\Pi\vert_U)$. 
\end{remark}

A good reason to focus on the groupoid-liftable setting is that any Poisson diffeomorphism of $(M,\Pi)$ integrates to a symplectic diffeomorphism (and in fact also a Lie groupoid morphism by \cite[Proposition 1.1]{coste_dazord_weinstein:groupoides_symplectiques}) of the canonical symplectic integration $\Sigma(M)$ (by Lie's second theorem for Lie groupoids, see e.g. \cite[Theorem A1]{MX}). 
The motivation behind considering liftable (or groupoid-liftable) Poisson homeomorphisms is that one should expect them to inherit more rigidity from their symplectic lift. 

For example, suppose a groupoid-liftable Poisson homeomorphism $\varphi$ is smooth. 
In that case, the $C^0$-rigidity of symplectic diffeomorphisms proved by Eliashberg--Gromov implies that its lift is in fact a \emph{symplectic} Lie groupoid (smooth) morphism. 
So the base map $\varphi$ is in fact a Poisson diffeomorphism.

\begin{example}
    \label{exa:symplectic_homeos_are_liftable_Poisson_homeos}
    Symplectic homeomorphisms of simply connected manifolds are groupoid-liftable Poisson homeomorphisms. 
    Indeed, the canonical symplectic groupoid $\Sigma(M)$ of a simply connected symplectic manifold $(M,\omega)$ is given (up to isomorphism) by the product $(M\times M,\omega\oplus -\omega)$, with target and source maps given respectively by the projection onto the first and second factors. If $\varphi\in\Homeo(M,\omega)$ then the induced $\widehat\varphi=(\varphi,\varphi)\colon M\times {M} \to M\times {M}$ is naturally a symplectic homeomorphism as well.
\end{example}

\subsubsection{Examples from loops of groupoid-liftable Poisson homeomorphisms.} Here we describe a source of examples of liftable Poisson homeomorphisms on Poisson manifolds of the form $(M\times S^1,\Pi)$ where $\Pi$ is a Poisson structure on the factor $M$.

\medskip

The input of the construction is a ``loop of groupoid-liftable Poisson homeomorphisms'' on $(M,\Pi)$, by which we don't mean a map from the circle to the subset of $\Homeo(M,\Pi)$ made of liftable Poisson homeomorphisms, but rather the following object. 
We call \textbf{loop of groupoid-liftable Poisson homeomorphisms} a continuous map
$\varphi\colon S^1\times M \to M$ such that:
\begin{itemize}
    \item[-] $\varphi_\theta:=\varphi(\theta,\cdot)\colon M \to M$ is a Poisson homeomorphism for every $\theta\in S^1$, and for $0\in S^1=\R/\mathbb{Z}$ we have $\varphi_0=\mathrm{Id}$;
    \item[-] there is a continuous map $\Phi\colon S^1\times \Sigma(M)\to \Sigma(M)$ such that, for all $\theta\in S^1$, $\Phi_\theta:=\Phi(\theta,\cdot)\colon \Sigma(M)\to \Sigma(M)$ is a lift of $\varphi_\theta$;
    \item[-] there is a sequence of smooth maps $\Psi_n\colon S^1\times \Sigma(M) \to \Sigma(M)$ such that, for all $\theta\in S^1$, $\Psi_{n,\theta}:=\Psi_n(\theta,\cdot)\colon \Sigma(M) \to \Sigma(M)$ is a symplectomorphism, that $C^0$-converges to $\Phi_\theta$ as $n\to \infty$.
\end{itemize}
Note that, by uniqueness of lifts for Poisson diffeomorphisms, $\Phi_0=\mathrm{Id}$. 
Similarly, up to changing the approximating family $\psi_n$ by precomposing with $\psi_{n,0}^{-1}$, we can arrange that $\psi_{n,0}=\mathrm{Id}$ for all $n$.

A loop of groupoid-liftable Poisson homeomorphisms will also be said to be \textbf{exact} if there are a $C^0$-Hamiltonian (time-dependent) function $H\colon S^1\times M \to M$, and a sequence $(\tilde H_n\colon S^1\times \Sigma(M) \to \Sigma(M))_n$ of smooth (time-dependent) Hamiltonian functions such that $\tilde H_n\xrightarrow{C^0} \tilde H$, where $\tilde H := H\circ s $ with $s\colon \Sigma(M)\to M$ being the source map, and so that the $\Psi_n$ above can be chosen to be generated by the time-dependent Hamiltonian vector fields $\tilde X_{n,\theta}$ on $\Sigma(M)$ associated to $\tilde H_{n,\theta}:=\tilde H_n(\theta,\cdot)$, i.e.\  $\frac{\partial \Psi_{n,\theta}}{\partial \theta}=\tilde X_{n,\theta}\circ\Psi_{n,\theta}$.

\smallskip

Natural examples of exact loops of groupoid-liftable Poisson homeomorphisms are the following.
\begin{enumerate}
    \item Exact loops of Poisson diffeomorphisms, i.e.\ loops of Poisson diffeomorphisms generated by a time-dependent vector field $X_\theta$ that can be written as $X_\theta=\iota_{\d H_\theta}\Pi$ for some family of functions $(H_\theta)_{\theta\in S^1}$.
    These naturally give an exact loop of groupoid-liftable Poisson homeomorphisms, as Poisson diffeomorphisms are liftable in a unique way to a symplectomorphism of $\Sigma(M)$, hence the sequence $\Psi_n$ can be taken to be constant in $n$ and $\Phi_\theta$ is generated by the Hamiltonian vector field associated to the (smooth) lift $\tilde H_\theta$ of $H_\theta$.

    In order to introduce some continuous and honestly-non-smooth behavior, in this case one can simply precompose with a \emph{homeomorphism} $f\colon S^1\to S^1$ all the involved $S^1$-families, i.e.\ consider $\Phi_{f(\theta)}$, $\Psi_{n,f(\theta)}$ and $H_{f(\theta)}$.

    \item If $\Pi$ is dual to a symplectic structure $\omega$ on $M$, and the latter is simply connected, then $\Sigma(M)=M\times M$ with symplectic form $\omega\oplus(-\omega)$.
    In this case, every homeomorphism is easily liftable, hence exact loops of groupoid-liftable Poisson homeomorphisms can be obtained for instance from loops $\varphi_\theta$ of symplectic homeomorphisms of $(M,\omega)$ such that there is a time-dependent $C^0$-Hamiltonian $H$ such that $\varphi_\theta$ is the associated ($S^1$-periodic) Hameotopy.
\end{enumerate}

\medskip

It is now easy to get a groupoid-liftable Poisson homeomorphism on $M \times S^1$ from an exact loop of groupoid-liftable Poisson homeomorphisms on $M$.
Indeed, consider the homeomorphism
    \begin{align*}
        \widehat\phi\colon M\times S^1 &\longrightarrow M\times S^1\\
        (p,\theta)&\longmapsto(\varphi_\theta(p),\theta)
    \end{align*}
    which admits as lift to the Poisson homotopy groupoid $\Sigma(M)\times T^*S^1$ of $M\times S^1$ the homeomorphism
    \begin{equation}
    \label{eqn:groupoid_lift}
    \begin{split}
        \widehat\Phi\colon \Sigma(M)\times T^*S^1 &\longrightarrow \Sigma(M)\times T^*S^1\\
        (x,t,\theta)&\longmapsto (\Phi_\theta(x),t-\widetilde H_\theta (x),\theta)
    \end{split}
    \end{equation} 
Using the notation from the definition of exact loop of groupoid-liftable Poisson homeomorphism, it is easy to check that the diffeomorphisms
\begin{align*}
        \widehat\Psi_n\colon \Sigma(M)\times T^*S^1 &\longrightarrow \Sigma(M)\times T^*S^1\\
        (x,t,\theta) &\longmapsto (\Psi_{n,\theta}(x),t-\widetilde H_{n,\theta} (x),\theta)
\end{align*}
are in fact symplectomorphisms and, for $n\to \infty$, they $C^0$-converges to $\widehat \Phi$. 
This concludes the construction.

\begin{remark}
    More generally, one can work with loops of Poisson homeomorphisms that are liftable to a concrete symplectic realization $(W,\omega)$, and obtain Poisson liftable homeomorphisms in $M\times S^1$ to the symplectic realization $W\times T^*S^1$ with the symplectic form $\omega\oplus \omega_{std}$. 
\end{remark}
\medskip

\subsubsection{Examples from Lagrangian $C^0$-bisections}
\label{sec:lagrangian_C0-bisections}

Let us first start by recalling how ``Lagrangian bisections'' give Poisson diffeomorphisms in the smooth case.

Let $(M,\Pi)$ be an integrable Poisson manifold, and $\Sigma(M)\rightrightarrows M$ its Poisson homotopy groupoid.
A \textbf{bisection} is a smooth map $\sigma\colon M\to \Sigma(M)$ such that $s\circ \sigma = \Id_M$ and $t \circ \sigma\colon M\to M$ is a diffeomorphism.
One moreover says that a bisection $\sigma$ is \textbf{Lagrangian} if $\sigma^*\Omega = 0$, i.e.\ if it is a Lagrangian embedding for the symplectic structure $\Omega$ on $\Sigma(M)$.

Lagrangian bisections naturally act on $\Sigma(M)$ in several ways: there are left and right translation actions, and one by conjugation combining both the previous ones.
We will be especially interested in the latter, so we now recall its definition and properties.

A bisection $\sigma \colon M\to \Sigma(M)$ acts via the \textbf{adjoint action}, i.e.\ by conjugation, on $\Sigma(M)$ in the following way: for any $g\in\Sigma(M)$, 
\begin{equation}
    \label{eqn:adjoint_action_bisection}
    \Ad_\sigma g := \sigma(t(g))\cdot g \cdot (\sigma(s(g))^{-1} \, ,
\end{equation}
where $\cdot$ at the right hand side denotes the groupoid multiplication.\footnote{The adjoint action is best interpreted by looking at elements $g$ of $\Sigma(M)$ as arrows from $s(g)$ to $t(g)$.
Indeed, $\sigma(t(g))$ is an arrow from $t(g)$, and $\sigma(s(g))$ an arrow from $s(g)$, so the composition in \eqref{eqn:adjoint_action_bisection} indeed makes sense.}
What's more, $\Ad_g\colon \Sigma(M)\to \Sigma(M)$ is in fact an automorphism of groupoids, lifting the map $t\circ \sigma \colon M\to M$, that is a diffeomorphism by definition of bisection.
Lastly, if the bisection $\sigma$ is Lagrangian, $\Ad_g\colon \Sigma(M)\to \Sigma(M)$ is a symplectomorphism w.r.t.\ $\Omega$.

\bigskip

In analogy with the smooth case just recalled, another natural source of examples of groupoid-liftable Poisson homeomorphisms is given by ``Lagrangian $C^0$-bisections'' via their natural action by conjugation.

Let us start from the following:
\begin{definition}
    \label{def:lagr_C0-bisections}
    A $C^0$-section $\sigma\colon M\to \Sigma(M)$ is a \textbf{$C^0$-bisection} if 
    \begin{itemize}
        \item[-]  $t \circ \sigma \colon M\to M$ is a homeomorphism, and
        \item[-] there is a sequence of smooth bisections $\tau_n\colon M\to \Sigma(M)$ which uniformly $C^0$-converges to $\sigma$.
    \end{itemize}
    We also say that a $C^0$-bisection $\sigma$ is \textbf{Lagrangian} if there is a sequence $\tau_n$ as above which are moreover Lagrangian as smooth bisections.
\end{definition}

It follows from \cite[Theorem 1]{HLS15} (or directly from the earlier work \cite[Theorem 2]{laudenbach_sikorav:C0_limits_lagrangians} in the case of closed source manifold and trivial relative $\pi_2$) that any Lagrangian $C^0$-bisection that is a (smooth) bisection is in fact automatically Lagrangian. 
In particular, the space of Lagrangian bisections is closed with respect to the $C^0$-norm in the space of bisections. \Cref{def:lagr_C0-bisections} above is then well-posed, in the sense that if $\sigma$ happens to be smooth, this just reduces to the usual notion of Lagrangian (smooth) bisection.

Moreover, if $\sigma$ is a Lagrangian $C^0$-bisection, $t\circ \sigma\colon M \to M$ is a Poisson homeomorphism, as the sequence $t\circ \tau_n$ of Poisson diffeomorphisms uniformly $C^0$-converges to it. 

\begin{example}
    Consider the case of a Poisson manifold given by a split product $M=X\times Y$ of a closed symplectic manifold $(X,\omega)$ and a smooth closed manifold $Y$, equipped with the natural Poisson structure induced by $\omega$, of rank equal to $\dim X$.
    Then, the Poisson homotopy groupoid is $\Sigma(M)=X\times \overline{X}\times T^*Y$, equipped with the natural symplectic structure given by $\omega\oplus (-\omega)\oplus \d\lambda_{\mathrm{std}}$. As explained in \cite[Proposition 26]{HLS15}, any $C^0$ 1-form $\alpha$ on $Y$ that is closed in the sense of distribution gives a $C^0$ Lagrangian in $T^*X$.
    Then, in our split product situation, for any choices of symplectic homeomorphism $\varphi\colon X\to X$ and $C^0$ 1-form $\alpha$ on $Y$ closed in the sense of distributions, the section $\sigma_{\varphi,\alpha}\colon X\times Y \to X\times \overline{X}\times T^*Y$ of the form $\sigma_{\varphi,\alpha}(x,y)=(\varphi(x),\varphi(x),\alpha(y))$ gives a $C^0$ Lagrangian bisection in the sense of \Cref{def:lagr_C0-bisections}.
    Note moreover that $t\circ \sigma_{\varphi,\alpha}=(\varphi,\Id)\colon X\times Y \longrightarrow X\times Y$.
\end{example}

The adjoint action can be defined in the $C^0$-setup analogously to the smooth case, as follows: 
\begin{definition}
    \label{def:adjoint_C0-action}
    Given $\sigma$ a $C^0$-bisection, the \emph{adjoint $C^0$-action} $\Ad_\sigma$ of $\sigma$ on $\Sigma(M)$ is defined by the same formula as in \eqref{eqn:adjoint_action_bisection}, i.e.: for any $g\in\Sigma(M)$, 
\begin{equation}
    \label{eqn:adjoint_action_C0-bisection}
    \Ad_\sigma g := \sigma(t(g))\cdot g \cdot (\sigma(s(g))^{-1} \, ,
\end{equation}
where $\cdot$ at the right hand side denotes the groupoid multiplication..
\end{definition}
Analogously to the smooth setup, $\Ad_\sigma$ is a ``$C^0$-automorphism of groupoids'', i.e.\ a homeomorphism commuting with the underlying groupoid structural maps.
What's more, it is not hard to see that if $\sigma$ is a Lagrangian $C^0$-bisection, $\Ad_\sigma$ and $t\circ \sigma$ are respectively a symplectic and a Poisson homeomorphism, and the first is the lift of the second;
in other words, $t\circ \sigma$ is a groupoid-liftable Poisson homeomorphism.

%%%%%%%%%%%%%%%%%%%%%%%%%%%%%%%%%%%%%%%%%%%%%%%%%%%%%%%%%%%%%%%%%%%%%%%%%%%%%%%%%%%%%%%%%%%%%%%%%%%%%%%%%%%%%%%%%

\subsection{Complete coisotropic rigidity}

Some additional rigidity is inherited from the symplectic lift of a liftable Poisson homeomorphism. Indeed, it turns out that Theorem \ref{thm:C0-rigidity_coisotropics} can be improved in the liftable case to obtain $C^0$-rigidity of all characteristic leaves.

\begin{proposition}\label{prop:full_coisotropic_rigidity_liftable}
Let $\varphi: (M,\Pi) \longrightarrow (M,\Pi)$ be a (local) Poisson homeomorphism and $\mu: (W,\omega) \longrightarrow (M,\Pi)$ a symplectic realization. 
Suppose that $C\subset M$ is a coisotropic submanifold and that $C'=\varphi(C)$ is smooth. 
If $\varphi$ lifts to a (local) symplectic homeomorphism $\tilde \varphi$ of $(W,\omega)$, then $C'$ is coisotropic and every leaf of the characteristic foliation of $C$ is mapped homeomorphically to a leaf of the characteristic foliation of $C'$.
\end{proposition}
\begin{proof}
By assumption, we have the following commutative diagram. 
\begin{center}
    \begin{tikzcd}
W \arrow[d, "\mu"] \arrow[r, "\tilde \varphi"] & W \arrow[d, "\mu"] \\
M \arrow[r, "\varphi"]                 & M                         
\end{tikzcd}
\end{center}
Recall that given a smooth submanifold $C\subset M$, we know that $C$ is coisotropic if and only if $\mu^{-1}(C)$ is coisotropic, see \cite[Proposition 8.61]{CFM_book}. We now consider
$$\tilde C:= \mu^{-1}(C), \qquad \text{and} \qquad \tilde C':=\mu^{-1}(C').$$
Let us first show that $\tilde \varphi(\tilde C)=\tilde C'$. Indeed, for the first inclusion, we have
\begin{align*}
    \mu (\tilde \varphi(\tilde C))= \varphi \circ \mu(\tilde C)=\varphi(C)=C',
\end{align*}
and thus $\tilde \varphi(\tilde C)\subseteq \tilde C'$. For the reverse inclusion, suppose that there is $x\in \tilde C' \setminus \tilde \varphi(\tilde C)$. Then, since $\tilde \varphi$ is a homeomorphism, there must be $y\in W\setminus \tilde C$ such that $\tilde \varphi(y)=x$, and thus
$\mu(\tilde \varphi(y))\in C'$. This implies that 
$$\varphi\circ \mu(y)\in C', \qquad \text{and thus} \qquad \mu(y)\in C,$$
which is a contradiction because this can only happen if $y\in \tilde C$. We conclude that $\tilde \varphi(\tilde C)=\tilde C'=\mu^{-1}(C')$, which in particular implies that $\tilde \varphi(\tilde C)$ is smooth. Hence, the map $\tilde \varphi$ is a symplectic homeomorphism mapping a coisotropic submanifold $\tilde C$ to a smooth submanifold $\tilde C'$. By \cite[Theorem 1]{HLS15}, the submanifold $\tilde C'$ is coisotropic, and the map $\tilde \varphi$ maps characteristic leaves to characteristic leaves homeomorphically.
Since $\mu^{-1}(C')$ is coisotropic, we know that $C'$ is coisotropic. Lastly, the proof of \cite[Proposition 8.61]{CFM_book} shows that at any point $x\in W$ we have
 $$ d_x\mu(\mathcal{K}_{\tilde C'})=\mathcal{K}_{C'},$$
 where $\mathcal{K}_{\tilde C'}$ and $\mathcal{K}_{C'}$ denote the characteristic distributions of $\tilde C'$ and $C'$ respectively. 
 Thus, the map $\mu$ maps every leaf of $\mu^{-1}(C)$ (respectively, of $\mu^{-1}(C')$) to a single leaf of $C$ (respectively, of $C'$). 
 Note that, in general, the preimage of a single leaf by $\mu$ will contain several leaves of the lifted coisotropic. 
 To be precise, given $w\in \tilde C$, $w'\in \tilde C'$, $x\in C$ and $x'\in C'$, we denote by $\tilde L_w, \tilde L'_{w'}, L_x, L'_{x'}$ the leaves, containing the respective points, of the characteristic foliation of $\tilde C, \tilde C', C, C'$ respectively. 
 Then, the map $\mu$ satisfies $$\mu(\tilde L_w) \subseteq L_x, \qquad  \text{and} \qquad \mu(\tilde L_{w'}') \subseteq L_{x'}',$$ 
 as long as $\mu(w)=x$ and $\mu(w')=x'$. 
 \medskip

For any point $p\in C$ and any lift $\tilde p\in \tilde C$ of $p$, using local Hamiltonians near $p$ vanishing at $C$ and their lifts $\mu\circ H$ (whose flow is defined for some short enough time), there is a neighborhood $U_p$ of $p$ in $L_p$ that admits a lift $\tilde U_p$ which is an open neighborhood of $\tilde p$ inside $\tilde L_{\tilde p}$. In addition, since $\tilde \varphi$ maps leaves to leaves homeomorphically, we know that 
$$\tilde\varphi(\tilde U_p)\subset \tilde L'_{\tilde \varphi(p)},$$
which leads us to 
$$\mu(\tilde\varphi(\tilde U_p)) \subset L'_{\varphi(p)}.$$
and then
$$\varphi(\mu(\tilde U_p)) \subset L'_{\varphi(p)}.$$
Using again that $\mu$ maps leaves to leaves, we conclude that
$$\varphi(U_p)\subset L'_{\varphi(p)}.$$
We can cover any path $\gamma:I\longrightarrow L_p$ starting at $p$ and ending at any point $q\in L_p$ by finitely many such open sets, and then the previous argument shows that $q\in L'_{\varphi(p)}$. This shows that $\varphi(L_p)\subset L_\varphi(p)'$. Lastly, note that $\tilde \varphi^{-1}$ is a symplectic homeomorphism lifting $\varphi^{-1}$, so the same argument applied to $\varphi^{-1}$ proves the other inclusion.
\end{proof}

%%%%%%%%%%%%%%%%%%%%%%%%%%%%%%%%%%%%%%%%%%%%%%%%%%%%%%%%%%%%%%%%%%%%%%%%%%%%%%%%%%%%%%%%%%%%%%%%%%%%%%%%%%%%%%%%%

\subsection{Compactly supported examples}\label{ss:nonlift}

In the previous section, we established that liftable Poisson homeomorphisms that map a coisotropic submanifold to a smooth submanifold must also map all characteristic leaves to characteristic leaves homeomorphically. This gives us an obstruction for a Poisson homeomorphism to be liftable, as we have seen that, in general, a Poisson homeomorphism might fail to map every characteristic leaf homeomorphically. In this section, we give constructions of non-liftable Poisson homeomorphisms that lead to a proof of Theorem \ref{thm:main3}.

%%%%%%%%%%%%%%%%%%%%%%%%%%%%%%%%%%%%%%%%%%%%%%%%%%%%%%%%%%%%%%%%%%%%%%%%%%%%%%%%%%%%%%%%%%%%%%%%%%%%%%%%%%%%%%%%%

\subsubsection{The zero Poisson structure}

We first work out the case of the zero Poisson structure, which, as noted in \cite{Jo}, is already interesting: can every homeomorphism of a manifold $M$ be lifted to a symplectic homeomorphism of its canonical symplectic integrating groupoid (e.g., its cotangent bundle if $M$ is simply connected)? We show that the answer is negative. 
Even worse, we prove that there always exist homeomorphisms that are not liftable for \emph{any} symplectic realization, not just for the canonical symplectic integration:

\begin{proposition}\label{prop:nonliftzeroPoisson}
    Let $M$ be a manifold of dimension at least one, which we endow with the trivial Poisson structure $\Pi=0$. There exists a Poisson homeomorphism\footnote{In the case of the trivial Poisson structure, a Poisson homeomorphism is simply a homeomorphism that is $C^0$-limit of diffeomorphisms.} $\varphi:M\longrightarrow M$, compactly supported near a point, which is not liftable for any symplectic realization of $M$.
\end{proposition}

\begin{proof}
Consider a chart $U\subset \R^n$ of $M$ with coordinates $(x_1,...,x_n)$, and assume without loss of generality that $I^n=[-1,1]^n \subset  U$. 
Consider a homeomorphism $f$ of the form
\begin{equation*}
    f: \mathbb{R}^n \longrightarrow \R^n \, , 
    \quad 
    x=(x_1,...,x_n) \longmapsto (f_1(x),...,f_n(x)),
\end{equation*}
where 
\begin{itemize}
    \item[-] $f_i(x)=\sqrt[3]{x_i}$ for $x\in I^n$,
   \item[-] $f_i(x)=x_i$ for $x$ away from an open neighborhood of $I^n$. 
\end{itemize}

We now want to prove that such an $f$ is not liftable.
Suppose by contradiction that it is liftable, i.e. that there is a symplectic realization $(W,\omega,\mu\colon W\to M)$ of dimension $2m$ and a symplectic homeomorphism 
$$g: W\longrightarrow W\qquad  \text{such that} \qquad \mu\circ g=f\circ \mu \, .$$
Take a point $q\in \mu^{-1}(p)$. 
Since $\mu$ is a submersion, there is a neighborhood $V$ of $q$ in $W$ with coordinates $(x_1,...,x_n,y_1,...,y_{2m-n})$ such that $\mu$ is the standard projection to the first $n$ coordinates. Using the standard Poisson bracket notation for the bracket induced by $\omega$, we have
$$\{x_i,x_j\}= \{x_1\circ \mu, x_j\circ \mu\}=\{x_i,x_j\}_{\Pi_0}\circ \mu=0,$$
i.e., these functions pairwise commute. 
By the Darboux-Carathéodory theorem (see \cite[17.2]{LM}, also known as the Carathéodory-Jacobi-Lie theorem), there exist coordinates $z_1,...,z_n$ and $u_1,... u_{m-n}, v_1,...,v_{m-n}$ such that in a neighborhood $V$ of $q$ such that
$$\omega|_V= \sum_{i=1}^n dx_i\wedge dz_i + \sum_{k=1}^{m-n} du_k \wedge dv_k.$$

Analogously, we find coordinates $(\hat x_i,\hat z_i,\hat u_k,\hat v_k)$ in a neighborhood $V'$ of $g(q)$ with the same properties as above. 
Notice that, because $\mu(g(q))=f(\mu(q))=f(p)=p$, where $p$ is the initial point we chose in $M$, we can assume without loss of generality that $\hat x_i=x_i$. 
\medskip 

Up to shrinking $V$, we can assume that it is a cube of the form $[-\varepsilon_0,\varepsilon_0]^{2m}$, and by continuity, we can assume that $g(V)\subset V'$, and then it induces a (symplectic) homeomorphism of the form
\begin{align*}
    g|_V: V &\longrightarrow g(V)\subset V' \\
    (x_i,z_i,u_k,v_k) &\longmapsto (\sqrt[3]{x_i},z_i', u_k',v_k'),
\end{align*}
where $z_i',u_k',v_k'$ are functions of $x_i,z_i,u_k,v_k$. Define the non-negative real number
    $$D:=d\Big(g(0), g\left((\{0\}\times \partial ([-\varepsilon_0,\varepsilon_0]^{2m-n})\right)\Big)$$
that is, the Euclidean distance between the image of the slice $\{0\}\times \partial [-\varepsilon_0, \varepsilon_0]^{2m-n}$ and the image of the origin (which is $g(q)$).
Note that $D>0$, because $g$ is a homeomorphism and hence injective.
We now aim to prove that in fact $D$ must on the other hand be zero if $g$ is $C^0$-limit of symplectic diffeomorphisms, thus reaching a contradiction.
For this, we argue as follows.

\smallskip

For any positive $\varepsilon\leq \varepsilon_0$ small enough, the rectangle 
$$C_\varepsilon:= [-\varepsilon,\varepsilon]^n\times [-\varepsilon_0,\varepsilon_0]^{2m-n}$$
is contained in $V$. 
Its image $G_\varepsilon:=g(C_{\varepsilon}) \subset g(V)$ satisfies
$$ \mu (G_\varepsilon)= [-\sqrt[3]{\varepsilon}, \sqrt[3]{\varepsilon}]^n.$$
By continuity of $g$, for $\varepsilon$ small enough, we can assume that
$$d\Big(g(x,0), g\left(\{x\}\times \partial([-\varepsilon_0,\varepsilon_0]^{2m-n})\right)\Big)>\frac{D}{2}, \quad \text{for all} \quad x\in[-\varepsilon,\varepsilon].$$
This implies that the symplectic volume (which is the Euclidean volume times $\frac{1}{m!}$) of $G_\varepsilon$ is at least that of 
$$[-\sqrt[3]{\varepsilon},\sqrt[3]{\varepsilon}]^n \times [0,D/2]^{2m-n},$$ namely
$$\operatorname{vol}(G_\varepsilon)\geq \frac{1}{m!} (2\sqrt[3]{\varepsilon})^n \left(\frac{D}{2}\right)^{2m-n}.$$
On the other hand, as a $C^0$-limit of measure-preserving homeomorphisms is measure-preserving, the map $g$ is volume-preserving with respect to the volume form $\omega^n$,
and thus
$$\operatorname{vol}(G_\varepsilon)=\operatorname{vol}(C_\varepsilon)=\frac{1}{m!}(2\varepsilon)^n (2\varepsilon_0)^{2m-n}.$$
Combining these two, we deduce
$$(2\varepsilon)^n (2\varepsilon_0)^{2m-n} \geq (2\sqrt[3]{\varepsilon})^n \left(\frac{D}{2}\right)^{2m-n}$$
from which it follows that 
$$ D^{2m-n} \leq \varepsilon^{\frac{2n}{3}} (4\varepsilon_0)^{2m-n}. $$
Since $\varepsilon$ is arbitrarily small, we deduce that $D=0$, as desired.
As already mentioned, this contradicts the fact that $g$ is a homeomorphism, thus concluding.
\end{proof}

\subsubsection{Regular Poisson manifolds}
We now give examples of compactly supported non-liftable Poisson homeomorphisms for regular Poisson structures, using the obstruction provided by Proposition \ref{prop:full_coisotropic_rigidity_liftable}.

In the following proposition, we exhibit a Hamiltonian example for which $C'=C$ and the same leaves of $\cK_C$ are not mapped homeomorphically to leaves, and make it compactly supported. 

\begin{proposition}
\label{prop:nonliftregular}
Consider $\mathbb{R}^{2n+k}$ with $n,k\geq 1$ endowed with the Poisson structure $\Pi=\sum_{i=1}^n \pp{}{x_i}\wedge \pp{}{y_i}$. There exists a compactly supported Hameomorphism 
$$\varphi:(M,\Pi)\longrightarrow (M,\Pi)$$
preserving a hypersurface $C$ that does not map some characteristic leaf of $\cK_C$ homeomorphically to some other leaf of $\cK_C$. In particular, it is not liftable for any symplectic realization. The corresponding continuous Hamiltonian $H$ generating $\varphi=\phi^1_H$ is autonomous and vanishes along $C$.
\end{proposition}
\begin{proof}
\textbf{Case $n,k=1$.}
    To better illustrate the ideas, we will first explain the case $n=k=1$. Consider a hypersurface of the form 
    $$C=\{(x,y,z)\mid z=f(x,y)\},$$
    where $f$ is chosen such that 
    \begin{itemize}
        \item[-] $f(x,y)=x^3$ for $y\in (-1,1)$,
        \item[-] $f(x,y)=x$ for $y\not \in (-1-\varepsilon,1+\varepsilon)$,
        \item[-] $f$ vanishes along the $y$ axis.
    \end{itemize}
One can moreover arrange that for each $(y,z)$ there is a unique $z$ such that $z=f(x,y)$. The hypersurface is sketched in Figure \ref{fig:interpolation}, the orange line is the set of clean intersection points with $\cF$, which corresponds with the set of zero-dimensional leaves of $\cK_C$ and with the set of tangency points between $C$ and the leaf $\{z=0\}$.

\begin{figure}[!ht]
    \centering
    \begin{tikzpicture}
        \node at (0,0) {\includegraphics[width=0.9\linewidth]{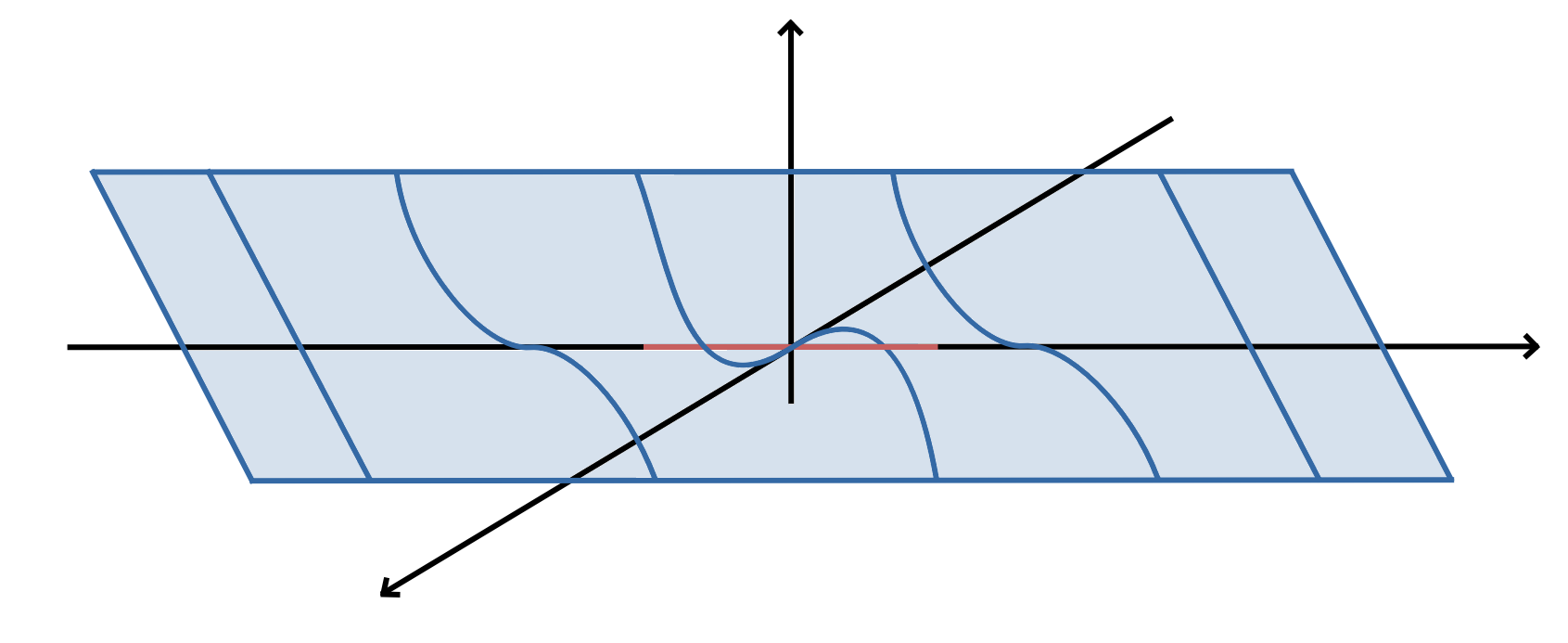}};
        \node at (6,-0.65) {$y$};
        \node at (-2.4,-2.1) {$x$};
        \node at (0.5, 2.3) {$z$};
    \end{tikzpicture}
    \caption{The hypersurface $C$, in blue, and its non-clean intersection points in orange}
    \label{fig:interpolation}
\end{figure}
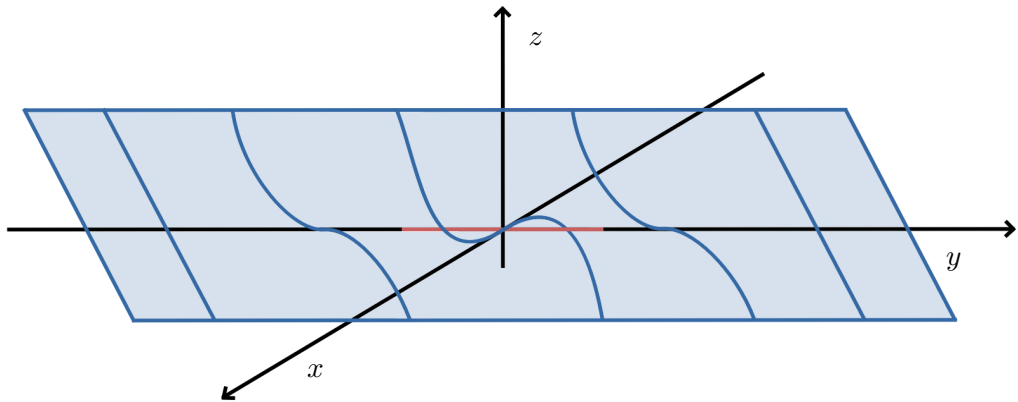
Consider then the (a priori not smooth) function $g=g(y,z)$ such that $x=g(y,z)$ if and only if $z=f(x,y)$. 
Note that $g$ is continuous but not smooth: indeed, for $y\in (-1,1)$ we have  $g(y,z)=\sqrt[3]{z}$.
    We then claim that the continuous function 
   $$H=-x+g(y,z)$$
    is a $C^0$-Hamiltonian. To prove this, consider a sequence $(H_i)_i$ of smooth Hamiltonians of the form
    $$H_i(x,y,z)=-x+g_i(y,z),$$ 
    where $g_i\colon \R^2_{(y,z)}\to \R$ is a sequence of smooth functions that $C^0$-converges to $g$, so that $H_i$ uniformly $C^0$-converge to $H$.
    If $f$ is chosen properly, namely smooth in $y$, we can moreover arrange that $\frac{\partial g_i}{\partial y}$ $C^0$-converges uniformly to $\frac{\partial g}{\partial y}$ (that is well defined by choice of $f$). 
    Then, the time-$t$ flows $\phi_{X_{H_i}}^t$, that are just the maps
    $$(x,y,z)\mapsto (x+\int_0^t \pp{g_i}{y}(y+t,z)dt \, ,\; y+t \, ,\; z)$$ 
    uniformly $C^0$-converge to a map $\phi^t$ given by 
    $$\phi^t(x,y,z)=(x+\int_0^t \pp{g}{y}(y+t,z)dt  \,, \; y+t \, ,\; z) \, .$$
    This shows that $H$ is a $C^0$-Hamiltonian.
    \noindent
    Now, it is immediate to check that, by definition of $g$, $H$ vanishes on $C$. 
    This being said, its flow $\phi^t$ does not preserve every leaf of the characteristic foliation of $C$, as for instance the image of the $1$-dimensional leaves $C\cap \{z=0,\,y<-1\}$ intersects the collection of $0$ dimensional leaves $C\cap \{z=0, y\in (-1,1)\}$ for $t>0$ large enough. Since $H$ is smooth away from a neighborhood of the origin, we can cut off $H$ to get a compactly supported function $\tilde H$ which is still a $C^0$-Hamiltonian, and consider $C$ restricted to the region where $\tilde H=H$, to obtain the conclusion of the lemma.
\medskip

\textbf{General case.}
    Consider a hypersurface (hence, a coisotropic submanifold) of the form
    $$C=\{(x_i,y_i, z_j) \mid z_1=f(x_i,y_i, z_2,...,z_n)\},$$ 
    where 
    $$f\colon \R^{2n+k} \longrightarrow  \R$$ 
    is a smooth function satisfying
    \begin{itemize}
        \item[-] $f(x_i,y_i,z_2,...,z_n)=x_1^3$ for $(x_i,y_i,z_2,...,z_k)\in B_1^{2n+k-1}(0)$, where the latter is the ball of radius $1$ centered at the origin,
        \item[-] $f(x_i,y_i,z_2,...,z_k)=x_1$ for $(x_i,y_i,z_2,...,z_k)\not \in B_{1+\varepsilon}^{2n+k-1}(0)$, for some fixed $\varepsilon>0$,
        \item[-] $f$ vanishes along the $y_1$ axis $\mathcal{O}_{y_1}$ (and thus $\pp{f}{y_1}=0$ there),
        \item[-] in a neighborhood of $\{y_1\in (-1-2\varepsilon, 1+2\varepsilon)\}\cap \mathcal{O}_{y_1}$, the function is of the form $f=f(x_1,y_1)$.
    \end{itemize}
    One can think of $f$ near $\{y_1\in (-1-2\varepsilon, 1+2\varepsilon)\}\cap \mathcal{O}_{y_1}$ as a family of functions parametrized by $y_1$, smoothly interpolating, as $y_1$ increasing, from $x_1$ to $x_1^3$, and then back to $x_1$.

    We can moreover arrange that for each $(x_2,...x_n,y_1,...,y_n,z_1,...,z_k)$ there is a unique $x_1$ such that $z_1=f(x_i,y_i,z_2,...,z_k)$.
    Consider then the a priori non-smooth function 
    $$g=g(y_1,x_2,y_2,...,x_n,y_n,z_1,...,z_k),$$ 
    such that $x_1=g(y_1,x_2,y_2,...,x_n,y_n,z_1,...,z_k)$ if and only if $z_1=f$. It satisfies that it vanishes along the $y_1$ axis, and $\pp{g}{y_1}(x_1,0,0)=0$. 
    Note that $g$ is only continuous: indeed, for radius larger than $1+\varepsilon$ (in all coordinates byt $z_1$), we have  $g=\sqrt[3]{z_1}$.
    
    Now, as in the case $n=k=1$, the function
    \begin{align*}
    H\colon \R^{2n+k} &\longrightarrow  \R\\
    (x_i,y_i,z_j)&\longmapsto x_1+g(y_1,x_2,y_2,...,x_n,y_n,z_1,...,z_k) 
    \end{align*}
    can be shown to be a $C^0$-Hamiltonian by considering smooth smooth Hamiltonians $H_i$ obtained from smooth approximations $g_i$ of $g$. If $f$ is chosen properly, namely differentiable in every coordinate different from $x_1$, we can moreover arrange that $g$ is differentiable in any coordinate except $z_1$, and $\frac{\partial g_i}{\partial \alpha}$ uniformly $C^0$-converges to $\frac{\partial g}{\partial \alpha}$ for any $\alpha\neq z_1$.
    Then, the time-$t$ flows $\phi_{X_{H_i}}^t$ (which only depend on integrating the derivative of $g_i$ with respect to the $x_j, y_j$ coordinates) will $C^0$-converge 
    uniformly to the flow $\phi^t$ of the ODE defined by 
    $$\begin{cases}
        \dot x_j=\pp{g}{y_j}, \quad \text{for}\quad  j=1,..,n,\\
        \dot y_1= 1,\\
        \dot y_j=-\pp{g}{x_j}, \quad \text{for} \quad j=2,...,n,\\
        \dot z_j=0, \quad \text{for}\quad j=1,...,k.
    \end{cases}$$
    This shows that $H$ is a $C^0$-Hamiltonian.
    \noindent
    Now, it is immediate to check that, by definition of $g$, $H$ vanishes on $C$.
    We claim that $\phi^t$ does not map homeomorphically the characteristic leaves of $C$. Indeed, notice that the $y_1$ axis contains both one-dimensional (for example, when $(x_j,y_j,z_2,...,z_k)$ is not in $B_{1+\varepsilon}^{2n+k-1}$) and zero-dimensional leaves (for example,  when $(x_j,y_j,z_2,...,z_k)$ lies in $B_{1+\varepsilon}^{2n+k-1}$). 
    A point of the form $p=(0,\hat y_1,0,...,0)$ with $\hat y_1\in (-1-2\varepsilon,-1-\varepsilon)$ lies in a one-dimensional characteristic leaf of $C$, but for an adequate $t$, we have $\phi^t(p)=(0,\hat y_1+t,0,...,0)$ with $|\hat y_1+t|<1$. In particular, it is a point that belongs to a zero-dimensional characteristic leaf of $C$. Since $H$ is smooth away from $B^{2n+k}_{1+2\varepsilon}$, we can cut it off to produce a compactly-supported function $\tilde H$ which is still a $C^0$-Hamiltonian. We can also consider the hypersurface $C$ restricted to the region where $\tilde H=H$, and then the flow produced by $\tilde H$ satisfies that for some $t$, it does not map all the characteristic leaves of $C$ to other leaves homeomorphically. 
\end{proof}

\begin{remark}\label{rmk:C0Hamiltonians_coisotropic_fails}
     Recall that \cite[Theorem 3]{HLS15} in the symplectic context states that a $C^0$-Hamiltonian restricts to a coisotropic as a function depending only on time if and only if the flow preserves each characteristic leaf of the coisotropic. Notice that in the previous construction, the $C^0$-Hamiltonian $H$ satisfies $H|_C=0$, which shows that the analog statement for Poisson manifolds is false. This is to be expected, as \cite[Theorem 3]{HLS15} can be understood as a generalization of the uniqueness of generators of Hameotopies, which is false in the Poisson context. 
\end{remark}

%%%%%%%%%%%%%%%%%%%%%%%%%%%%%%%%%%%%%%%%%%%%%%%%%%%%%%%%%%%%%%%%%%%%%%%%%%%%%%%%%%%%%%%%%%%%%%%%%%%%%%%%%%%%%%%%%

\subsection{Proof of Theorem \ref{thm:main3}}

The previous particular cases can now be combined to prove Theorem \ref{thm:main3}. Let $(M,\Pi)$ be a Poisson manifold that is not almost everywhere symplectic. Then there exists some open set $U$ where the rank of $\Pi$ is constant but different from the dimension of $M$.

If $\operatorname{Rank}(\Pi)|_{U}=0$, Proposition \ref{prop:nonliftzeroPoisson} provides a Poisson homeomorphism $\varphi: (U,\Pi=0) \longrightarrow (U,\Pi=0)$ that is not liftable for any symplectic realization. Extending $\varphi$ as the identity elsewhere in $M$, we obtain a non-liftable Poisson homeomorphism of $(M,\Pi)$.

Otherwise, the rank of $\Pi\vert_U$ is constant, different from zero, but has positive corank. 
Up to shrinking $U$, we can assume that we are in a splitting chart where $\Pi$ has the form of the standard regular Poisson structure on the corresponding Euclidean space. 
By \Cref{prop:nonliftregular}, there exists a coisotropic submanifold $C\subset U\subset M$ and a compactly supported Hameomorphism of $(U,\Pi)$ that maps $C$ to a coisotropic submanifold $C'$ in a way that characteristic leaves are not all mapped homeomorphically to characteristic leaves. This Hameomorphism extends as the identity to a Hameomorphism of $(M,\Pi)$. By Proposition \ref{prop:full_coisotropic_rigidity_liftable}, if $\varphi$ was liftable for some symplectic realization, it would map characteristic leaves of $C$ to characteristic leaves of $C'$ homeomorphically. This shows that $\varphi$ is not liftable and concludes the proof. \qed

\subsection{An example in the almost everywhere symplectic case}\label{sec:examplebPoisson}

The example we now describe is for a b-Poisson structure, and is a variation on the construction in the proof of \Cref{prop:nonliftregular}.
We explain it in ambient dimension $4$ for simplicity, but the argument can easily be generalized to higher ambient even dimensions.

\medskip

Consider in $\R^4$ the Poisson structure $\Pi=\pp{}{x}\wedge \pp{}{y} + u \pp{}{u}\wedge \pp{}{z}$. Mimicking the example in the regular case, we consider
$$C:=\{u=f(x,y)\},$$ where $f\colon \R^2_{x,y}\to \R$ is a smooth function, that is equal to $x$ for $y\leq -1$ and to $x^3$ for $y\geq 1$, and such that for each $(y,z)$ there is a 
unique $u$ such that $u=f(x,y)$. 
As in the previous examples, $C$ is automatically coisotropic, being a codimension $1$ submanifold.
Consider then the (non-smooth) function $g=g(y,u)$ such that $x=g(y,u)$ if and only if $u=f(x,y)$.
    We then claim that the continuous function 
    \[
    H\colon \R^4\to \R \, , 
    \quad 
    (x,y,z,u)\mapsto x+g(y,u) 
    \]
    is a $C^0$-Hamiltonian with respect to the Poisson structure $\Pi$.

    \noindent
    Indeed, consider the sequence of smooth Hamiltonians $H_n\colon \R^4\to \R$ defined by $H_n(x,y,z)=x+g_n(y,u)$, where $g_n\colon \R^2_{(y,u)}\to \R$ is defined as the convolution $g_n=\rho_{1/n}*g$, with $\rho_{1/n}:=\rho(n \, \cdot)/n^4$ where $\rho\colon \R^2\to \R$ is a mollifier. 
    In particular, each $g_n$ is smooth, and as a sequence $g_n$ converge in $C^0$-norm, uniformly on compact sets, to $g$; hence, $H_n$ converge in $C^0$-norm, uniformly on compact sets, to $H$.
    If $f$ is chosen properly, namely smooth in $y$, we can moreover arrange that $\frac{\partial g_n}{\partial y}$ in fact uniformly $C^0$-converges to $\frac{\partial g}{\partial y}$ (that is well defined by choice of $f$). 
    We also compute
    $$ \Pi(dH_n,\cdot)= \pp{}{y}-\pp{g_n}{y}\pp{}{x} + u\pp{g_n}{u}(y,u)\pp{}{z}$$
    Integrating, we find that the flow is given by 
    \[
    \phi_n^t:(x,y,z,u) \mapsto (x-\int_0^t \pp{g_n}{y}(y+s,u)\d s , \,  y+t, \, z+\int_{0}^t u\pp{g_n}{u}(y+s,u)\d s , \, u)
    \]
    where one can more explicitly write $\int_0^t \pp{g_n}{y}(y+s,u)\d s = \int_{t}^{t+s}\pp{g_n}{y}(y,u)\d y = g_n(y+t,u)-g_n(y)$.
    Note also that $u\pp{g_n}{u}(y,u) \to u\pp{g}{u}(y,u)$ in the $C^0$-norm uniformly on compact sets, and that $u\pp{g}{u}(y,u)$ is in fact a $C^0$ function.
    Hence, $\phi_n^t$ converges in $C^0$-norm uniformly on compact subsets to the continuous flow 
    \[
        \phi^t:(x,y,z,u) \mapsto (x-g(y+t,u)+g(y,u) , \,  y+t, \, z+\int_{0}^t u\pp{g}{u}(y+s,u)\d s , \, u)
    \]
    It is clear from the expression that $\phi^t$ preserves $C$ and that it sends each point $p=(0,y,0,0)\in C$ to $p'=(0,y+t,0,0)\in C$.
    Now, if $y\leq -1$ and $t> 2$, $p$ belongs to a $1$-dimensional leaf of the characteristic foliation of $C$, while $p'$ to a $0$-dimensional one.
    This concludes the example.

\begin{remark}
    It is very likely that examples of $C^0$-Hamiltonians that map at least one characteristic leaf of a certain coisotropic to a non-homeomorphic one should exist for any Poisson structure. 
    This being said, the specific formulas in the example above is highly dependent on the local formula for $\Pi$ near its singular locus; in other words, the expectation is that explicit examples in the almost everywhere symplectic case must be found in an ad hoc manner depending on the explicit type of the singular locus of the Poisson structure at hand. 
\end{remark}

\bibliographystyle{alpha}
\bibliography{biblio}

\Addresses

\end{document}